\documentclass[a4paper,11pt]{article}
\usepackage{fullpage}

\usepackage{preamble}

\newcommand{\footremember}[2]{%
    \footnote{#2}
    \newcounter{#1}
    \setcounter{#1}{\value{footnote}}%
}
\newcommand{\footrecall}[1]{%
    \footnotemark[\value{#1}]%
}

\title{Stochastic optimization with momentum:
convergence, fluctuations, and traps avoidance}%\thanksref{T1}}

\author{%
Anas Barakat \footremember{1}{LTCI, T\'el\'ecom Paris, IP Paris, France (firstname.name@telecom-paris.fr)}
\and Pascal Bianchi \footrecall{1}
\and Walid Hachem \footremember{2}{LIGM, CNRS, Univ Gustave Eiffel, ESIEE Paris, F-77454 Marne-la-Vallée, France (firstname.name@univ-eiffel.fr)}
\and Sholom Schechtman \footrecall{2}
}

\date{%
    \today
    }

\begin{document}

\maketitle

\begin{abstract}
In this paper, a general stochastic optimization procedure is studied, unifying
several variants of the stochastic gradient descent such as, among others, the
stochastic heavy ball method, the Stochastic Nesterov Accelerated Gradient
algorithm (S-NAG), and the widely used \textsc{Adam} algorithm. The algorithm
is seen as a noisy Euler discretization of a non-autonomous ordinary
differential equation, recently introduced by Belotto da Silva and Gazeau,
which is analyzed in depth.  Assuming that the objective function is non-convex
and differentiable, the stability and the almost sure convergence of the
iterates to the set of critical points are established.  A noteworthy special
case is the convergence proof of S-NAG in a non-convex setting.  Under some
assumptions, the convergence rate is provided under the form of a Central Limit
Theorem.  Finally, the non-convergence of the algorithm to undesired critical
points, such as local maxima or saddle points, is established.  Here, the main
ingredient is a new avoidance of traps result for non-autonomous settings,
which is of independent interest.
\end{abstract}

\tableofcontents

\section{Introduction}

Given a probability space $\Xi$, an integer $d > 0$, and a function $f : \RR^d
\times \Xi \to \RR$, consider the problem of finding a local minimum of the
function $F(x)\eqdef\EE_\xi[f(x,\xi)]$ w.r.t.~$x\in \RR^d$, where $\EE_\xi$
represents the expectation w.r.t.~the random variable $\xi$ on $\Xi$.
The paper focuses on the case where $F$ is possibly non-convex.
It is assumed that the function $F$ is unknown to the observer, either because the distribution of $\xi$
is unknown, or because the expectaction cannot be evaluated.
Instead, a sequence $(\xi_n:n\geq 1)$ of
i.i.d.~copies of the random variable $\xi$ is revealed online.

While the Stochastic Gradient Descent is the most classical algorithm that is
used to solve such a problem, recently, several other algorithms became very
popular.  These include the Stochastic Heavy Ball (\textsc{SHB}), the
stochastic version of Nesterov's Accelerated Gradient method (\textsc{S-NAG})
and the large class of the so-called \textit{adaptive} gradient algorithms,
among which \textsc{Adam} \cite{kingma2014adam} is perhaps the most used in
practice. As opposed to the vanilla Stochastic Gradient Descent, the study of
such algorithms is more elaborate, for three reasons. First, the update of the
iterates involves a so-called \emph{momentum} term, or inertia, which has the
effect of ``smoothing'' the increment between two consecutive iterates.
Second, the update equation at the time index $n$ is likely to depend on $n$,
making these systems inherently \emph{non-autonomous}.  Third, as far as
adaptive algorithms are concerned, the update also depends on some additional
variable (\emph{a.k.a.} the learning rate) computed online as a function of
the history of the computed gradients.

In this work, we
study in a unified way the asymptotic behavior of these algorithms in the
situation where $F$ is a differentiable function which
is not necessarily convex, and where the stepsize of the algorithm is decreasing.

Our starting point is a generic non-autonomous Ordinary Differential Equation
(ODE) introduced by Belotto da Silva and Gazeau \cite{das-gaz-20} (see also \cite{barakat-bianchi21} for \textsc{Adam}),
depicting the continuous-time versions of the aforementioned florilegium of algorithms.
The solutions to the ODE are shown to converge to the set of critical points of $F$.
This suggests that a general provably convergent algorithm can be obtained by
means of an Euler discretization of the ODE, including possible stochastic perturbations.
Special cases of our general algorithm include \textsc{SHB}, \textsc{Adam} and \textsc{S-NAG}.
We establish the almost sure boundedness and the convergence to critical points.
Under additional assumptions, we obtain convergence
rates, under the form of a central limit theorem.
These results are new. They extend the works of~\cite{gad-pan-saa18,barakat-bianchi21} to a general setting.
In particular, we highlight the almost sure convergence result of \textsc{S-NAG} in a non-convex setting,
which is new to the best of our knowledge.

Next, we address the question of the avoidance of ``traps''.
In a non-convex setting, the set of critical points of a function $F$
is generally larger than the set of local minimizers. A ``trap'' stands for a
critical point at which the Hessian matrix of
$F$ has negative eigenvalues, namely, it is a local maximum or saddle
point. We establish that the iterates cannot converge to such a point,
if the noise is exciting in some directions.
The result extends previous works of~\cite{gad-pan-saa18}
obtained in the context of \textsc{SHB}. This result not only allows to study
a broader class of algorithms but also significantly weakens the assumptions.
In particular, \cite{gad-pan-saa18} uses a sub-Gaussian assumption on the noise
and a rather stringent assumption on the stepsizes.
The main difficulty in the approach of~\cite{gad-pan-saa18} lies in the use of the
classical autonomous version of Poincar\'e's invariant manifold theorem.
The key ingredient of our proof is a general avoidance of traps result, adapted
to non-autonomous settings, which we believe to be of independent interest.
It extends usual avoidance of traps results to a non-autonomous setting, by making use of
a non-autonomous version of Poincar\'e's theorem \cite{dal-krei-(livre)74, klo-ras-(livre)11}. \\

\noindent\textbf{Paper organization.}
In Section~\ref{sec:odes}, we introduce and study the ODE's governing our general stochastic algorithm.
We establish the existence and uniqueness of the solutions, as well as the convergence to the set of critical points.
In Section~\ref{sec:as_convergence}, we introduce the main algorithm. We provide sufficient conditions under which
the iterates are bounded and converge to the set of critical points. A central limit theorem is stated.
Section~\ref{sec:avt} introduces a general avoidance of traps result for non-autonomous settings.
Next, this result is applied to the proposed algorithm. Sections~\ref{sec:proofs_odes}, \ref{sec:proof_as_convergence} and \ref{sec:proofs_avt} are devoted to the
proofs of the results of Sections~\ref{sec:odes}, \ref{sec:as_convergence} and \ref{sec:avt}, respectively.\\

\noindent\textbf{Notations.} Given an integer $d\geq 1$, two vectors $x, y \in
\RR^d$, and a real $\alpha$, we denote by $x \odot y$, $x^{\odot \alpha}$,
$x/y$, $|x|$, and $\sqrt{|x|}$ the vectors in $\RR^d$ whose $i$-th coordinates
are respectively given by $x_iy_i$, $x_i^{\alpha}$, $x_i/y_i$, $|x_i|$,
$\sqrt{|x_i|}$.  Inequalities of the form $x\leq~y$ are to be read
componentwise.  The standard Euclidean norm is denoted $\|\cdot\|$.  Notation
$M^\T$ represents the transpose of a matrix $M$.  For $x \in \RR^d$ and $\rho >
0$, the notation $B(x,\rho)$ stands for the open ball of $\RR^d$ with center
$x$ and radius~$\rho$.  We also write $\RR_+ = [0,\infty)$.  If  $z\in \RR^d$
and $A \subset \RR^d$, we write $\dist(z,A) \eqdef \inf\{ \|z-z'\| :z'\in A\}$.
By $\1_A(x)$, we refer to the function that is equal to one if $x\in A$ and to
zero elsewhere.  The set of zeros of a function $h : \RR^d \to \RR^{d'}$ is
$\zer h = \{ x \, : \, h(x) = 0 \}$. Let $D$ be a domain in $\RR^d$.  Given an
integer $k\geq 0$, the class $\mC^k(D, \RR)$ is the class of $D\to\RR$ maps
such that all their partial derivatives up to the order $k$ exist and are
continuous.  For a function $h \in \mC^k(D, \RR)$ and for every $i \in
\{1,\ldots,d\}$, we denote as $\partial_i^k h(x_1, \ldots, x_d)$ the
$k^{\text{th}}$ partial derivative of the function $h$ with respect to $x_i$.
When $k=1$, we just write $\partial_i h(x_1,\ldots,x_d)$.  The gradient of a
function $F: \RR^d \to \RR$ at a point $x \in \RR^d$ is denoted as $\nabla
F(x)$, and its Hessian matrix at $x$ is $\nabla^2 F(x)$ as usual.  For a
function $S: \RR^d \to \RR^d$, the notation $\nabla S(x)$ stands for the
jacobian matrix of $S$ at point $x$.

% =============================================
%   Main results
% =============================================

\section{Ordinary Differential Equations}
\label{sec:odes}

\subsection{A general ODE}
\label{subsec-ode-adam}

Our starting point will be a non-autonomous ODE which is almost identical to the one introduced in~\cite{das-gaz-20} and close to the one in~\cite{barakat-bianchi21}.
Let~$F$ be a
function in $\mC^1(\RR^d, \RR)$, let $S$ be a continuous $\RR^d \to \RR_+^d$
function, let $\sh,\sr,\spp, \sq  : (0,\infty) \to \RR_+$ be four continuous
functions, and let $\varepsilon > 0$.  Let $v_0 \in \RR_+^d$ and $x_0, m_0 \in
\RR^d$.  Starting at $\sv(0) = v_0$, $\sm(0) = m_0$, and $\sx(0) = x_0$, our
ODE on $\RR_+$ with trajectories in $\cZ_{+}\eqdef \RR_+^d \times
\bR^d\times\bR^d$ reads
\begin{equation}
\begin{cases}
  \dot \sv(t) &= \spp(t) S(\sx(t)) - \sq(t) \sv(t)  \\
  \dot \sm(t) &= \sh(t) \nabla F(\sx(t)) - \sr(t) \sm(t)   \\
  \dot \sx(t) &= - {\sm(t)} / {\sqrt{\sv(t) + \varepsilon}}
\end{cases}
\tag{ODE-$1$}\label{ode-generale}
\end{equation}
This ODE can be rewritten compactly in the following form.  Write $z_0 =
(v_0, m_0, x_0)$, and let $\sz(t) = ( \sv(t), \sm(t),
\sx(t)) \in \cZ_+$ for $t\in\RR_+$. Let $\cZ \eqdef \RR^d \times
\bR^d\times\bR^d$, and define the map $g : \cZ_+ \times (0,\infty)  \to \cZ$ as
\begin{equation}
\label{eq:gode}
g(z, t) = \begin{bmatrix}
   \spp(t) S(x) - \sq(t) v \\
   \sh(t) \nabla F(x) - \sr(t) m \\
   - {m} / {\sqrt{v+\varepsilon}} \end{bmatrix}
\end{equation}
for $z = (v, m, x) \in \cZ_+$. With these notations, we can
rewrite~\eqref{ode-generale} as
\[
\sz(0) = z_0, \quad \dot \sz(t) = g(\sz(t), t) \ \text{for} \ t > 0 .
\]
By setting $S(x) = \nabla F(x)^{\odot 2}$ when necessary and by properly
choosing the functions $\sh$, $\sr$, $\spp$, and $\sq$, a large number of
iterative algorithms used in Machine Learning can be obtained by an Euler's
discretization of this ODE. For instance, choosing $\sh(t) = \sr(t) =
a(t,\lambda, \alpha_1)$ and $\spp(t) = \sq(t) = a(t,\lambda, \alpha_2)$ with
$a(t,\lambda,\alpha) = \lambda^{-1}(1 - \exp(-\lambda\alpha)) / (1 -
\exp(-\alpha t))$ and $\lambda,\alpha_1,\alpha_2 > 0$, one obtains a version of
the \textsc{Adam} algorithm \cite{kingma2014adam} (see \cite[Sections~2.4-4.2]{das-gaz-20} for details).
 To give another less
specific example, if we set $\spp = \sq \equiv 0$, then the resulting ODE
covers a family of algorithms to which the well-known \textsc{Heavy Ball} with friction algorithm~\cite{attouch2000heavy} belongs.
%Alternatively, if we set
% $\sh = \sr \equiv 0$, then the resulting ODE covers the algorithms of the
% family of \textsc{RMSProp} \cite{tieleman2012lecture}.
For a comprehensive and
more precise view of the deterministic algorithms that can be deduced
from~\eqref{ode-generale} by an Euler's discretization, the reader is referred
to \cite[Table~1]{das-gaz-20}.

In this paper, since we are rather interested in stochastic versions of these
algorithms, Eq.~\eqref{ode-generale} will be the basic building block of
the classical ``ODE method'' which is widely used in the field of stochastic
approximation~\cite{ben-(cours)99}. In order to analyze the behavior of this
equation in preparation of the stochastic analysis, we need the following
assumptions.

\begin{assumption}
\label{hyp:F_loclip}
The function $F$ belongs to $\mC^1(\RR^d, \RR)$ and $\nabla F$ is
locally Lipschitz continuous.
\end{assumption}

\begin{assumption}
\label{hyp:F_coerc}
  $F$ is coercive, \emph{i.e.}, $F(x) \to +\infty$ as $\| x \| \to +\infty$.
\end{assumption}
Note that this assumption implies that the infimum $F_\star$ of $F$ is finite,
and the set $\zer\nabla F$ of zeros of $\nabla F$ is nonempty.

\begin{assumption}
\label{hyp:S_loclip}
% The function $F:\bR^d\to\bR$ is continuously differentiable, its gradient $\nabla F$ is
% locally Lipschitz continuous and
The map $S:\bR^d\to  \RR_+^d$ is locally Lipschitz continuous.
\end{assumption}

\begin{assumption}
\label{hyp:hrpq} The continuous functions $\sh,\sr,\spp, \sq  : (0,+\infty) \to \RR_+$ satisfy: %the following :

  \begin{enumerate}[{\sl i)}]
  \item\label{hyp:h}
    $\sh \in \mC^{1}((0,+\infty), \RR_+)$, $\dot\sh(t) \leq 0$ on $(0,+\infty)$ and
    the limit $h_\infty \eqdef \lim_{t\to\infty} \sh(t)$ is positive.

  \item \label{hyp:rq} $\sr$ and $\sq$ are non-increasing and
  $r_\infty \eqdef \lim_{t\to\infty} \sr(t)$\,, $q_\infty \eqdef \lim_{t\to\infty} \sq(t)$ are positive.

  \item\label{hyp:p}
  $\spp$ converges towards $p_\infty$ as $t\to\infty$.

  % \item\label{hyp:q} $\sq$ is non-increasing and
  % the limit $q_\infty \eqdef \lim_{t\to\infty} \sq(t)$ is positive.

  \item\label{hyp:stabode} For all $t \in (0,+\infty)$, $\sr(t) \geq \sq(t) / 4$ and $r_\infty > q_\infty / 4$.
  \end{enumerate}
\end{assumption}

These assumptions are sufficient to prove the existence and the uniqueness of
the solution to~\eqref{ode-generale} starting at a time $t_0 > 0$. The
following additional assumption extends the solution to $t_0 = 0$.

\begin{assumption}
\label{hyp:arzela-ascoli}
   Either $\sh, \sr, \spp, \sq \in \mathcal{C}^1([0, + \infty),\RR_+)$, or the following holds:
  \begin{enumerate}[\sl i)]
  \item\label{ass-S}
   For every $x \in \RR^d$, we have $S(x) \geq \nabla F(x)^{\odot 2}$.
    \item  The functions $\frac{\sh}{\spp}$, $\frac{\sh}{\sq - 2 \sr}$, $t
\mapsto t\sh(t)$, $t \mapsto t \sr(t)$, $t \mapsto t\spp(t)$, $t \mapsto
t\sq(t)$ are bounded near zero.
    \item There exists $t_0 > 0$ such that for all $t< t_0$,
    $%\begin{equation*}
      2 \sr(t) - \sq(t) > 0 \,.
    $%\end{equation*}
    \item There exists $\delta > 0$ such that $\frac{\sh}{\sr}\,, \frac{\spp}{\sq} \in \mathcal{C}^1([0, \delta),\RR_+)$\,.
    \item The initial condition $z_0= (v_0, m_0,x_0) \in\cZ_+$ satisfies
    \begin{equation*}
    m_0 = \nabla F(x_0) \lim_{t \downarrow 0} \frac{\sh(t)}{\sr(t)} \quad \text{and}\quad v_0 = S(x_0) \lim_{t \downarrow 0} \frac{\spp(t)}{\sq(t)}\,.
  \end{equation*}
  \end{enumerate}
\end{assumption}

\begin{remark}
The functions $\sh, \sr, \spp,\sq$ corresponding to \textsc{Adam} satisfy these
conditions. We leave the straightforward verifications to the reader. We
just observe here that the function $S$ that will correspond to our
stochastic algorithm in Section~\ref{sec:as_convergence} below will
satisfy Assumption~\ref{hyp:arzela-ascoli}--\ref{ass-S} by an immediate
application of Jensen's inequality.
\end{remark}

The following theorem slightly generalizes the results of
\cite[Th.~3 and Th.~5]{das-gaz-20}.

\begin{theorem}%[Existence and uniqueness]
%\label{th:existence_uniqueness}
\label{th:ode1}
Let Assumptions~\ref{hyp:F_loclip} to \ref{hyp:hrpq} hold true. Consider $z_0
\in \cZ_+$ and $t_0 > 0$.  Then, there exists a unique global solution $\sz :
[t_0, + \infty) \to \cZ_+$ to \eqref{ode-generale} with initial condition
$\sz(t_0) = z_0$. Moreover, $\sz([t_0, + \infty))$ is a bounded subset of~$\cZ_+$.
As $t\to +\infty$, $\sz(t)$ converges towards the set
  \begin{equation}
    \label{eq:Eq_infty}
   \Upsilon \eqdef
   \{ z_\star = ( p_\infty S(x_\star) / q_\infty, 0, x_\star)
   \, : \, x_\star \in \zer\nabla F \}\,.
  \end{equation}
If, additionally, Assumption~\ref{hyp:arzela-ascoli} holds, then we can take $t_0 =0$.
\end{theorem}

\begin{remark}
Th.~\ref{th:ode1} only shows the convergence of the trajectory $\sz(t)$ towards a set.
Convergence of the trajectory towards a single point is not guaranteed when the set $\Upsilon$
is not countable.
\end{remark}

\begin{remark}\label{rmk:simp_ode}
A simpler version of~\eqref{ode-generale} is obtained when omitting the momentum term. It reads:
\begin{equation}
\begin{cases}
  \dot \sv(t) &= \spp(t) S(\sx(t)) - \sq(t) \sv(t) \\
  \dot \sx(t) &= - {\nabla F(\sx(t))}/ {\sqrt{\sv(t) + \varepsilon}}\,.
\end{cases}
 \tag{ODE-$1'$} \label{ode-adagrad}
\end{equation}
This ODE encompasses the algorithms of the family of \textsc{RMSProp}
\cite{tieleman2012lecture}, as shown in \cite{barakat-bianchi21,das-gaz-20}.
The approach for proving the previous theorem can be adapted to~\eqref{ode-adagrad}
with only minor modifications. In the proofs below, we will point out the
particularities of~\eqref{ode-adagrad} when necessary.
% , since
% they violate Assumptions~\ref{hyp:hrpq} and \ref{hyp:arzela-ascoli}. However,
% by removing from these assumptions the statements that concern the
% functions $\sh$ and $\sr$ (respectively $\spp$ and $\sq$), and by doing
% straightforward adaptations to the proof of Th.~\ref{th:ode1}, the
% results of this theorem are recovered for these ODE's.
\end{remark}

The following paragraph is devoted to a particular case
of~\eqref{ode-generale}, which does not satisfy
Assumption~\ref{hyp:hrpq}, and which requires a more involved treatment
than~\eqref{ode-adagrad}.

\subsection{The Nesterov case}

The authors of \cite{cab-eng-gad09}, \cite{su_boyd_candes2016} and others
studied the ODE
\[
\ddot{\sx}(t) + \frac{\alpha}{t} \dot{\sx}(t) + \nabla F(\sx(t)) = 0, \quad
 \alpha > 0 , \quad F\in \mC^1(\RR^d, \RR),
\]
which Euler's discretization generates the well-known Nesterov's accelerated
gradient algorithm, see also  \cite{att-chb-pey-red18,auj-dos-ron18}. This ODE
can be rewritten as
\begin{equation}
\begin{cases}
    \dot \sm(t) &= \nabla F(\sx(t)) - \frac{\alpha}{t} \sm(t) \\
    \dot \sx(t) &= - \sm(t),
\end{cases}
\tag{ODE-N}\label{ode-true-nesterov}
\end{equation}
which is formally the particular case of~\eqref{ode-generale} that is taken for
$\spp(t) = \sq(t) = 0$, $\sh(t) = 1$, and $\sr(t) = \alpha / t$.  Obviously,
this case is not covered by Assumption~\ref{hyp:hrpq}.  Moreover, it turns out
that, contrary to the situation described in Remark~\ref{rmk:simp_ode} above,
this case cannot be dealt with by a straightforward adaptation of the proof of
Th.~\ref{th:ode1}. The reason for this is as follows. Heuristically, the
proof of Th.~\ref{th:ode1} is built around the fact that the solution
of~\eqref{ode-generale} ``shadows'' for large $t$ the solution of the
autonomous ODE
\begin{equation*}
  \begin{cases}
    \dot \sv(t) &= p_\infty S(\sx(t)) - q_\infty \sv(t) \\
    \dot \sm(t) &= h_\infty \nabla F(\sx(t)) - r_\infty \sm(t) \\
    \dot \sx(t) &= - \frac{\sm(t)}{\sqrt{\sv(t) + \varepsilon}} ,
  \end{cases}
\end{equation*}
and the latter can be shown to converge to the set $\Upsilon$ defined in
Eq.~\eqref{eq:Eq_infty}, either under Assumption~\ref{hyp:hrpq} or for the
algorithms covered by Remark~\ref{rmk:simp_ode}. This idea does not work anymore for~\eqref{ode-true-nesterov}, for its large--$t$ autonomous
counterpart
\begin{equation*}
  \begin{cases}
    \dot{\sm}(t)&= \nabla F(\sx(t)) \\
    \dot{\sx}(t)&= - \sm(t) .
  \end{cases}
\end{equation*}
can have solutions that do not converge to the critical points of $F$.
As an example of such solutions, take $d=1$ and $F(x) = {x^2}/{2}$. Then,
$t \mapsto (\cos(t), \sin(t))$ is an oscillating solution of the latter ODE.

Yet, we have the following result. Up to our knowledge, the proof of the convergence below as $t\to +\infty$ is new.
\begin{theorem}
\label{th:ode-nesterov}
Let Assumptions~\ref{hyp:F_loclip} and \ref{hyp:F_coerc} hold true. Then, for
each  $x_0 \in \RR^d$, there exists a unique bounded global solution $(\sm,
\sx) : \RR_+ \to \RR^d \times \RR^d$ to~\eqref{ode-true-nesterov} with the
initial condition $(\sm(0), \sx(0)) = (0, x_0)$.  As $t\to +\infty$, $(\sm(t),
\sx(t))$ converges towards the set
  \begin{equation}
    \label{eq:bar_Upsilon}
  \bar{\Upsilon} \eqdef \{ (0, x_\star) \, : \,
   x_\star\in \zer\nabla F \}.
  \end{equation}
\end{theorem}

\subsection{Related works}

The continuous-time dynamical system \eqref{ode-generale} we consider was first
introduced in \cite[Eq.~(2.1)]{das-gaz-20} with $S = \nabla F^{\odot 2}$.
Th.~\ref{th:ode1} above is roughly the same as \cite[Ths.~3
and~5]{das-gaz-20}, with some slight differences regarding the assumptions
on the function $F$, or Assumption~\ref{hyp:hrpq}-\ref{hyp:stabode}.
We point out that the main focus of \cite{das-gaz-20} is to study the
properties of the \textit{deterministic continous-time} dynamical system
\eqref{ode-generale}. In the present work, we highlight that the purpose of
Th.~\ref{th:ode1} is to pave the way to our analysis of the corresponding
\textit{stochastic algorithms} in Section~\ref{sec:as_convergence}.

Concerning Th.~\ref{th:ode-nesterov}, the existence and the uniqueness of a
global solution to \eqref{ode-true-nesterov} has been previously shown in the
literature, for instance in \cite[Prop.~2.1]{cab-eng-gad09} or in
\cite[Th.~1]{su_boyd_candes2016}.
The convergence statement in Th.~\ref{th:ode-nesterov} is new to the best of
our knowledge. In particular, we stress that we do not make any convexity
assumption on $F$. The closest result we are aware of is the one of
Cabot-Engler-Gadat \cite{cab-eng-gad09}. In \cite[Prop.~2.5]{cab-eng-gad09}, it
is shown that if $\sx(t)$ converges towards some point $\bar x$, then
necessarily $\bar x$ is a critical point of $F$.
Our result in Th.~\ref{th:ode-nesterov} strengthens this statement, by establishing that
 $\sx(t)$ actually converges to the set of critical points.

% Furthermore, assuming that~$F$ has finitely many critical points
% and attains different values on them, authors show that there exists exactly
% one critical point that is visited for arbitrarily long time intervals and that
% the trajectory converges to it if this point is a local minimum of $F$
% (\cite[Prop.~5.1]{cab-eng-gad09}). It is also proved that the density of the
% times when the trajectory is close to this critical point approaches one.

\section{Stochastic Algorithms}
\label{sec:as_convergence}

In this section, we discuss the asymptotic behavior of stochastic algorithms
that consist in noisy Euler's discretizations of~\eqref{ode-generale}
and~\eqref{ode-true-nesterov} studied in the previous section.

We first set the stage.  Let $(\Xi, \mcT, \mu)$ be a probability space.
Denoting as $\mcB(\RR^d)$ the Borel $\sigma$-algebra on $\RR^d$, consider a
$\mcB(\RR^d) \otimes \mcT$--measurable function $f:\bR^d\times \Xi\to \RR$ that
satisfies the following assumption.

\begin{assumption}
  \label{hyp:model} The following conditions hold:
  \begin{enumerate}[{\sl i)}]
  \item For every $x \in \RR^d$, $f(x,\cdot)$ is $\mu$--integrable.
  \item\label{nablaf-integ}
   For every $s \in \Xi$, the map $f(\cdot, s)$ is differentiable.
  Denoting as $\nabla f(x,s)$ its gradient w.r.t.~$x$, the function
  $\nabla f(x,\cdot)$ is integrable.
  \item\label{kappa}
   There exists a measurable map $\kappa : \RR^d \times \Xi \to \RR_+$
   s.t.~for every $x \in\RR^d$ :
   \begin{enumerate}[{\sl a)}]
   \item The map $\kappa(x,\cdot)$ is $\mu$--integrable,

   \item There exists $\varepsilon > 0$ s.t. for every $s \in \Xi$,
   \end{enumerate}
   \end{enumerate}
   $$
   \forall\, u, v \in B(x,\varepsilon), \
\| \nabla f(u,s) - \nabla f(v,s) \| \leq \kappa(x,s) \| u - v \|\,.
$$
\end{assumption}

Under Assumption~\ref{hyp:model}, we can define the mapping $F:\bR^d\to\bR$ as
\begin{equation}
\label{eq:F_S}
F(x) = \EE_\xi [ f(x,\xi)]
\end{equation}
for all $x \in \RR^d$, where we write
$\EE_\xi\varphi(\xi) = \int \varphi(\xi) \mu(d\xi)$. It is easy to see that
the mapping $F$ is differentiable,
\[
\nabla F(x) = \EE_\xi [\nabla f(x,\xi)]
\]
for all $x \in \RR^d$, and $\nabla F$ is locally Lipschitz.

Let $(\gamma_n)_{n \geq 1}$ be a sequence of positive real numbers satisfying
\begin{assumption}
  \label{hyp:stepsizes}
   $\gamma_{n+1}/\gamma_n\to 1$ and $\sum_{n} \gamma_n = +\infty$.
\end{assumption}

Define for every integer $n \geq 1$
\[
 \tau_n = \sum_{k=1}^n \gamma_k\,.
\]
Let $(\Omega, \mcF, \PP)$ be a probability space, and let
$(\xi_n:n\geq 1)$ be a sequence of iid random variables defined from
$(\Omega, \mcF, \PP)$ into $(\Xi, \mcT, \mu)$ with the distribution $\mu$.

\subsection{General algorithm}
\label{gal-adam-sto}

Our first algorithm is a discrete and noisy version of \eqref{ode-generale}.
Let $z_0 = (v_0, m_0, x_0) \in \cZ_+$ and
$h_0,r_0,p_0,q_0 \in (0,\infty)$. Define for every $n \geq 1$
\begin{equation}
  \label{eq:hrpq_seq}
 h_n = \sh(\tau_n), \ r_n = \sr(\tau_n), \ p_n = \spp(\tau_n), \ \text{and}
 \ q_n = \sq(\tau_n).
\end{equation}

The algorithm is written as follows.
\begin{algorithm}%[tb]
   \caption{(general algorithm)}
   \label{algosto}
\begin{algorithmic}
   \STATE {\bfseries Initialization:} $z_0 \in \cZ_+$.
   \FOR{$n=1$ {\bfseries to} $n_{\text{iter}}$}
   \STATE $v_{n+1} = ( 1 - \gamma_{n+1} q_n ) v_n
     + \gamma_{n+1} p_{n} \nabla f(x_n,\xi_{n+1})^{\odot 2}$
   \STATE $m_{n+1} = ( 1 - \gamma_{n+1} r_n) m_n
       + \gamma_{n+1} h_n \nabla f(x_n,\xi_{n+1})$
   \STATE $x_{n+1} = x_n - \gamma_{n+1} {m_{n+1}} / {\sqrt{v_{n+1} + \varepsilon}}$ \,.
   \ENDFOR
\end{algorithmic}
\end{algorithm}

We suppose throughout the paper that $1 - \gamma_{n+1}q_n \geq 0$ for all $n \in \NN$.
This will guarantee that the quantity $\sqrt{v_n + \varepsilon}$ is always well-defined (see~Algorithm~\ref{algosto}).
This mild assumption is satisfied as soon as $q_0 \leq \frac{1}{\gamma_1}$ since the sequence $(q_n)$ is non-increasing
and the sequence of stepsizes $(\gamma_n)$ can also be supposed to be non-increasing.

Since this algorithm makes use of the function $\nabla f(x,\xi)^{\odot 2}$\,,
a strengthening of Assumption~\ref{hyp:model} is required:
\begin{assumption}
  \label{hyp:model_bis}
 In Assumption~\ref{hyp:model}, Conditions \ref{nablaf-integ}
 and \ref{kappa} are respectively replaced with the stronger conditions
 \begin{itemize}
 \item[ii')] For each $x\in\RR^d$, the function
 $\nabla f(x,\cdot)^{\odot 2}$ is $\mu$~-integrable.
 \item[iii')] There exists a measurable map
    $\kappa : \RR^d \times \Xi \to \RR_+$ s.t.~for every $x \in\RR^d$:
   \begin{enumerate}[{\sl a)}]
   \item The map $\kappa(x,\cdot)$ is $\mu$--integrable.
   \item There exists $\varepsilon > 0$ s.t.
   \end{enumerate}
   \end{itemize}
 \[
   \forall\, u, v \in B(x,\varepsilon), \
\| \nabla f(u,s) - \nabla f(v,s) \| \vee
  \| \nabla f(u,s)^{\odot 2} - \nabla f(v,s)^{\odot 2} \|
      \leq \kappa(x,s) \| u - v \|.
  \]
\end{assumption}

Under Assumption~\ref{hyp:model_bis}, we can also define the mapping $S:\bR^d\to\bR^d$ as:
\[
S(x) = \EE_\xi [\nabla f(x,\xi)^{\odot 2}]
\]
for all $x \in \RR^d$. Notice that Assumptions~\ref{hyp:F_loclip}
 and~\ref{hyp:S_loclip} are satisfied for $F$ and $S$.

\begin{assumption}
  \label{hyp:moment}
  Assume either of the following conditions.
  \begin{enumerate}[i)]
  \item \label{hyp:moment-q}
  There exists $q \geq 2$ s.t. for every compact set $\mK \subset \RR^d$,
  \[
  \sup_{x\in \mK} \EE_\xi \| \nabla f(x,\xi) \|^{2q} < \infty
  \quad \text{and} \quad
  \sum_n \gamma_n^{1+q/2} < \infty\,.
  \]
  \item \label{hyp:subgauss-noise}
  For every compact set $\mK \subset \RR^d$, there exists a real $\sigma_\mK \neq 0$ s.t.
  \begin{align*}
  \EE_\xi \exp \ps{u, \nabla f(x,\xi) - \nabla F(x)} \1_{x\in\mK}
   &\leq \exp\left( \sigma_\mK^2 \|u\|^2 / 2\right)
   \  \text{and} \\
  \EE_\xi \exp \ps{u, \nabla f(x,\xi)^{\odot 2} - S(x)} \1_{x\in\mK}
   &\leq \exp\left( \sigma_\mK^2 \|u\|^2 / 2\right)\,,
  \end{align*}
  for every $x, u \in \RR^d$.
  Moreover, for every $\alpha > 0$,
  $%\[
  \sum_n \exp(-\alpha / \gamma_n) < \infty\,.
  $%\]
  \end{enumerate}
\end{assumption}

\begin{remark} We make the following comments regarding
 Assumption~\ref{hyp:moment}.
\begin{itemize}
  \item Assumption~\ref{hyp:moment}-\ref{hyp:moment-q} allows to use larger stepsizes in comparison to the classical condition~$\sum_n \gamma_n^2<\infty$ which corresponds to the particular case $q = 2$.

  \item Recall that a random vector $X$ is said to be subgaussian if there
exists a real $\sigma \neq 0$ s.t.  $ \EE e^{\ps{u,X}} \leq e^{\sigma^2 \| u
\|^2 / 2} $ for every constant vector $u \in \RR^d$.  In
Assumption~\ref{hyp:moment}-\ref{hyp:subgauss-noise}, the subgaussian noise
offers the possibility to use a sequence of stepsizes with an even slower decay
rate than in Assumption~\ref{hyp:moment}--\ref{hyp:moment-q}.
\end{itemize}
\end{remark}

  \begin{assumption}
    \label{hyp:sard}
    The set $F(\{x \, : \, \nabla F(x) = 0 \})$ has an empty interior.
  \end{assumption}

  \begin{remark}
  Assumption~\ref{hyp:sard} excludes a pathological behavior of the objective
  function $F$ at critical points.  It is satisfied when $F \in \mC^k(\RR^d,
  \RR)$ for $k \geq d$. Indeed, in this case, Sard's theorem stipulates that the
  Lebesgue measure of $F(\{x \, : \, \nabla F(x) = 0 \})$ is zero in $\RR$.
  \end{remark}

\begin{theorem}
\label{th:as_conv_under_stab}
Let Assumptions~\ref{hyp:F_coerc}, \ref{hyp:hrpq}, and
%\ref{hyp:model},
\ref{hyp:stepsizes}--\ref{hyp:sard} hold true.
% \ref{hyp:stepsizes}, \ref{hyp:model_bis} and \ref{hyp:moment} hold true.
Assume that the random sequence
$(z_n =(v_n,m_n,x_n):n\in \NN)$ given by Algorithm~\ref{algosto}
is bounded with probability one. Then, w.p.1, the sequence $(z_n)$ converges
towards the set $\Upsilon$ defined in Eq.~\eqref{eq:Eq_infty}.  If, in
addition, the set of critical points of the objective function $F$ is finite or
countable, then w.p.1, the sequence $(z_n)$ converges to a single point of
$\Upsilon$.
\end{theorem}

We now deal with the boundedness problem of the sequence $(z_n)$.
We introduce an additional assumption for this purpose.
\begin{assumption}
  \label{hyp:stability}
    The following conditions hold.
    \begin{enumerate}[{\sl i)}]
    \item $\nabla F$ is (globally) Lipschitz continuous.
  \item \label{momentgrowth} There exists $C>0$ s.t. for all $x \in \RR^d$,
    $\EE_\xi [\|\nabla f(x,\xi)\|^2] \leq C (1+F(x))$\,,
  \item $\sum_n \gamma_{n}^2 < \infty$\,.
  \end{enumerate}
  \end{assumption}

\begin{theorem}
\label{th:stab}
Let Assumptions~\ref{hyp:F_coerc}, \ref{hyp:hrpq}, \ref{hyp:stepsizes},
 \ref{hyp:model_bis}, \ref{hyp:moment}-\ref{hyp:moment-q} (with $q=2$)
and \ref{hyp:stability} hold.
Then, the sequence
$(v_n,m_n,x_n)$ given by Algorithm~\ref{algosto} is bounded with probability one.
\end{theorem}

\begin{remark}
\label{rmk:stab_larger_stepsizes}
The above stability result requires square summable step sizes. Showing the
same boundedness result under the Assumption~\ref{hyp:moment} that allows for
larger step sizes is a challenging problem in the general case. In these
situations, the boundedness of the iterates can be sometimes ensured by
\emph{ad hoc} means.
\end{remark}

\begin{remark}
\label{rmk:simp_ode_as_conv}
We can also consider the noisy discretization of~\eqref{ode-adagrad}
introduced in Remark~\ref{rmk:simp_ode} above. This algorithm reads
  \begin{subequations}
     \begin{numcases}{}
       v_{n+1} &$= ( 1 - \gamma_{n+1} q_n ) v_n
         + \gamma_{n+1} p_{n} \nabla f(x_n,\xi_{n+1})^{\odot 2}$\\
      x_{n+1} &$= x_n - \gamma_{n+1} {\nabla f(x_n,\xi_{n+1})}/ {\sqrt{v_{n+1} + \varepsilon}}$
       \end{numcases}
  \label{algo-adagrad}
  \end{subequations}

\noindent for $(v_0,x_0) \in \RR_+^d \times \RR^d$. With only minor
adaptations, Th.~\ref{th:as_conv_under_stab} and Th.~\ref{th:stab} can
be shown to hold as well for this algorithm.
We refer to the concomitant paper \cite[Sec.~2.2]{gadat-gavra20} for the link
between this algorithm and the seminal algorithms \textsc{AdaGrad} \cite{duchi2011adaptive}
and \textsc{RMSProp} \cite{tieleman2012lecture}.

\end{remark}

\subsection{Stochastic Nesterov's Accelerated Gradient (\textsc{S-NAG})}
% \subsection{Stochastic Nesterov's algorithm}
\label{nester-sto}
\textsc{S-NAG} is the noisy Euler's discretization of
\eqref{ode-true-nesterov}. Given $\alpha > 0$, it generates the sequence
$(m_n,x_n)$ on $\RR^d \times \RR^d$ given by Algorithm~\ref{algo-nesterov}.

\begin{algorithm}%[tb]
   % \caption{(Nesterov's stochastic algorithm with decreasing steps)}
   \caption{(S-NAG with decreasing steps)}
   \label{algo-nesterov}
\begin{algorithmic}
   \STATE {\bfseries Initialization:} $m_0 = 0, x_0 \in \RR^d$.
   \FOR{$n=1$ {\bfseries to} $n_{\text{iter}}$}
   \STATE $m_{n+1} = ( 1 - \alpha\gamma_{n+1} /\tau_n ) m_n
       + \gamma_{n+1} \nabla f(x_n,\xi_{n+1})$
   \STATE $x_{n+1} = x_n - \gamma_{n+1} {m_{n+1}}$ \,.
   \ENDFOR
\end{algorithmic}
\end{algorithm}

  \begin{assumption}
    \label{hyp:moment-nes}
    Assume either of the following conditions.
    \begin{enumerate}[i)]
    \item \label{hyp:moment-nes-q}
    There exists $q \geq 2$ s.t. for every compact set $\mK \subset \RR^d$,
    \[
    \sup_{x\in \mK} \EE_\xi \| \nabla f(x,\xi) \|^{q} < \infty
    \quad \text{and} \quad
    \sum_n \gamma_n^{1+q/2} < \infty\,.
    \]
    \item \label{hyp:subgauss-noise-nes}
    For every compact set $\mK \subset \RR^d$, there exists a real $\sigma_\mK \neq 0$ s.t.
    \begin{align*}
    \EE_\xi \exp \ps{u, \nabla f(x,\xi) - \nabla F(x)} \1_{x\in\mK}
     &\leq \exp\left( \sigma_\mK^2 \|u\|^2 / 2\right)\,,
    \end{align*}
    for every $x, u \in \RR^d$.
    Moreover, for every $\alpha > 0$,
    $%\[
    \sum_n \exp(-\alpha / \gamma_n) < \infty\,.
    $%\]
    \end{enumerate}
  \end{assumption}

  \begin{theorem}
  \label{th:as_conv_under_stab_nesterov}
  Let Assumptions~\ref{hyp:F_coerc}, \ref{hyp:model}, \ref{hyp:stepsizes}, \ref{hyp:sard} and~\ref{hyp:moment-nes} hold true. Assume that the random sequence
  $(y_n = (m_n,x_n):n\in \NN)$ given by Algorithm~\ref{algo-nesterov}
  is bounded with probability one. Then, w.p.1, the sequence $(y_n)$ converges
  towards the set~$\bar\Upsilon$ defined in Eq.~\eqref{eq:bar_Upsilon}.  If, in
  addition, the set of critical points of the objective function $F$ is finite or
  countable, then w.p.1, the sequence $(y_n)$ converges to a single point of
  $\bar\Upsilon$.
  \end{theorem}

 The almost sure boundedness of the sequence $(y_n)$ is handled
 in what follows.

\begin{theorem} %[Stability]
\label{th:stab_nesterov}
Let Assumptions~\ref{hyp:F_coerc}, \ref{hyp:model}, \ref{hyp:stepsizes} and \ref{hyp:stability} hold.
%and~\ref{hyp:moment-nes}-\ref{hyp:moment-nes-q} hold.
Then, the sequence $(y_n = (m_n,x_n):n\in \NN)$ given by Algorithm~\ref{algo-nesterov} is bounded with probability one.
\end{theorem}

\begin{remark}
Assumption~\ref{hyp:moment}-\ref{hyp:moment-q} in Th.~\ref{th:stab} is not needed for Th.~\ref{th:stab_nesterov}.
\end{remark}

\subsection{Central Limit Theorem}

In this section, we establish a conditional central limit theorem for Algorithm~\ref{algosto}.
\begin{assumption}
  \label{hyp:tcl_regularity_moments}
  Let $x_{\star}\in \zer \nabla F$. The following holds.
  \begin{enumerate}[i)]
    \setlength{\itemsep}{0pt}
  \item $F$ is twice continuously differentiable on a neighborhood of $x_{\star}$ and the Hessian $\nabla^2 F(x_{\star})$ is positive definite.
  \item $S$ is continuously differentiable on a neighborhood of $x_{\star}$.
\item There exists $M>0$ and  $b_M>4$ s.t.
\begin{equation}
\sup_{x\in B(x_{\star},M)}\EE_\xi[\|\nabla f(x,\xi)\|^{b_M}]<\infty\,.\label{eq:momentsTCL}
\end{equation}
% \item $-\Q$ is Hurwitz \emph{i.e.}, denoting by $\ell$ the smallest real part of the
% eigenvalues of $\Q$, it holds that $\ell>0$.
  \end{enumerate}
\end{assumption}

Under Assumptions~\ref{hyp:hrpq}-\ref{hyp:h}~to~\ref{hyp:p}, it follows from Eq.~\eqref{eq:hrpq_seq} that the sequences $(h_n), (r_n), (p_n)$ and $(q_n)$ of nonnegative reals converge respectively to $h_\infty, r_\infty, p_\infty$ and $q_\infty$ where $h_\infty$, $r_\infty$ and $q_\infty$ are supposed positive.
Define $v_{\star}\eqdef q_\infty^{-1}p_\infty S(x_{\star})$.
Consider the matrix
% $
% V \eqdef \textrm{diag}\left((\varepsilon + q_\infty^{-1}p_\infty S_1(x_{\star}))^{-\frac 12},\cdots,(\varepsilon + q_\infty^{-1}p_\infty S_d(x_{\star}))^{-\frac 12}\right)\,.
% $
\begin{equation}
  \label{eq:V_matrix}
V \eqdef \textrm{diag}\left((\varepsilon + v_\star)^{\odot -\frac 12}\right)\,.
\end{equation}
Let $P$ be an orthogonal matrix s.t. the following spectral decomposition holds:
$$
V^{\frac 12}\nabla^2F(x_{\star})V^{\frac 12} = P\textrm{diag}(\pi_1,\cdots,\pi_d)P^{-1}\,,
$$
where $\pi_1 \leq \cdots \leq \pi_d$ are the (positive) eigenvalues of
$V^{\frac 12} \nabla^2F(x_{\star})V^{\frac 12}$.
Define
$$
{\mathcal H} \eqdef
\begin{bmatrix}
  -r_\infty I_d & h_\infty \nabla^2 F(x_{\star}) \\
  - V & 0
\end{bmatrix}
$$
where $I_d$ is the $d \times d$ identity matrix. Then the matrix $\mathcal H$ is Hurwitz. Indeed, it can be shown that the largest real part of the eigenvalues of $\mathcal H$ coincides with $-L$, where
\begin{equation}
\label{eq:L}
L\eqdef \frac{r_\infty}{2}\left( 1-\sqrt{\left(1-\frac{4h_\infty\pi_1}{r_\infty^2}\right)\vee 0}\right) >0\,.
\end{equation}

\begin{assumption}
  \label{hyp:tcl_stepsizes}
  The sequence $(\gamma_n)$ is given by $\gamma_n = \frac{\gamma_0}{n^\alpha}$ for some $\alpha\in (0,1]$, $\gamma_0>0$.
Moreover, if $\alpha=1$, we assume that $\gamma_0> \frac{1}{2(L\wedge q_\infty)}$.
\end{assumption}

\begin{theorem}
  \label{th:clt}
  Let Assumptions~\ref{hyp:hrpq}-\ref{hyp:h} to \ref{hyp:p}, \ref{hyp:model_bis}, \ref{hyp:tcl_regularity_moments} and \ref{hyp:tcl_stepsizes} hold.
  Consider the iterates
  $z_n = (v_n, m_n, x_n)$ given by Algorithm~\ref{algosto}.
  %Define $v_{\star} = q_\infty^{-1}p_\infty S(x_{\star})$ %\pb{Déjà défini}.
   Set $\theta \eqdef 0$ if $\alpha<1$ and $\theta \eqdef 1/(2\gamma_0)$ if $\alpha=1$.
Assume that the event $\{z_n \to z_\star\}$, where $z_\star  = (v_{\star}, 0, x_{\star})$, has a positive probability. Then, given that event,
$$
\frac{1}{\sqrt{\gamma_n}}
\begin{bmatrix}
  m_n \\ x_n-x_{\star}
\end{bmatrix}
\Rightarrow {\mathcal N}(0,\Gamma)\,,
$$
where $\Rightarrow$ stands for the convergence in distribution and ${\mathcal N}(0,\Gamma)$ is a centered Gaussian distribution on $\bR^{2d}$ with a covariance matrix $\Gamma$ given by the unique solution to the Lyapunov equation
$$
({\mathcal H}+\theta I_{2d})\Gamma + \Gamma({\mathcal H}+\theta I_{2d})^\T = -
\begin{bmatrix}
  \mathrm{Cov}(h_\infty \nabla f(x_{\star},\xi)) & 0 \\ 0 & 0
\end{bmatrix}
\,.
$$
In particular, given $\{z_n\to z_{\star}\}$, the vector $\sqrt{\gamma_n}^{-1}(x_n-x_{\star})$
converges in distribution to a centered Gaussian distribution with a covariance matrix given~by:
\begin{equation}
\Gamma_2 = V^{\frac 12} P
\left[
\frac{C_{k,\ell}}
{ \frac{r_\infty -2\theta}{h_\infty} (\pi_k+\pi_\ell
+ \frac{2 \theta (\theta - r_\infty)}{h_\infty})
+\frac{(\pi_k-\pi_\ell)^2}{2(r_\infty - 2 \theta)}}
\right]_{k,\ell=1\dots d}
P^{-1}V^{\frac 12}\label{eq:cov}
\end{equation}
where $C\eqdef P^{-1}V^{\frac 12}\EE_\xi\left[\nabla f(x_{\star},\xi)\nabla f(x_{\star},\xi)^\T\right]V^{\frac 12}P$.
\end{theorem}

A few remarks are in order.
\begin{itemize}[leftmargin=*]

\item The matrix $\Gamma_2$ coincides with the limiting covariance matrix associated to the iterates
\begin{equation*}
     \begin{cases}
  m_{n+1} &= m_n + \gamma_{n+1}(h_\infty V\nabla f(x_n,\xi_{n+1})- r_\infty m_n)\\
  x_{n+1} &= x_n - \gamma_{n+1} m_{n+1}   \,.
     \end{cases}
\end{equation*}
 This procedure can be seen as a preconditioned version of the stochastic heavy ball algorithm~\cite{gad-pan-saa18} although the iterates are not implementable because of the unknown matrix $V$. Notice also that the limiting covariance $\Gamma_2$ depends on~$v_\star$ but does not depend on the fluctuations of the sequence $(v_n)$.

\item When $h_\infty = r_\infty$ (which is the case for \textsc{Adam}), we recover the expression of the asymptotic covariance matrix previously provided in \cite[Section~5.3]{barakat-bianchi21} and the remarks formulated therein.

\item The assumption $r_\infty > 0$ is crucial to establish Th.~\ref{th:clt}. For this reason, Th.~\ref{th:clt} does not generalize immediately to Algorithm~\ref{algo-nesterov}.
  The study of the fluctuations of Algorithm~\ref{algo-nesterov} is left for future works.
\end{itemize}

\subsection{Related works}
\label{related-1storder-clt}

In \cite{gad-pan-saa18}, Gadat, Panloup and Saadane study the \textsc{SHB}
algorithm, which is a noisy Euler's discretization of \eqref{ode-generale} in
the situation where $\sh = \sr$ and $\spp = \sq \equiv 0$ (\emph{i.e.}, there
is no $\sv$ variable).  In this framework, if we set $\sh = \sr \equiv r > 0$
in Algorithm~\ref{algosto} above, then Th.~\ref{th:as_conv_under_stab} above
recovers the analogous case in \cite[Th.~2.1]{gad-pan-saa18}, which is termed
as the exponential memory case.  The other important case treated
in~\cite{gad-pan-saa18} is the case where $\sh(t) = \sr(t) = r / t$ for some $r
> 0$, referred to as the polynomially memory case.  Actually, it is known that
the ODE obtained for $\sh(t) = \sr(t) = r / t$ and $\spp = \sq \equiv 0$ boils
down to~\eqref{ode-true-nesterov} after a time variable change (see,
\emph{e.g.}, Lem.~\ref{lm:ch_var} below).
Nevertheless, we highlight that the stochastic
algorithm that stems from this ODE and that is studied in~\cite{gad-pan-saa18}
is different from the \textsc{S-NAG} algorithm introduced above which stems
from a different ODE \eqref{ode-true-nesterov}.
Hence, the convergence result of Th.~\ref{th:as_conv_under_stab_nesterov}
for the \textsc{S-NAG} algorithm we consider is not
covered by the analysis of \cite{gad-pan-saa18}.

The specific case of the \textsc{Adam} algorithm is analyzed in
\cite{barakat-bianchi21} in both the constant and vanishing stepsize settings
(see \cite[Ths.~5.2-5.4]{barakat-bianchi21} which are the analogues of our
Ths.~\ref{th:as_conv_under_stab}-\ref{th:stab}).
Note that we deal with a more general algorithm in the present paper.
Indeed, Algorithm~1 offers some freedom
in the choice of the functions $h,r,p,q$ satisfying Assumption~2.4 beyond
the specific case of the \textsc{Adam} algorithm studied in \cite{barakat-bianchi21}.
Apart from this generalization, we also emphasize some small improvements.
Regarding Theorem~3.1, we provide noise conditions allowing to
choose larger stepsizes (see Assumption~3.4 compared to \cite[Assumption~4.2]{barakat-bianchi21}).
Concerning the stability result (Th.\ref{th:stab}), we relax \cite[Assumption 5.3-(iii)]{barakat-bianchi21}
which is no more needed in the present paper (see Assumption~\ref{hyp:stability}) thanks
to a modification of the discretized Lyapunov function used in the proof
(see Section~6.4 compared to \cite[Section~9.2]{barakat-bianchi21}).

In most generality, the almost sure convergence result of the iterates of
Algorithm~\ref{algosto} using vanishing stepsizes
(Ths.~\ref{th:as_conv_under_stab}-\ref{th:stab}) is new to the best of our knowledge.
Moreover, while some recent results exist for \textsc{S-NAG} in the constant
stepsize and for convex objective functions (see for e.g.  \cite{ass-rabb20}),
Ths.~\ref{th:as_conv_under_stab_nesterov} and~\ref{th:stab_nesterov} which
tackle the possibly non-convex setting are also new to the best of our
knowledge.

In the work \cite{gadat-gavra20} that was posted on the arXiv repository a few days
after our submission, Gadat and Gavra study the specific case of
the algorithm described in Eq.~\eqref{algo-adagrad} encompassing both \textsc{Adagrad} and \textsc{RMSProp},
with the possibility to use mini-batches. For this specific algorithm, the authors establish a similar almost sure convergence result to ours \cite[Th.~1]{gadat-gavra20} for decreasing stepsizes
and derive some quantitative results bounding in expectation the gradient of the objective function
along the iterations for constant stepsizes \cite[Th.~2]{gadat-gavra20}.
We highlight though that they do not consider the presence of momentum in the algorithm. Therefore,
their analysis does not cover neither Algorithm~\ref{algosto} nor Algorithm~\ref{algo-nesterov}.

In contrast to our analysis, some works in the literature explore the constant
stepsize regime for some stochastic momentum methods either for smooth
\cite{yan2018unified} or weakly convex objective functions
\cite{mai_johansson2020convergence}. Furthermore, concerning \textsc{Adam}-like
algorithms, several recent works control the minimum of the norms of the
gradients of the objective function evaluated at the iterates of the algorithm
over~$N$ iterations in expectation or with high probability
\cite{basu2018convergence,zhou2018convergence,chen2018closing,zou2019sufficient,chen2018convergence,zaheer2018adaptive,alacaoglu2020convergence,def-bot-bach-usu20,ala-mal-mert-cev20} and establish regret bounds in the convex setting \cite{ala-mal-mert-cev20}.

Similar central limit theorems to Th.~\ref{th:clt} are established in the cases of the stochastic heavy ball algorithm with exponential memory \cite[Th.~2.4]{gad-pan-saa18} and \textsc{Adam} \cite[Th.~5.7]{barakat-bianchi21}. In comparison to \cite{gad-pan-saa18}, we precise that our theorem recovers their result and provides a closed formula for the asymptotic covariance matrix~$\Gamma_2$. Our proof of Th.~\ref{th:clt} differs from the strategies adopted in \cite{gad-pan-saa18} and~\cite{barakat-bianchi21}.

\section{Avoidance of Traps}\label{sec:avt}

In Th.~\ref{th:as_conv_under_stab} and Th.~\ref{th:as_conv_under_stab_nesterov}
above, we established the almost sure convergence of the iterates $x_n$ towards
the set of critical points of the objective function $F$ for both
Algorithms~\ref{algosto} and~\ref{algo-nesterov}.  However, the landscape of
$F$ can contain what is known as ``traps'' for the algorithm, namely, critical
points where the Hessian matrix of~$F$ has negative eigenvalues, making these
critical points local maxima or saddle points. In this section, we show that
the convergence of the iterates to these traps does not take place if the noise
is exciting in some directions.

Starting with the contributions of Pemantle~\cite{pem-90} and Brandi\`ere and
Duflo~\cite{bra-duf-96}, the numerous so-called avoidance of traps results that
can be found in the literature deal with the case where the ODE that underlies
the stochastic algorithm is an autonomous ODE. Obviously, this is neither the
case of~\eqref{ode-generale}, nor of~\eqref{ode-true-nesterov}.  To deal with
this issue, we first state a general avoidance of traps result that extends
\cite{pem-90,bra-duf-96} to a non-autonomous setting, and that has an interest
of its own.  We then apply this result to Algorithms~\ref{algosto}
and~\ref{algo-nesterov}.

\subsection{A general avoidance-of-traps result in a non-autonomous setting}
\label{subsec:gal-avt}

The notations in this subsection and in
Sections~\ref{subsec:prel-avt}--\ref{proof:th-av-traps} are independent from
the rest of the paper.  We recall that for a function $h : \RR^d \rightarrow
\RR^{d'}$, we denote by $\partial_i^k h(x_1, \ldots, x_d)$ the $k^{\text{th}}$
partial derivative of the function $h$ with respect to $x_i$.

The setting of our problem is as follows. Given an integer $d > 0$
and a continuous function $b: \RR^d \times \RR_+ \to \RR^d$, we consider
a stochastic algorithm built around the non-autonomous ODE
$\dot\sz(t) = b(\sz(t), t)$. Let $z_\star \in \RR^d$, and assume that
on $\cV \times \RR_+$ where $\cV$ is a certain neighborhood of $z_\star$,
the function $b$ can be developed as
\begin{equation}
\label{b-dl-z*}
b(z,t) = D (z - z_\star) + e(z,t) ,
\end{equation}
where $e(z_\star,\cdot) \equiv 0$, and where the matrix $D \in \RR^{d\times d}$
is assumed to admit the following spectral factorization: Given $0\leq d^- < d$
and $0 < d^+ \leq d$ with $d^- + d^+ = d$, we can write
\begin{equation}
\label{jordan}
D = Q \Lambda Q^{-1} , \quad
\Lambda = \begin{bmatrix} \Lambda^- \\ & \Lambda^+ \end{bmatrix}\, ,
\end{equation}
where the Jordan blocks that constitute $\Lambda^- \in \RR^{d^-\times d^-}$
(respectively~$\Lambda^+ \in \RR^{d^+ \times d^+}$) are those that contain the
eigenvalues $\lambda_i$ of $D$ for which $\Re\lambda_i \leq 0$ (respectively
$\Re\lambda_i > 0$). Since $d^+ > 0$, the point $z_\star$ is an unstable
equilibrium point of the ODE $\dot\sz(t) = b(\sz(t), t)$, in the sense that
the ODE solution will only be able to converge to $z_\star$ along a specific
so-called invariant manifold which precise characterization will be given in
Section~\ref{subsec:prel-avt} below.

We now consider a stochastic algorithm that is built around this ODE. The
condition $d^+ > 0$ makes that $z_\star$ is a trap that the algorithm should
desirably avoid. The following theorem states that this will be the case if the
noise term of the algorithm is omnidirectional enough.  The idea is to show
that the case being, the algorithm trajectories will move away from the invariant
manifold mentioned above.

\begin{theorem}
\label{th:av-traps}
Given a sequence $(\gamma_n)$ of nonnegative deterministic stepsizes such that
$\sum_n \gamma_n = +\infty$, $\sum_n \gamma_n^2 < +\infty$, and a filtration
$(\mcF_n)$, consider the stochastic approximation algorithm in~$\RR^d$
\[
z_{n+1} = z_n + \gamma_{n+1} b(z_n, \tau_n) + \gamma_{n+1} \eta_{n+1}
 + \gamma_{n+1} \rho_{n+1}
\]
where $\tau_n = \sum_{k=1}^n \gamma_k$.  Assume that the sequences $(\eta_n)$
and $(\rho_n)$ are adapted to~$\mcF_n$, and that $z_0$ is $\mcF_0$--measurable.
Assume that there exists $z_\star \in \RR^d$ such that Eq.~\eqref{b-dl-z*}
holds true on $\mathcal V \times \RR_+$, where $\mathcal V$ is a
neighborhood of $z_\star$. Consider the spectral factorization~\eqref{jordan},
and assume that $d^+ > 0$.  Assume moreover that the function $e$
at the right hand side of Eq.~\eqref{b-dl-z*} satisfies the
conditions:
\begin{enumerate}[{\it i)}]
\item\label{e=0}  $e(z_\star,\cdot) \equiv 0$.
\item \label{e_reg} On $\mathcal V \times \RR_+$, the functions
$\partial_2^n\partial_1^k e(z,t)$ exist and are continuous
for $0\leq n< 2$ and $0\leq k+n \leq 2$.
\item \label{hyp:d1_eps} The following convergence holds :
\begin{equation}
\label{d1_eps}
\lim_{(z,t)\to(z_\star,\infty)} \| \partial_1 e(z,t) \| = 0\,.
\end{equation}
%holds true.
\item \label{dt_eps} There exist $t_0 > 0$ and a neighborhood $\cW \subset \RR^d$ of $z_\star$ s.t.
$$%\begin{equation*}
  \sup_{z \in \cW, t \geq t_0} \norm{ \partial_2 e(z, t)} <~+~\infty\,
  \quad
  \text{and} \quad \sup_{z \in \cW, t \geq t_0} \norm{ \partial_1^2 e(z, t)} <~+~\infty\,.
$$%\end{equation*}
\end{enumerate}

Moreover, suppose that :%on $\mathcal W \subset \RR^d$ :
\begin{enumerate}[{\it i)},resume]
\item\label{sum_rho_carre} $\sum_n \| \rho_{n+1} \|^2 \1_{z_n \in \mathcal W} < \infty$ almost surely.
\item\label{traps-noise-mom}
 $\limsup \EE [\| \eta_{n+1} \|^{4} \, | \, \mcF_n ]
     \1_{z_n \in \mathcal W}< \infty$, and
   $\EE [\eta_{n+1} \, | \, \mcF_n ] \1_{z_n \in \mathcal W} = 0$.
\item
\label{traps-noise}
Writing $\tilde\eta_n = Q^{-1} \eta_n = (\tilde\eta^-_n, \tilde\eta^+_n)$
with $\tilde\eta^\pm_n \in \RR^{d^\pm}$, for some $c^2 > 0$, it holds that
$$
\liminf \EE [\| \tilde\eta^+_{n+1} \|^{2} \, | \, \mcF_n ]
   \1_{z_n \in \mathcal W} \geq c^2 \1_{z_n \in \mathcal W}\,.
$$
  %for some $c^2 > 0$.
\end{enumerate}
Then, $\PP( [z_n \to z_\star] )= 0$.
\end{theorem}

\begin{remark}
Assumptions~\ref{e=0} to \ref{dt_eps} of Th.~\ref{th:av-traps} are related to
the function $e$ defined in Eq.~\eqref{b-dl-z*}, which can be seen as a
non-autonomous perturbation of the autonomous linear ODE $\dot{\sz}(t) =
D(\sz(t)-z_\star)$.  These assumptions guarantee the existence of a local
(around the unstable equilibrium $z_\star$) non-autonomous invariant manifold
of the non-autonomous ODE $\dot{\sz}(t) = b(\sz(t),t)$ with enough regularity
properties, as provided by Prop.~\ref{prop:inv-man}
and Prop.~\ref{prop:inv-man-local} below.
\end{remark}

\subsection{Application to the stochastic algorithms}
\label{subsec:appli-avt}

\subsubsection{Trap avoidance of the general algorithm~\ref{algosto}}
\label{subsec-apt-algo-gal}

In Th.~\ref{th:as_conv_under_stab} above, we showed that the sequence $(z_n)$
generated by Algorithm~\ref{algosto} converges almost surely towards the set
$\Upsilon$ defined in Eq.~\eqref{eq:Eq_infty}. Our purpose now is to show that
the traps in~$\Upsilon$ (to be characterized below) are avoided by the
stochastic algorithm~\ref{algosto} under a proper omnidirectionality assumption
on the noise.

Our first task is to write Algorithm~\ref{algosto} in a manner compatible
with the statement of Th.~\ref{th:av-traps}.  The following decomposition holds
for the sequence $(z_n = ( v_n, m_n, x_n), n \in \NN)$ generated by this
algorithm:
\[
z_{n+1} = z_n + \gamma_{n+1} g(z_n, \tau_n) + \gamma_{n+1} \eta_{n+1} + \gamma_{n+1}\tilde{\rho}_{n+1},
\]
where
$\tilde{\rho}_{n+1} =
\left(
0\,, 0\,, \frac{m_{n}}{\sqrt{v_{n} + \varepsilon}} - \frac{m_{n+1}}{\sqrt{v_{n+1} + \varepsilon}}
\right)$,
and where $\eta_{n+1}$ is the martingale increment with respect to the
filtration $(\mcF_n)$ which is defined by Eq.~\eqref{eq:noise_eta}.

Observe from Eq.~\eqref{eq:Eq_infty} that each $z_\star \in \Upsilon$ is
written as $z_\star = (v_\star, 0, x_\star)$ where $x_\star \in \zer \nabla F$,
and $v_\star = q_\infty^{-1}p_\infty S(x_\star)$ (in particular, $x_\star$ and
$z_\star$ are in a one-to-one correspondence). We need to linearize the
function $g(\cdot,t)$ around $z_\star$. The following assumptions will be
required.

\begin{assumption}
  \label{hyp:F_C4_S_C3}
  The functions $F$ and $S$ belong respectively to $\mC^3(\RR^d, \RR)$ and~$\mC^2(\RR^d, \RR_+^d)$.
\end{assumption}

\begin{assumption}
  \label{hyp:hrpq_C1_deriv_bounded}
  The functions $\sh,\sr,\spp, \sq$ belong to
$\mC^1((0,\infty),\RR_+)$ and have bounded derivatives on $[t_0,+\infty)$ for some $t_0 > 0$.
\end{assumption}

\begin{lemma}
\label{lem:lm-lin-g}
Let Assumptions~\ref{hyp:hrpq}-\ref{hyp:h} to~\ref{hyp:p}, \ref{hyp:F_C4_S_C3}
and~\ref{hyp:hrpq_C1_deriv_bounded} hold. Let $z_\star = (v_\star, 0, x_\star)
\in \Upsilon$. Then, for every $z \in \cZ_+$ and every $t > 0$, the following
decomposition holds true:
\[
g(z, t) = D (z-z_\star) + e(z,t) + c(t) ,
\]
\[
\text{where}
\,\,
D %= \partial g_\infty(z_\star)
  %  g_\infty is not yet defined at this stage
  = \begin{bmatrix}
 - q_\infty I_d & 0 & p_\infty \nabla S(x_\star) \\
 0 & - r_\infty I_d & h_\infty \nabla^2 F(x_\star) \\
 0 & - V & 0
 %\diag((v_\star + \varepsilon)^{-\frac 12})
\end{bmatrix},
\quad
c(t) = \begin{bmatrix}
   \spp(t) S(x_\star) - \sq(t) v_\star
     \\ 0 \\ 0 \end{bmatrix}\,,
\]
and the function $e(z,t)$ (defined in Section~\ref{prf:lm-lin-g} below for
conciseness) has the same properties as its analogue in the statement of
Th.~\ref{th:av-traps}.
\end{lemma}

Using this lemma, the algorithm iterate $z_{n+1}$ can be rewritten as an
instance of the algorithm in the statement of Th.~\ref{th:av-traps}, namely,
\begin{equation}
\label{eq:algo-gal-avt}
z_{n+1} = z_n + \gamma_{n+1} b(z_n, \tau_n) +  \gamma_{n+1} \eta_{n+1}
 + \gamma_{n+1} \rho_{n+1} ,
\end{equation}
where in our present setting, $b(z,t) = g(z,t) - c(t) = D (z-z_\star) + e(z,t)$
and $\rho_n = c(\tau_{n-1}) + \tilde{\rho}_n$. In the following assumption,
we use the well-known fact that a symmetric matrix $H$ has the same inertia
as $A H A^\T$ for an arbitrary invertible matrix $A$.
\begin{assumption}
\label{hyp:avt}
Let $x_\star \in \zer \nabla F$, let $v_\star = q_\infty^{-1}p_\infty
S(x_\star)$, and define the diagonal matrix $V= \diag( (v_\star +
\varepsilon)^{\odot -\frac 12})$ as in~\eqref{eq:V_matrix}. Assume the
following conditions:
\begin{enumerate}[{\it i)}]
  \item\label{rho_sum_carre_appli} $\sum_n \left( q_\infty p_n - p_\infty q_n \right)^2 < \infty\,,$

  \item\label{trap} The Hessian matrix $\nabla^2 F(x_\star)$ has a negative eigenvalue.
  \item\label{moment_appli_avt} There exists $\delta>0$ such that
  $
  \sup_{x\in B(x_\star,\delta)}\EE_\xi[\|\nabla f(x,\xi)\|^{8}]<\infty\,.
  $
\item\label{cond-noise}
 Defining $\Pi_{\text{u}}$ as the orthogonal projector on the eigenspace of
$V^{\frac 12} \nabla^2 F(x_\star) V^{\frac 12}$ that is associated with the
negative eigenvalues of this matrix, it holds that
\[
\Pi_{\text{u}} V^{\frac 12} \EE_\xi
  (\nabla f(x_\star, \xi) - \nabla F(x_\star) )
  (\nabla f(x_\star, \xi) - \nabla F(x_\star) )^\T V^{\frac 12}
  \Pi_{\text{u}} \neq 0.
\]
\end{enumerate}
\end{assumption}

\begin{theorem}
\label{th:avt-application}
Let Assumptions~\ref{hyp:hrpq}, \ref{hyp:model_bis}, and \ref{hyp:F_C4_S_C3},
\ref{hyp:hrpq_C1_deriv_bounded} hold true. Let $z_\star \in \Upsilon$ be such
that Assumption~\ref{hyp:avt} holds true for this $z_\star$.
Then, the eigenspace associated with the eigenvalues of $D$ with positive real parts has the same dimension as the eigenspace of $\nabla^2 F(x_\star)$
associated with the negative eigenvalues of this matrix.
Let $(z_n =(v_n,m_n,x_n):n\in \NN)$ be the random sequence generated by
Algorithm~\ref{algosto} with stepsizes satisfying $\sum_n \gamma_n = + \infty$
and~$\sum_n \gamma_n^2 < +\infty$.  Then, $\PP( [z_n \to z_\star] )= 0$.
\end{theorem}

The assumptions and the result call for some comments.

\begin{remark}
The definition of a trap as regards the general algorithm in the statement of
Th.~\ref{th:av-traps} is that the matrix $D$ in Eq.~\eqref{b-dl-z*} has
eigenvalues with positive real parts.  Th.~\ref{th:avt-application} states that
this condition is equivalent to $\nabla^2 F(x_\star)$ having negative
eigenvalues. What's more, the dimension of the invariant subspace of $D$
corresponding to the eigenvalues with positive real parts is equal to the
dimension of the negative eigenvalue subspace of $\nabla^2 F(x_\star)$. Thus,
Assumption~\ref{hyp:avt}--\ref{cond-noise} provides the ``largest'' subspace
where the noise energy must be non zero for the purpose of avoiding the trap.
\end{remark}

\begin{remark}
Assumptions~\ref{hyp:hrpq_C1_deriv_bounded}
and~\ref{hyp:avt}-\ref{rho_sum_carre_appli} are satisfied by many widely
studied algorithms, among which \textsc{RMSProp} and \textsc{Adam}.
\end{remark}

\begin{remark}
The results of Th.~\ref{th:avt-application} can be straightforwardly adapted to
the case of~\eqref{ode-adagrad}. Assumption~\ref{hyp:avt}-\ref{cond-noise} on
the noise is unchanged.
\end{remark}

In the case of the \textsc{S-NAG} algorithm, the assumptions become
particularly simple. We state the afferent result separately.

\subsubsection{Trap avoidance for \textsc{S-NAG}}
\label{subsubsec:aplit-avt-nest}

\begin{assumption}
\label{hyp:avt-nesterov}
Let $x_\star \in \zer \nabla F$ and
  let the following conditions hold.
\begin{enumerate}[{\it i)}]
  \item\label{trap-nesterov} The Hessian matrix $\nabla^2 F(x_\star)$ has a negative eigenvalue.
  \item\label{moment_appli_avt-nesterov} There exists $\delta>0$ such that
  $
  \sup_{x\in B(x_\star,\delta)}\EE_\xi[\|\nabla f(x,\xi)\|^{4}]<\infty\,.
  $
\item\label{cond-noise-nesterov}
$\tilde\Pi_{\text{u}} \EE_\xi (\nabla f(x_\star, \xi)
- \nabla F(x_\star) ) (\nabla f(x_\star, \xi) - \nabla F(x_\star) )^\T
\tilde\Pi_{\text{u}} \neq 0$, where $\tilde\Pi_{\text{u}}$ is the orthogonal projector on
the eigenspace of $\nabla^2 F(x_\star)$ associated with its negative eigenvalues.
\end{enumerate}
\end{assumption}

\begin{theorem}
  \label{th:avt-application-nesterov}
  Let Assumptions~\ref{hyp:hrpq}, \ref{hyp:model}, \ref{hyp:F_C4_S_C3} and~\ref{hyp:avt-nesterov} hold.
  Define $y_\star = (0, x_\star)$. Let $(y_n =(m_n,x_n):n\in \NN)$ be the random sequence given by Algorithm~\ref{algo-nesterov} with stepsizes satisfying $\sum_n \gamma_n = + \infty$ and~$\sum_n \gamma_n^2 < +\infty$.
  Then, $\PP( [y_n \to y_\star] )= 0$\,.
\end{theorem}

\subsection{Related works}

Up to our knowledge, all the avoidance of traps results that can be found in
the literature, starting from \cite{pem-90, bra-duf-96}, refer to stochastic
algorithms that are discretizations of autonomous ODE's
(see for e.g., \cite[Sec.~9]{ben-(cours)99} for general Robbins Monro algorithms and \cite[Sec.~4.3]{mertikopoulos-et-al20} for SGD).
In this line of research, a powerful class of techniques relies on Poincar\'e's invariant
manifold theorem for an autonomous ODE in a neighborhood of some unstable
equilibrium point. In our work, we extend the avoidance of traps results to a
non-autonomous setting, by borrowing a non-autonomous version of Poincar\'e's
theorem from the rich literature that exists on the subject
\cite{dal-krei-(livre)74, klo-ras-(livre)11}.

In \cite{gad-pan-saa18}, the authors succeeded in establishing an avoidance of
traps result for their non-autonomous stochastic algorithm which is close to
our \textsc{S-NAG} algorithm (see the discussion at the end of
Section~\ref{related-1storder-clt} above), at the expense of a sub-Gaussian
assumption on the noise and a rather stringent assumption on the stepsizes.
The main difficulty in the approach of \cite{gad-pan-saa18} lies in the use of
the classical autonomous version of Poincar\'e's theorem (see \cite[Remark
2.1]{gad-pan-saa18}).  This kind of difficulty is avoided by our approach,
which allows to obtain avoidance of traps results with close to minimal
assumptions.
More recently, in the contribution of \cite{gadat-gavra20} discussed
in Sec.~\ref{related-1storder-clt}, the authors establish an avoidance of traps
result (\cite[Th.~3]{gadat-gavra20}) for the algorithm described in Eq.~\eqref{algo-adagrad}
using techniques inspired from \cite{pem-90,ben-(cours)99}. As previously mentioned,
this recent work does not handle momentum and hence neither Algorithm~\ref{algosto} nor Algorithm~\ref{algo-nesterov}.
Moreover, as indicated in our discussion of \cite{gad-pan-saa18}, our strategy of proof is different.

Taking another point of view as concerns the trap avoidance, some recent works
\cite{lee-pan-pil-sim-jor-rec19,du-jin-lee-jor-sin-poc17,jin-ge-net-kak-jor17,pan-pil17,pan-pil-wang19}
address the problem of escaping saddle points when the algorithm is
deterministic but when the initialization point is random.  In contrast to this
line of research, our work considers a stochastic algorithm for which
randomness enters into play at each iteration of the algorithm via noisy
gradients.

% =============================================
%   Proofs
% =============================================

\section{Proofs for Section~\ref{sec:odes}}
\label{sec:proofs_odes}

\subsection{Proof of Th.~\ref{th:ode1}}
\label{subsec:prf-th-ode1}

The arguments of the proof of this theorem that we provide here follow the approach of \cite{das-gaz-20} with some small differences. Close arguments can be found in~\cite{barakat-bianchi21}. We provide the proof here for
completeness and in preparation of the proofs that will be related with the
stochastic algorithms.

\subsubsection{Existence and uniqueness}

The following lemma guarantees that the term $\sqrt{\sv(t) + \varepsilon}$
in~\eqref{ode-generale} is well-defined.

\begin{lemma}
\label{lemma:v>0}
Let $t_0 \in \RR_+$ and $T \in (t_0, \infty]$. Assume that there
exists a solution $\sz(t) = (\sv(t), \sm(t), \sx(t))$
to~\eqref{ode-generale} on $[t_0, T)$ for which $\sv(t_0) \geq 0$. Then, for
all $t \in [t_0, T)$, $\sv(t) \geq 0$.
\end{lemma}
\begin{proof}
Assume that $\nu \eqdef \inf\{ t \in [t_0, T), \, \sv(t) < 0 \}$ satisfies $\nu
< T$.  If $\sv(t_0) > 0$, Gronwall's lemma implies that $\sv(t) \geq \sv(t_0)
\exp(-\int_{t_0}^t q(t))$ on $[t_0, \nu]$ which is in contradiction with the
fact that $\sv(\nu) = 0$. If $\sv(t_0) = 0$, since $\nu < T$, there exists $t_1
\in (t_0, \nu)$ s.t. $\dot\sv(t_1) < 0$. Hence, using the first equation from
\eqref{ode-generale}, we obtain $\sv(t_1) > 0$. This brings us back to the
first case, replacing $t_0$ by~$t_1$.
\end{proof}

Recall that $F_\star = \inf F$ is finite by Assumption~\ref{hyp:F_coerc}.
Of prime importance in the proof will be the energy (Lyapunov) function
$\mE : \RR_+ \times \cZ_+ \to \RR$, defined as
\begin{equation}
  \label{eq:lyap}
\mE(h, z) = h (F(x) - F_\star) +
  \frac{1}{2} \left\| \frac{m}{(v+\varepsilon)^{\odot \frac 14}} \right\|^2\,,
\end{equation}
for every $h \geq 0$ and every $z = (v,m,x) \in \cZ_+$.
This function is slightly different from its analogues that were used in
\cite{alv-00,barakat-bianchi21,das-gaz-20}.

Consider $(t,z)\in (0,+\infty)\times \cZ_+$ and set
$z=(v,m,x)$. Then, using Assumption~\ref{hyp:F_loclip},
we can write
  \begin{align}
    %\frac{d}{dt}
    % \ps{\nabla \mE(\sh(t),z), (1,g(z,t))}
    % &=
    &\partial_t \mE(\sh(t),z) +\ps{\nabla_z \mE(\sh(t),z),g(z,t)}
     \nonumber \\
    &= \dot\sh(t) (F(x) - F_\star)
    -\frac 14 \ps{\frac{m^{\odot 2}}{(v+\varepsilon)^{\odot \frac 32}},\spp(t) S(x) - \sq(t) v}
      \nonumber\\
    &\phantom{=}
    + \ps{\frac{m}{(v+\varepsilon)^{\odot \frac 12}},\sh(t) \nabla F(x) - \sr(t) m}
     - \ps{\frac{m}{(v+\varepsilon)^{\odot \frac 12}}, \sh(t) \nabla F(x)}
     \nonumber \\
    &\leq - \left( \sr(t) - \frac{\sq(t)}{4} \right)
     \left\| \frac{m}{(v+\varepsilon)^{\odot \frac 14}} \right\|^2 %\nonumber \\
    %&
    + \dot\sh(t) (F(x) - F_\star)
     - \frac{\spp(t)}{4} \ps{S(x),
              \frac{m^{\odot 2}}{(v+\varepsilon)^{\odot \frac 32}}} .
\label{lyap-decrease}
  \end{align}

With the help of this function, we can now establish the existence, the uniqueness and the boundedness of the solution of~\eqref{ode-generale}
on~$[t_0,\infty)$ for an arbitrary~$t_0 > 0$.

\begin{lemma}
\label{lm:exist-uniq-bnd}
For each $t_0 > 0$ and $z_0 \in \cZ_+$, \eqref{ode-generale} has a unique
solution on $[t_0,\infty)$ starting at $\sz(t_0) = z_0$. Moreover, the
orbit $\{ \sz(t) \, : \, t\geq t_0 \}$ is bounded.
\end{lemma}
\begin{proof}
Let $t_0 > 0$, and fix $z_0 \in \cZ_+$. On each set of the type $[t_0, t_0 + A]
\times \bar B(z_0, R)$ where $A, R > 0$ and $\bar B(z_0, R) \subset
(-\varepsilon, \infty)^d \times\RR^d\times\RR^d$, we easily obtain from our
assumptions that the function $g$ defined in~\eqref{eq:gode} is continuous, and
that $g(\cdot, t)$ is uniformly Lipschitz on $t\in [t_0, t_0 + A]$. In these
conditions, Picard's theorem asserts that \eqref{ode-generale} starting from
$\sz(t_0) = z_0$ has a unique solution on a certain maximal interval $[t_0,
T)$. Lem.~\ref{lemma:v>0} shows that $\sv(t) \geq 0$ on this interval.

Let us show that $T = \infty$. Applying Ineq.~\eqref{lyap-decrease} with
$(v,m,x) = (\sv(t), \sm(t), \sx(t))$ and using
Assumption~\ref{hyp:hrpq}, we obtain that the function $t\mapsto \mE(\sh(t),
\sz(t))$ is decreasing on $[t_0, T)$. By the coercivity of $F$
(Assumption~\ref{hyp:F_coerc}) and Assumption~\ref{hyp:hrpq}--\ref{hyp:h}, we
get that the trajectory $\{\sx(t)\}$ is bounded.  Recall the equation $\dot
\sm(t) = \sh(t) \nabla F(\sx(t)) - \sr(t) \sm(t)$.  Using the continuity of the
functions $\nabla F$, $\sh$ and $\sr$ along with Gronwall's lemma, we get that
$\{ \sm(t) \}$ is bounded if $T < \infty$. We can show a similar result for $\{
\sv(t) \}$.  Thus, $\{ \sz(t) \}$ is bounded on $[t_0, T)$ if $T < \infty$
which is a contradiction, see, \emph{e.g.}, \cite[Cor.3.2]{har-(livre)02}.

It remains to show that the trajectory $\{ \sz(t) \}$ is bounded. To that end,
let us apply the variation of constants method to the equation
$\dot \sm(t) = \sh(t) \nabla F(\sx(t)) - \sr(t) \sm(t)$.
Writing $R(t)=\int_{t_0}^t \sr(u) \, du$, we get that
   \[
  \frac{d}{dt} \left( e^{R(t)} \sm(t) \right)
    = e^{R(t)} \sh(t) \nabla F(\sx(t)) .
  \]
  Therefore, for every $t \geq t_0$\,,
  \[
  \sm(t) = e^{-R(t)} \sm(t_0) + \int_{t_0}^t e^{R(u)-R(t)}
     \sh(u) \nabla F(\sx(u)) du \,.
  \]
  Using the continuity of $\nabla F$ together with the boundedness of $\sx$,
  Assumption~\ref{hyp:hrpq} and the triangle inequality, we obtain the
  existence of a constant $C > 0$ independent of $t$ s.t.
  \begin{multline*}
  \norm{\sm(t) - \sm(t_0)} - \norm{\sm(t_0)}
    \leq  C \sh(t_0) \int_{t_0}^t e^{-\int_u^t \sr(s)\, ds} du\\
    \leq C \sh(t_0) \int_{t_0}^t e^{-r_\infty(t-u)} du%\\
    \leq  \frac{C \sh(t_0)}{r_\infty}\,.
  \end{multline*}
  The same reasoning applies to $\sv(t)$ using the continuity of $S$ and
  Assumption~\ref{hyp:hrpq}.
  This completes the proof.
\end{proof}

We can now extend this solution to $t_0 = 0$ along the approach of
\cite{das-gaz-20}, where the detailed derivations can be found. The idea is to
replace $\sh(t)$ with $\sh(\max(\eta,t))$ for some $\eta > 0$ and to do the
same for $\spp$, $\sq$, and $\sr$. It is then easy to see that the ODE that is
obtained by doing these replacements has a unique global solution on $\RR_+$.
By making $\eta\to 0$ and by using the Arzel\`a-Ascoli theorem along with
Assumption~\ref{hyp:arzela-ascoli}, we obtain that~\eqref{ode-generale} has
a unique solution on $\RR_+$.

\subsubsection{Convergence}\label{sec:ode_conv}

The first step in this part consists in transforming \eqref{ode-generale} into
an autonomous ODE by including the time variable into the state vector. More
specifically, we start with the following ODE:
\[
\begin{bmatrix} \dot \sz(t) \\ \dot u(t) \end{bmatrix} =
\begin{bmatrix}  g(\sz(t), u(t)) \\ 1 \end{bmatrix}
\quad \text{with} \quad
\begin{bmatrix}  \sz(0) \\ u(0) \end{bmatrix} =
\begin{bmatrix}  z_0 \\ t_0 \end{bmatrix} ,
\]
then, we perform the following change of variable in time
\[
\begin{bmatrix}  z \\ u \end{bmatrix} \mapsto
\begin{bmatrix}  z \\ s  = 1/u\end{bmatrix}
\]
allowing the solution to lie in a compact set.

We initialize the above ODE at a time instant $t_0 > 0$.
Define the functions $\sH, \sR, \sP, \sQ : \RR_+ \to \RR_+$ by setting
$\sH(s) = \sh(1/s)$, $\sR(s) = \sr(1/s)$, $\sP(s) = \spp(1/s)$;
$\sQ(s) = \sq(1/s)$ for $s > 0$;
$\sH(0) = h_\infty$, $\sR(0) = r_\infty$, $\sP(0) = p_\infty$ and
$\sQ(0) = q_\infty$.
Our autonomous dynamical system can then be described by the following system of equations:
\begin{equation}
     \begin{cases}
       \dot \sv(t) &= \sP(\sss(t)) S(\sx(t)) - \sQ(\sss(t)) \sv(t)\\
       \dot \sm(t) &= \sH(\sss(t)) \nabla F(\sx(t)) - \sR(\sss(t)) \sm(t)\\
       \dot \sx(t) &= - \frac{\sm(t)}{\sqrt{\sv(t) + \varepsilon}}\\
       \dot \sss(t)&= - \sss(t)^2
     \end{cases}
\label{odeauton}
\end{equation}
Since the solution  of the ODE $\dot \sss(t) = - \sss(t)^2$ for which
$\sss(t_0) = 1/t_0$ is $\sss(t) = 1/t$, the trajectory $\{\sss(t)\}$ is
bounded.  The three remaining equations are a reformulation
of~\eqref{ode-generale} for which the trajectories have already been shown to
exist and to be bounded in Lem.~\ref{lm:exist-uniq-bnd}.  In the sequel, we
denote by $\Phi : \cZ_+ \times \RR_+ \to \cZ_+ \times \RR_+$ the semiflow
induced by the autonomous ODE~\eqref{odeauton}, \emph{i.e.}, for every $u =
(z,s) \in \cZ_+ \times \RR_+$, $\Phi(u, \cdot)$ is the unique global solution
to the autonomous ODE~\eqref{odeauton} initialized at $u$. Observe that the
orbits of this semiflow are precompact. Moreover, the function
$\Phi((z,0), \cdot)$ is perfectly defined for each $z\in \cZ_+$ since the
associated solution satisfies the ODE \eqref{ode-infty-extended} defined below,
which three first equations satisfy the hypotheses of
Lem.~\ref{lm:exist-uniq-bnd}.

Consider now a continuous function $V : \cZ_+ \times \RR_+ \to \RR$ defined by:
\[
V(u) = \mE\left( \sH(s), z\right),
 \quad u = (z,s) \in \cZ_+ \times (0,\infty).
\]
As for Ineq.~\eqref{lyap-decrease} above, we have here that
\begin{multline*}
\frac{d}{dt} V\left(\Phi(u,t)\right) \leq
    - \left( \sr(t) - \frac{\sq(t)}{4} \right)
     \left\| \frac{\sm(t)}{(\sv(t)+\varepsilon)^{\odot \frac 14}} \right\|^2
     \\
    + \dot\sh(t) (F(\sx(t)) - F_\star)
     - \frac{\spp(t)}{4} \ps{S(\sx(t)),
              \frac{\sm(t)^{\odot 2}}{(\sv(t)+\varepsilon)^{\odot \frac 32}}}
\end{multline*}
if $s > 0$, and the same inequality with $(\dot\sh(t), \spp(t), \sr(t),
 \sq(t))$ being replaced with $(0, p_\infty, r_\infty, q_\infty)$ otherwise.

Since $V\circ \Phi(u,\cdot)$ is non-increasing and nonnegative, we can define
$V_\infty \eqdef \lim_{t\to\infty} V(\Phi(u, t))$.  Let $\omega(u) \eqdef
\bigcap_{s>0} \overline{\bigcup_{t \geq s} \Phi(u,t)} $ be the $\omega$-limit
set of the semiflow $\Phi$ issued from $u$. Recall that $\omega(u)$ is an
invariant set for the flow $\Phi(u,\cdot)$, and that
\[
\dist(\Phi(u,t), \omega(u)) \xrightarrow[t\to\infty]{} 0,
\]
see, \emph{e.g.}, \cite[Th.~1.1.8]{har-(livre)91}). In order to finish the
proof of Th.~\ref{th:ode1}, we need to make explicit the structure of
$\omega(u)$.

We know from La Salle's invariance principle that
$\omega(u) \subset V^{-1}(V_\infty)$. In particular,
\begin{equation}
\forall y \in \omega(u), \ \forall t \geq 0, \
  V(\Phi(y,t)) = V(y) = V_\infty
\label{eq:invariance}
\end{equation}
by the invariance of $\omega(u)$.

From ODE~\eqref{odeauton}, we have that any $y \in \omega(u)$ is of the form $y
= (z, 0)$ since $s(t) \to 0$. As a consequence, $\Phi(y,\cdot)$ is a
solution to the autonomous ODE
\begin{equation}
\begin{cases}
    \dot \sv(t) &= p_\infty S(\sx(t)) - q_\infty \sv(t) \\
    \dot \sm(t) &= h_\infty \nabla F(\sx(t)) - r_\infty \sm(t) \\
    \dot \sx(t) &= - \frac{\sm(t)}{\sqrt{\sv(t) + \varepsilon}}  \\
    \dot \sss(t) &= 0\,.
\end{cases}
\label{ode-infty-extended}
\end{equation}
The three first equations can be written in a more compact form :
\begin{equation}
\label{odeinfty}
\dot\sz(t) = g_\infty(\sz(t))
\end{equation}
where $\sz(t)= (\sv(t), \sm(t), \sx(t))$, and
\[
g_\infty(z) = \lim_{t \to \infty} g(z,t)
= \begin{bmatrix}
   p_\infty S(x) - q_\infty v \\
   h_\infty \nabla F(x) - r_\infty m \\
   - {m} / {\sqrt{v+\varepsilon}} \end{bmatrix}
\]
for each $z \in \cZ_+$. Consider $y = (v, m, x, 0) \in \omega(u)$.
Using Eq.~\eqref{eq:invariance}, we obtain that $d V(\Phi(y,t)) / dt = 0$,
which implies that
\[
  \left( r_\infty - \frac{q_\infty}{4} \right)
   \left\| \frac{\sm(t)}{(\sv(t)+\varepsilon)^{\odot \frac 14}} \right\|^2
     + \frac{p_\infty}{4} \ps{S(\sx(t)),
              \frac{\sm(t)^{\odot 2}}{(\sv(t)+\varepsilon)^{\odot \frac 32}}} = 0
\]
for all $(\sv(t), \sm(t), \sx(t), 0) = \Phi(y, t)$.  As a
consequence, Assumption~\ref{hyp:hrpq}-\ref{hyp:stabode} gives $\sm(t) = m =
0$, and then, $\sx(t) = x$ for some $x$ s.t. $\nabla F(x) = 0$ using
ODE~\eqref{ode-infty-extended}.  We now turn to showing that $\sv(t) = v =
p_\infty S(x) / q_\infty$.  We have proved so far that any element $y \in
\omega(u)$ is written $y = (v, 0, x, 0)$ where $\nabla F(x)
= 0$.  The component $\sv(\cdot)$ of $\Phi(y, \cdot)$ is a solution to the ODE
$\dot\sv(t) = p_\infty S(x) - q_\infty \sv(t)$ and is thus written
\begin{equation}
\label{vdec}
\sv(t) = \frac{p_\infty S(x)}{q_\infty} +
 e^{-q_\infty t}  \Bigl( v - \frac{p_\infty S(x)}{q_\infty} \Bigr) .
\end{equation}
Fixing $x$, let $\mS_x$ be the section of $\omega(u)$ defined by:
\[
\mS_x \omega(u) = \left\{ y \in \omega(u) \, :\,
  y = (\tilde v, 0, x, 0)\,,\, \tilde v \in \RR_+^d \right\} .
  %y = (\times, 0^\T, x^\T, 0)^\T \right\} .
\]
As $\omega(u)$ is invariant, we have $\mS_x\omega(u) = \mS_x \Phi(\omega(u),
t)$ for all $t \geq 0$. Since the set $\{ \tilde v \in \RR_+^d \,\text{s.t.}\,
(\tilde v, 0, x, 0) \in \mS_x \omega(u)\}$ lies in a
compact, we deduce from Eq.~\eqref{vdec} that this set is reduced to the
singleton $\{p_\infty S(x) / q_\infty\}$ and in particular $v = p_\infty S(x) /
q_\infty$.  Therefore, the union of $\omega$-limit sets of the semiflow $\Phi$
induced by ODE~\eqref{odeauton} coincides with the set of equilibrium points of
this semiflow. The latter set itself corresponds to the set of points $(z,0)$ s.t. $z\in \zer g_\infty$. It remains to notice that
$\Upsilon = \zer g_\infty$ to finish the proof.

\begin{remark} Commenting on Remark~\ref{rmk:simp_ode}, the same proof works
for~\eqref{ode-adagrad} by using the function $F-F_\star$ as a Lyapunov
function. The corresponding limit set (as $t\to+\infty$) is then of the form
\[
\{ \tilde z_\infty = (\tilde v_\infty, \tilde
x_\infty) \in \RR_{+}^d \times \RR^d \, : \, \nabla F(\tilde x_\infty) = 0\,,
\tilde v_\infty = p_\infty S(\tilde x_\infty) / q_\infty\}.
\]
Similarly, if we set $\spp = \sq \equiv 0$ in \eqref{ode-generale} and we keep what remains in
Assumption~\ref{hyp:hrpq}, the function
$\sh(t) (F(x) - F_\star) + \frac{1}{2} \| m\|^2$ works as a Lyapunov function,
and the limit set has the form $\{ (0, x) : \nabla F(x) = 0\}$.
\end{remark}

\subsection{Proof of Th.~\ref{th:ode-nesterov}}\label{sec:proof_simp}

%\subsubsection{Existence and uniqueness}
The existence and the uniqueness of the solution to~\eqref{ode-true-nesterov} have been shown in the literature. We refer to \cite[Prop.~2.1-2.2.c)]{cab-eng-gad09} for an identical statement of this result and \cite[Th.~1, Appendix~A]{su_boyd_candes2016} for a complete proof. The boundedness of the solution follows immediately from the coercivity of $F$ together with the fact that the function $t \mapsto F(\sx(t)) + \frac 12 \|\sm(t)\|^2$ is nonincreasing.

%\subsubsection{Convergence}
Concerning the convergence statement, our proof is based on comparing the solutions of~\eqref{ode-true-nesterov}
%to the so-called Heavy-Ball ODE (with asymptotically small %dissipation)
to the solutions of the ODE in \cite[Eq.~(2.3)]{gad-pan-saa18}.
We first note that under a change of variable,
a solution to~\eqref{ode-true-nesterov} gives a solution to \cite[Eq.~(2.3)]{gad-pan-saa18}.

\begin{lemma}\label{lm:ch_var}
  Let $(\sm, \sx)$ be a solution to~\eqref{ode-true-nesterov}. Define $\sy(t) =  \frac{\kappa \sm\left(\kappa \sqrt{t}\right)}{2\sqrt{t}}$\,, $\su(t) =  \sx\left(\kappa\sqrt{t}\right)$\,, with $\kappa = \sqrt{2 \alpha + 2}$ and $\beta = \frac{\kappa^2}{4}$.
  Then, $(\sy, \su)$ verifies
  \begin{equation}\label{eq:odeGad}
    \begin{cases}
      \dot{\sy}(t) &= \frac{\beta}{t}( \nabla F (\su(t))) - \sy(t)) \\
      \dot{\su}(t) &= -\sy(t)\,.
    \end{cases}
  \end{equation}
\end{lemma}
\begin{proof}
By simple differentiation, we get:
\begin{equation*}
  \dot{\sy}(t) = \frac{\beta}{t} \left[\nabla F\left(\sx  (\kappa \sqrt{t})\right) - \frac{\alpha}{\kappa \sqrt{t}}  \sm\left(\kappa\sqrt{t}\right)\right] - \frac{\kappa}{4 t^{\frac 32}}\sm\left(\kappa\sqrt{t}\right)
      = \frac{\beta}{t}\left( \nabla F( \su(t)) -  \sy(t)\right),
\end{equation*}
\begin{equation*}
  \dot{\su}(t) = -\frac{\kappa}{2 \sqrt{t}} \sm\left(\kappa \sqrt{t}\right) =  - \sy(t) \,.%\quad \text{and}
\end{equation*}
\end{proof}

Consider a solution $(\sm, \sx)$ of~\eqref{ode-true-nesterov} starting at
$(m_0, x_0) \in \RR^d \times \RR^d$. As in Section~\ref{sec:ode_conv},
for every $t_0>0$, on $[t_0, +\infty)$, we have that $(\sm, \sx, \sss)$ is a
solution to the autonomous ODE
\begin{equation}\label{eq:odeNest_aut}
  \begin{cases}
    \dot{\sm}(t) &= \nabla F(\sx(t)) - \alpha \sss(t)\sm(t) \\
     \dot{\sx}(t) &= - \sm(t)\\
    \dot{\sss}(t) &= - \sss(t)^2\,,
  \end{cases}
\end{equation}
starting at $(m_0, x_0, 1/t_0)$. Denote by $\Phi_N = (\Phi_N^m, \Phi_N^x, \Phi_N^s)$ the semiflow induced by ODE~\eqref{eq:odeNest_aut} and $\omega_N((m_0, x_0, 1/t_0))$ its limit set.

Define $(\sy, \su)$ as in Lem.~\ref{lm:ch_var}. Starting at $(\sy(t_0), \su(t_0), 1/t_0)$, we also have that $(\sy, \su, \sss)$ is a solution on $[t_0, + \infty)$ to the ``autonomized'' Heavy-Ball ODE

\begin{equation}\label{eq:odeGad_aut}
  \begin{cases}
    \dot{\sy}(t) &= \beta \sss(t)( \nabla F (\su(t))) - \sy(t)) \\
     \dot{\su}(t) &= - \sy(t)\\
    \dot{\sss}(t) &= - \sss(t)^2\,.
  \end{cases}
\end{equation}

Denote by $\Phi_H =  (\Phi_H^y, \Phi_H^u, \Phi_H^s)$ the semiflow induced by ODE~\eqref{eq:odeGad_aut} and $\omega_H((\sy(t_0), \su(t_0), 1/t_0))$ its limit set.

\begin{lemma}\label{lm:rel_comp}
  For any compact set $K \subset \RR^{2d+1}$ and any $T > 0$, the family of functions
  $\left\{\Phi(z, \cdot) : [0,T] \to \RR^{2d+1} \right\}_{z \in K}\,,$
   where $\Phi$ is either $\Phi_H$ or $\Phi_N$, is relatively compact in $(\mathcal{C}^0([0,T], \RR^{2d +1}), \lVert \cdot\rVert_{\infty})$.
\end{lemma}
\begin{proof}
  The map $\Phi : \RR^{2d +1} \times \RR_{+} \rightarrow \RR^{2d +1} $ is continuous, hence uniformly continuous on $K \times [0, T]$. The result follows from the application of the Arzel\`a-Ascoli theorem to the family $\left\{\Phi(z, \cdot) : [0,T] \to \RR^{2d+1} \right\}_{z \in K}$.
\end{proof}

Let $(m, x, 0) \in \omega_N((m_0, x_0, 1/t_0))$. There exists a sequence $(t_k)$ of nonnegative reals such that $(m, x, 0) = \lim_{k \rightarrow \infty} (\sm(t_k), \sx(t_k), 1/t_k)$. For any $T>0$\,, using Lem.~\ref{lm:rel_comp}, up to an extraction, we can say that the sequence of functions
$\{\Phi_N((\sm(t_k), \sx(t_k), 1/t_k), \cdot)\}_k$ converges towards $(\tilde{\sm}, \tilde{\sx}, 0)$ in $\mathcal{C}^0([0, T], \RR^d)$, where $(\tilde{\sm}, \tilde{\sx})$ is a solution to
\begin{equation}\label{eq:lim_nest}
  \begin{cases}
    \dot{\tilde{\sm}}(t) &= \nabla F(\tilde{\sx}(t))\\
    \dot{\tilde{\sx}}(t) &= - \tilde{\sm}(t)\,,\\
    \end{cases}
\end{equation}
with $ (\tilde{\sm}(0), \tilde{\sx}(0)) = (m,x)$.
Moreover, by Lem.~\ref{lm:ch_var}, we also have that:
%\begin{equation}\label{eq:conv_ch_var}
  %\begin{split}
  \begin{multline}
    \label{eq:conv_ch_var}
      %&
      \sup_{h \in [0, T^2/\kappa^2]} \norm{ \tilde{\sx}(\kappa\sqrt{h}) - \Phi_N^x((\sm(t_k), \sx(t_k), 1/t_k), \kappa \sqrt{h})}\\
       %=&
       = \sup_{h \in [0, T^2/\kappa^2]} \norm{ \tilde{\sx}(\kappa\sqrt{h}) - \Phi_H^u((\sm(t_k), \sx(t_k), 1/t_k), h)} \xrightarrow[k\to+\infty]{} 0\,.
  \end{multline}
  %\end{split} \, .
%\end{equation}

Using Lem.~\ref{lm:rel_comp}, up to an additional extraction, we get on $\mathcal{C}^0([0, T^2/\kappa^2], \RR^{2d +1})$ that $\{\Phi_H((\sx(t_k), \sm(t_k), 1/t_k), \cdot)\}_k$ converges to $(\su, \sy, 0)$, where $(\su, \sy)$ is a solution to
\begin{equation}\label{eq:lim_gad}
  \begin{cases}
    \dot{\sy}(t) &=  0\\
    \dot{\su}(t) &= - \sy(t)\,.
    \end{cases}
\end{equation}

Therefore, $\su(t) = A + Bt $ for some $A$ and $B$ in $\RR^{d}$.
Imagine that $B \neq 0$. We previously proved that $\sx$ (and therefore $\su$) is bounded by some constant $C > 0$. Let $T' > \frac{C + \|A\|}{\|B\|}$. Up to an extraction, we obtain that $\{\Phi_H((\sx(t_k), \sm(t_k), 1/t_k), \cdot)\}_k$ converges to $\su'$ on $\mathcal{C}^0([0, T'], \RR^{2d +1})$, with $\su'(t) = A' + B't$ for some $A'$ and $B'$ in $\RR^{d}$. We then have by uniqueness of the limit that $A' = A$ and $B' = B$.
 As a consequence, $\norm{\su'(T')}  = \norm{A + B T'} > C$ and we obtain a contradiction. Hence~$B = 0$.

 This implies that $\su$ is constant. Then, if we go back to Eqs.~\eqref{eq:conv_ch_var} and~\eqref{eq:lim_nest}, we get that $\tilde{\sx}$ is constant, hence $\tilde{\sm} \equiv 0$ and then $\nabla F(\tilde{\sx}) \equiv 0$. In particular, this means that $m = \tilde{\sm}(0) = 0$ and $\nabla F(x) = \nabla F(\tilde{\sx}(0))= 0$.

% ================== Proofs for the stochastic algorithm

\section{Proofs for Section~\ref{sec:as_convergence}}
\label{sec:proof_as_convergence}

 \subsection{Preliminaries}

We first recall some useful definitions and results.  Let $\Psi$ represent any
semiflow on an arbitrary metric space $(E,\sd)$. As in the previous section,
a point $z\in E$ is called an
equilibrium point of the semiflow $\Psi$ if $\Psi(z,t)=z$ for all $t\geq
0$.  We denote by $\Lambda_\Psi$ the set of equilibrium points of~$\Psi$.  A
continuous function $\sV:E\to\bR$ is called a Lyapunov function for the
semiflow $\Psi$ if $\sV(\Psi(z,t))\leq \sV(z)$ for all $z\in E$ and all $t\geq
0$.  It is called a \emph{strict} Lyapunov function if, moreover, $ \{ z\in
E\,:\, \forall t\geq 0,\,\sV(\Psi(z,t))=\sV(z) \}= \Lambda_\Psi $.  If $\sV$ is
a strict Lyapunov function for~$\Psi$ and if $z\in E$ is a point s.t.
$\{\Psi(z,t):t\geq 0\}$ is relatively compact, then it holds that
$\Lambda_\Psi\neq \varnothing$ and $\sd(\Psi(z,t),\Lambda_\Psi)\to 0$, see
\cite[Th.~2.1.7]{har-(livre)91}.  A continuous function $z:[0,+\infty)\to E$ is
said to be an asymptotic pseudotrajectory (APT, \cite{ben-hir-96}) for the
semiflow $\Psi$ if $ \lim_{t\to+\infty}
\sup_{s\in [0,T]} \sd(z(t+s),\Psi(z(t),s)) = 0\,$ for every $T\in (0,+\infty)$\,.

% Our proof is based on the application of
% The following result stems from \cite[Th.~5.7]{ben-(cours)99} and \cite[Prop.~6.4]{ben-(cours)99}.
% \begin{proposition}[\cite{ben-(cours)99}]%\hfill\\
% \label{prop:benaim}
%   Consider a semiflow $\Psi$ on $(E,\sd)$ and a map $z: \RR_+\to E$. Assume the following:
%   \begin{enumerate}[{\it i)}]
%     \setlength{\itemsep}{0pt}
%   \item $\Psi$ admits a strict Lyapunov function $\sV$.
%   \item The set $\Lambda_\Psi$ of equilibrium points of $\Psi$ is compact.
%   \item  $\sV(\Lambda_\Psi)$ has an empty interior.
%   \item $z$ is an APT of $\Psi$.
%   \item $z(\RR_+)$ is relatively compact.
%   \end{enumerate}
% Then, $
% \bigcap_{t\geq 0}\overline{z([t,\infty))}$ is a compact connected subset of $\Lambda_\Psi$\,.
% \end{proposition}

\subsection{Proof of Th.~\ref{th:as_conv_under_stab}}\label{proof:th_as_conv}

%Our proof is based on the application of Prop.~\ref{prop:benaim}.
Recall that $\Phi$ is the semiflow induced by the autonomous ODE~\eqref{odeauton} which is an ``autonomized'' version of our initial \eqref{ode-generale}. In the remainder of this section, the proof will be divided into two main steps : (a) we show that a certain continuous-time linearly interpolated process
constructed from the iterates of our algorithm~\ref{algosto} is an APT of~$\Phi$; (b) we exhibit a strict Lyapunov function for a restriction to a carefully chosen compact set of a well chosen semiflow related to $\Phi$.
Then, we characterize the limit set of the APT using \cite[Th.~5.7]{ben-(cours)99} and \cite[Prop.~3.2]{benaim96}.
The sequence $(z_n)$ converges almost surely to this same limit set.\\

\noindent\textbf{(a) APT.}
For every $n \geq 1$, define $\bar z_n = (v_n,m_n,x_{n-1})$ (note the shift in the index of the variable $x$).
We have the decomposition
$$
\bar z_{n+1} =
\bar z_n + \gamma_{n+1} g(\bar z_n, \tau_n) + \gamma_{n+1} \eta_{n+1} + \gamma_{n+1} \varsigma_{n+1}\,,
$$
where $g$ is defined in Eq.~(\ref{eq:gode}),
\begin{equation}
  \label{eq:noise_eta}
  % forme colonne
  % \eta_{n+1} =
  % \begin{bmatrix}
  %  p_{n} (\nabla f(x_{n}, \xi_{n+1})^{\odot 2} - S(x_{n}) ) \\
  %  h_{n} (\nabla f(x_{n}, \xi_{n+1}) - \nabla F(x_{n}) ) \\
  %  0
  % \end{bmatrix}\,,
   \eta_{n+1} = \left(p_{n} (\nabla f(x_{n}, \xi_{n+1})^{\odot 2} - S(x_{n})),\,h_{n} (\nabla f(x_{n}, \xi_{n+1}) - \nabla F(x_{n})),\, 0 \right)\,,
\end{equation}
is a martingale increment and where we set
% $\eta_{n+1} = (\eta_{n+1}^v,\eta_{n+1}^m,\eta_{n+1}^x)$ s.t.
% $$
% \begin{cases}
%   \eta_{n+1}^v &= p_n \left(\partial_1 f(x_n,\xi_{n+1})^{\odot 2}-S(x_n)\right)\\
%   \eta_{n+1}^m &= h_n \left(\partial_1 f(x_n,\xi_{n+1})-\nabla F(x_n)\right) \\
%   \eta_{n+1}^x &= 0\,,
% \end{cases}
% $$
% and
$\varsigma_{n+1} = (\varsigma_{n+1}^{v},\varsigma_{n+1}^{m},\varsigma_{n+1}^{x})$
with the components defined by:
$$
\begin{cases}
  \varsigma_{n+1}^v &= p_n (S(x_n) - S(x_{n-1}))\\
  \varsigma_{n+1}^m &= h_n (\nabla F(x_n) - \nabla F(x_{n-1}))\\
  \varsigma_{n+1}^x &= (\frac{\gamma_n}{\gamma_{n+1}} - 1) \frac{m_n}{\sqrt{ v_n + \varepsilon}}\,.
\end{cases}
$$
We first prove that $\varsigma_n\to 0$ a.s. by considering the components separately.
The components~$\varsigma_{n+1}^m$ and $\varsigma_{n+1}^v$ converge a.s. to zero by using Assumptions~\ref{hyp:F_loclip}, \ref{hyp:S_loclip}, together with the boundedness of the sequences $(p_n)$ and $(h_n)$ (which are both convergent). Indeed, since $\nabla F$ is locally Lipschitz continuous and the sequence $(z_n)$ is supposed to be almost surely bounded, there exists a constant~$C$ s.t. $\|\nabla F(x_n) - \nabla F(x_{n-1})\| \leq C \|x_n-x_{n-1}\| \leq \frac{C}{\varepsilon}\gamma_n\|m_n\|$. The same inequality holds when replacing $\nabla F$ by $S$ which is also locally Lipschitz continuous.
The component $\varsigma_{n+1}^x$ also converges a.s. to zero by observing that $\|\varsigma_{n+1}^x\| \leq |1- \frac{\gamma_n}{\gamma_{n+1}}|.\|m_n\|/\sqrt{\varepsilon}$ and using Assumption~\ref{hyp:stepsizes} together with the a.s. boundedness of $(z_n)$.
Now consider the martingale increment sequence $(\eta_n)$, adapted to $\cF_n$. Take $\delta > 0$. Since $(z_n)$ is a.s bounded, there is a constant $C' > 0$ such that $\PP(\sup \norm{x_n} > C') \leq \delta$. Denoting $\tilde{\eta}_n \eqdef \eta_{n}\1_{\norm{x_n} \leq C'}$ and combining Assumptions~\ref{hyp:hrpq} with \ref{hyp:moment}-\ref{hyp:moment-q} we can show using convexity inequalities that
\[
\sup_n \EE \|\tilde{\eta}_{n+1}\|^{q} < \infty.
\]
Then, we deduce from this result together with the corresponding stepsize assumption from \ref{hyp:moment}-\ref{hyp:moment-q} and \cite[Prop.~4.2]{ben-(cours)99} (see also \cite[Prop.~8]{met-pri-87}) the key property:
%\begin{enumerate}[resume,label=AS{\arabic*}]
%\item\label{arc}
\begin{equation}
  \label{eq:arc_tilde}
\forall T > 0\,, \quad
 \max \Bigl\{
\Bigl\| \sum_{k=n}^{L-1} \gamma_{k+1} \tilde{\eta}_{k+1} \Bigr\|
 \ : \ L = n+1, \ldots, J(\tau_n + T) \Bigr\}
 \toaslong 0
\end{equation}
where
$%\[
J(t) = \max \{ n \geq 0 \, : \, \tau_n\leq t \}
$%\].
%\end{enumerate}
. Hence, for all $T>0$, with probability at least $1 - \delta$ :
\begin{equation}
  \label{eq:arc}
%\forall T > 0\,, \quad
 \max \Bigl\{
\Bigl\| \sum_{k=n}^{L-1} \gamma_{k+1} \eta_{k+1} \Bigr\|
 \ : \ L = n+1, \ldots, J(\tau_n + T) \Bigr\} \,
 \xrightarrow[n \to \infty]{} 0 \,.
\end{equation}
Since $\delta$ can be chosen arbitrary small, Eq.~\eqref{eq:arc} remains true with probability 1.
 This result also holds under Assumption~\ref{hyp:moment}-\ref{hyp:subgauss-noise} (instead of \ref{hyp:moment}-\ref{hyp:moment-q}) by applying \cite[Prop.~4.4]{ben-(cours)99}.

 Let $\bs z : [0,+\infty)\to \cZ_+$ be the continous-time linearly interpolated process given by
 \[
 \bs z(t) = \bar z_n + (t-\tau_n) \frac{\bar z_{n+1}- \bar z_n}{\gamma_{n+1}} \qquad \left(\forall n \in \NN\,,\, \forall t \in [\tau_n,\tau_{n+1})\right)
 \]
 (where $\tau_n = \sum_{k=1}^n\gamma_k$).
 Let $t_0 > 0$. Define $\bs u : [t_0, \infty) \to \cZ \times (0,1/t_0]$ by
 \[
 \bs u(t) = \begin{bmatrix} \bs z(t) \\ 1/t \end{bmatrix}, \quad
  \text{for}\quad t \geq t_0 > 0.
 \]

 Using Eq.~\eqref{eq:arc} and the almost sure boundedness of the sequence $(z_n)$ along with the fact that $\varsigma_n$ converges a.s. to zero, it follows from \cite[Prop.~4.1, Remark~4.5]{ben-(cours)99} that
 $\bs u(t)$ is an APT of the already defined semiflow $\Phi$ induced by~\eqref{odeauton}. Remark that it also holds that $\bs z(t)$ is an APT of the semiflow $\Phi^\infty$ induced by~(\ref{odeinfty}).
 %{\color{blue}[Anas] cette dernière remarque permet de s'affranchir %d'introduire les ICT sets}.
 As the trajectory of $\bs u(t)$ is precompact, the limit set
 \[
 \bs L(\bs u) = \bigcap_{t\geq t_0} \overline{\bs u([t,\infty))}
 \]
 %is an ICT set  of the semiflow $\Phi$ (see \cite[Th.~5.7]{ben-(cours)99}).
 is compact.
 Moreover, it has the form
 \begin{equation}
   \label{eq:S}
 \bs L(\bs u) = \begin{bmatrix} \bs S \\ 0 \end{bmatrix}\,,
 \quad
 \text{where}
 \quad
 \bs S \eqdef \bigcap_{t\geq t_0} \overline{\bs z([t,\infty))}\,.
\end{equation}
%where $\bs S \eqdef \bigcap_{t\geq t_0} \overline{\bs z([t,\infty))}$\,.
 Our objective now is to prove that
 \begin{equation}
 \label{Seq}
 \bs S \subset \Lambda_{\Phi^{\infty}}\,.%\Eq g_\infty
 %\quad
 %\text{(see Eq.}~\eqref{odeinfty}\text{).}
 \end{equation}
 In order to establish this inclusion, we study the behavior of the restriction $\Phi|\bs L$ of the semiflow $\Phi$ to the set $\bs L$ (which is well-defined since $\bs L$ is $\Phi$-invariant).
 Remark that
 \[
 \Phi|\bs L = \begin{bmatrix} \Phi^\infty | \bs S \\ 0 \end{bmatrix},
 \]
 where $\Phi^\infty$ is the semiflow associated to \eqref{odeinfty}.
 %the ODE~\eqref{odeinfty}.
 %We immediately see that $\bs S$ is an ICT set for $\Phi^\infty$.
 In the second part of the proof, we establish Eq.~\eqref{Seq} combining item (a) we just proved with \cite[Th.~5.7]{ben-(cours)99} and \cite[Prop.~6.4]{ben-(cours)99}. In order to use the latter proposition, we prove a useful proposition in item (b).
 % To establish Eq.~\eqref{Seq} and conclude the proof, we exhibit a strict Lyapunov function for $\Eq g_\infty$ (see \cite[\S 6.2]{ben-(cours)99}) and apply \cite[Prop.~6.4]{ben-(cours)99}. This proposition shows that every ICT set for $\Phi^\infty$ is included in $\Eq g_\infty$ if the image of the set $\Eq g_\infty$ by the aforementioned Lyapunov function is of empty interior.

\noindent\textbf{(b) Strict Lyapunov function and convergence.}
For every $\delta>0$ and every $z = (v,m,x)\in \cZ_+$, define:
\begin{equation}
W_\delta(v,m,x) \eqdef \mE_{\infty}(z)
  - \delta\ps{\nabla F(x), m} + \delta \| q_\infty v - p_\infty S(x) \|^2\,,
\label{eq:Wdelta}
\end{equation}
%where $\mE_{\infty}$ is defined in Eq.~\eqref{eq:lyap_infty}.
where, under Assumption~\ref{hyp:hrpq}-\ref{hyp:h}, the function $\mE_{\infty}$ is defined by
\begin{equation}
  \label{eq:lyap_infty}
\mE_\infty(z) \eqdef \lim_{t \to +\infty} \mE(t, z)
              = h_\infty (F(x) - F_\star)
              + \frac{1}{2} \left\| \frac{m}{(v+\varepsilon)^{\odot \frac 14}} \right\|^2\,.
\end{equation}
%for every $z = (v^{\T},m^{\T},x^{\T})^{\T} \in \cZ_+$.
%%$\mE_{\infty}(z) \eqdef \lim_{t \to \infty} \mE(t,z)$ %which is well-defined given Eq.~\eqref{eq:lyap} and %Assumption~\ref{hyp:hrpq}-\ref{hyp:h}.

\begin{proposition}
  \label{prop:Wstrict}
Let $t_0 > 0$ and let
Assumptions~\ref{hyp:F_loclip} to %, \ref{hyp:S_loclip}, \ref{hyp:F_coerc},
\ref{hyp:hrpq} %, \ref{hyp:arzela-ascoli}
and \ref{hyp:sard} hold true.
% Let $K\subset \cZ_+$ be a compact set. Define $K'\eqdef \overline{\{\flot^\infty(t,z):t\geq 0, z\in K\}}$.
% Let $\bflot^\infty:[0,+\infty)\times K'\to K'$ be the restriction of the semiflow $\flot^\infty$ to $K'$ \emph{i.e.},
% $\bflot^\infty(t,z) = \flot^\infty(t,z)$ for all  $t\geq 0, z\in K'$.
Let $\bs S$ be the limit set defined in Eq.~\eqref{eq:S}.
Let $\bflot^\infty: \bs S \times [t_0,+\infty) \to \bs S$ be the restriction of the semiflow $\flot^\infty$ to $\bs S$ \emph{i.e.},
$\bflot^\infty(z,t) = \flot^\infty(z,t)$ for all $z\in \bs S, t \geq t_0$.%t\geq 0$.
Then,
\begin{enumerate}[{\it i)}]
  \setlength{\itemsep}{0pt}
\item $\bs S$ is compact.
\item $\bflot^\infty$ is a well-defined semiflow on $\bs S$.
\item \label{eq_restricted} The set of equilibrium points of $\bflot^\infty$ is equal to $\Lambda_{\Phi^\infty} \cap \bs S$.
%$\Eq g_\infty \cap \bs S$.
\item There exists $\delta>0$ s.t. $W_\delta$ is a strict Lyapunov function for the semiflow~$\bflot^\infty$.
\end{enumerate}
\end{proposition}

\begin{proof}
  The first point is a consequence of the definition of $\bs S$ and the boundedness of $\bs z$. %Prop.~\ref{prop:boundedness_ode_infty}.
  % {\color{blue} [Anas] Ce résultat de bornitude uniforme en le point d'initialisation n'est pas démontré pour le moment.}
  The second point stems from the definition of $\flot^\infty$.
  % peut-être serait-il utile de rajouter proprement une Prop.~\ref{prop:semiflow} pour définir le flot correctement.
  Observing that $\bflot^\infty$ is valued in $\bs S$, the third point is immediate from the definition of~$\Lambda_{\Phi^\infty}$.
  %$\Eq g_\infty$.
  We now prove the last point.
  Consider $z\in \bs S$ and write $\bflot^{\infty}(z,t)$ under the form $\bflot^{\infty}(z,t) = (\sv(t),\sm(t),\sx(t))$. Notice that this quantity is bounded as a function of the variable $t$. For \emph{any} map ${\mathsf W}:\cZ_+\to\bR$, define
  for all $t \geq t_0$, $ \cL_{\mathsf W}(t) \eqdef \limsup_{s\to 0} s^{-1}({\mathsf W}(\bflot^{\infty}(z,t+s)) - {\mathsf W}(\bflot^{\infty}(z,t)))\,.$
  Introduce $G(z)\eqdef -\ps{\nabla F(x),m}$ and $H(z)\eqdef \|q_\infty v - p_\infty S(x)\|^2$ for every $z=(v,m,x)\in \cZ_+$.
  Consider $\delta>0$ (to be specified later on).
  We study $\cL_{W_\delta} = \cL_{\mE_\infty} + \delta \cL_{G} + \delta \cL_{H}$.
  Note that $\bflot^{\infty}(z,t)\in \bs S\cap \cZ_+$ for all $t \geq t_0$ by an analogous result to Lem.~\ref{lemma:v>0} for $\flot^\infty$. Thus, $t\mapsto \mE_\infty(\bflot^{\infty}(z,t))$
  is differentiable at any point~$t \geq t_0$ and $\cL_{\mE_\infty}(t) = \frac{d}{dt} \mE_{\infty}(\bflot^{\infty}(z,t))$.
  Using similar derivations to Ineq.~\eqref{lyap-decrease},
  %Lem.~\ref{lemma:lyap},
  we obtain that
  \begin{equation}
  \label{eq:L_E_infty}
   \cL_{\mE_\infty}(t) \leq - \left( r_\infty - \frac{q_\infty}{4} \right)
      \left\| \frac{\sm(t)}{(\sv(t)+\varepsilon)^{\odot \frac 14}} \right\|^2 \,.
  \end{equation}
  We now study $\cL_G$. For every~$t \geq t_0$,
  \begin{align*}
    \cL_G(t)&=  \limsup_{s\to 0} s^{-1}(-\ps{\nabla F(\sx(t+s)),\sm(t+s)}+\ps{\nabla F(\sx(t)),\sm(t)}) \\
  &\leq  \limsup_{s\to 0} s^{-1}\|\nabla F(\sx(t))- \nabla F(\sx(t+s))\| \|\sm(t+s)\| - \ps{\nabla F(\sx(t)),\dot \sm(t)}\,.
  \end{align*}

  Let $L_{\nabla F}$ be the Lipschitz constant of $\nabla F$
  on the bounded set $\{x:(v,m,x)\in \bs S\}$. Define $C_1 \eqdef \sup_t \| \sqrt{\sv(t)+\varepsilon} \|$. Then,
  \begin{align}
    \cL_G(t)&\leq  L_{\nabla F}\limsup_{s\to 0} s^{-1}\|\sx(t)- \sx(t+s)\| \|\sm(t+s)\| - \ps{\nabla F(\sx(t)),\dot \sm(t)}\nonumber\\
  &\leq  L_{\nabla F}\|\dot \sx(t)\| \|\sm(t)\| - \ps{\nabla F(\sx(t)),\dot \sm(t)}\nonumber\\
  &\leq L_{\nabla F}\|\dot \sx(t)\| \|\sm(t)\| - h_\infty \|\nabla F(\sx(t))\|^2  + r_\infty \ps{\nabla F(\sx(t)),\sm(t)}\nonumber\\
  &\leq \left( \frac{L_{\nabla F} C_1^{\frac 12}}{\varepsilon^{\frac 14}} + \frac{r_\infty C_1}{2 u_1^2}\right) \left\| \frac{\sm(t)}{(\sv(t)+\varepsilon)^{\odot \frac 14}} \right\|^2
  - \left(h_\infty - \frac{r_\infty u_1^2}{2} \right)\| \nabla F(\sx(t)) \|^2 \label{eq:L_G}
  \end{align}
where we used the classical inequality $|\ps{a,b}| \leq \| a\|^2 / (2u^2) + u^2 \| b \|^2/2$ for any non-zero real $u$ to derive the last above inequality.
  We now study $\cL_H$. For every~$t \geq t_0$,
  \begin{align*}
    \cL_H(t)&=  \limsup_{s\to 0} s^{-1}(\|q_\infty \sv(t+s)- p_\infty S(\sx(t+s))\|^2-\|q_\infty \sv(t) - p_\infty S(\sx(t))\|^2) \\
  % &=  \limsup_{s\to 0} s^{-1}(\| q_\infty \sv(t+s) - p_\infty S(\sx(t)) + p_\infty(S(\sx(t))- S(\sx(t+s)))\|^2-\|q_\infty \sv(t) - p_\infty S(\sx(t))\|^2)\,.
  &= \limsup_{s\to 0} s^{-1}(p_\infty^2\|S(\sx(t)) - S(\sx(t+s))\|^2\\  &+ 2 p_\infty \ps{S(\sx(t)) - S(\sx(t+s)), q_\infty \sv(t+s) - p_\infty S(\sx(t))})\\
  &+ \lim_{s\to 0} s^{-1}(\|q_\infty \sv(t+s) - p_\infty S(\sx(t))\|^2
     -\|q_\infty \sv(t) - p_\infty S(\sx(t))\|^2) \,.
  \end{align*}
  %Expanding the square norm, we obtain:
  % \begin{multline*}
  %   \cL_H(t)= \limsup_{s\to 0} s^{-1}(p_\infty^2\|S(\sx(t)) - S(\sx(t+s))\|^2  + 2 p_\infty \ps{S(\sx(t)) - S(\sx(t+s)), q_\infty \sv(t+s) - p_\infty S(\sx(t))})\\ + \lim_{s\to 0} s^{-1}(\|q_\infty \sv(t+s) - p_\infty S(\sx(t))\|^2
  %   -\|q_\infty \sv(t) - p_\infty S(\sx(t))\|^2) \,.
  % \end{multline*}
  The second term in the righthand side coincides with $-2 q_\infty \ps{p_\infty S(\sx(t))-q_\infty \sv(t),\dot \sv(t)} = -2 q_\infty \|p_\infty S(\sx(t))-q_\infty \sv(t)\|^2$.
  Denote by $L_S$ the Lipschitz constant of $S$ on the set $\{x:(v,m,x)\in \bs S\}$.
  Note that $s^{-1}(\|S(\sx(t+s))-S(\sx(t))\|^2)\leq L_S^2 s\| s^{-1}(\sx(t+s)-\sx(t))\|^2$ which converges
  to zero as $s\to 0$. Thus,
  \begin{align}
    \cL_H(t)&= -2 q_\infty \|p_\infty S(\sx(t))-q_\infty \sv(t)\|^2\nonumber\\
  &+\limsup_{s\to 0} 2 p_\infty s^{-1} \ps{S(\sx(t)) - S(\sx(t+s)), q_\infty \sv(t+s) - p_\infty S(\sx(t))} \nonumber\\
  &\leq  -2 q_\infty \|p_\infty S(\sx(t))-q_\infty \sv(t)\|^2
  + 2 p_\infty  \|\dot \sx(t)\| L_S \|q_\infty \sv(t) - p_\infty S(x(t))\| \nonumber\\
  &\leq \frac{p_\infty}{\varepsilon^{\frac 12} u_2^2} \left\| \frac{\sm(t)}{(\sv(t)+\varepsilon)^{\odot \frac 14}} \right\|^2 - (2 q_\infty - p_\infty u_2^2 L_S^2) \|p_\infty S(\sx(t))-q_\infty \sv(t)\|^2\,. \label{eq:L_H}
   % -2 q_\infty \|p_\infty S(\sx(t))-q_\infty \sv(t)\|^2 +2L_S\varepsilon^{-1} \|m(t)\| \|S(x(t))-v(t)\|\,.
  \end{align}

  Recalling that $\cL_{W_\delta} = \cL_{\mE_\infty} + \delta \cL_{G} + \delta \cL_{H}$ and combining Eqs.~\eqref{eq:L_E_infty}, \eqref{eq:L_G} and~\eqref{eq:L_H}, we obtain for every $t \geq t_0$,
  \begin{multline}%$$
    \cL_{W_\delta}(t) \leq -M(\delta)
    \left\| \frac{\sm(t)}{(\sv(t)+\varepsilon)^{\odot \frac 14}} \right\|^2
    - \delta \left( h_\infty - \frac{r_\infty u_1^2}{2}\right)
       \| \nabla F(\sx(t)) \|^2\\
    - \delta \left( 2 q_\infty - p_\infty u_2^2 L_S^2 \right)
        \| p_\infty S(\sx(t)) - q_\infty \sv(t) \|^2\,.
  \end{multline}%$$
  where $M(\delta)\eqdef r_\infty - \frac{q_\infty}{4}
  - \delta\left( \frac{r_\infty C_1}{2 u_1^2}
     +  \frac{L_{\nabla F} C_1^{\frac 12}}{\varepsilon^{\frac 14}}
     +  \frac{p_\infty}{\varepsilon^{\frac 12} u_2^2} \right)\,.$
  Now select $u_1$, $u_2$ small enough s.t. $h_\infty - r_\infty u_1^2/2 > 0$ and $2 q_\infty - p_\infty u_2^2 L_S^2 > 0$. Then, choose
   $\delta$ in such a way that $M(\delta)>0$. Thus, there exists a constant $c$ depending on $\delta$ s.t.
  \begin{equation}
    \label{eq:strict_lyap}
  \forall t \geq t_0,\ \  \cL_{W_\delta}(t) \leq -c\left(
  \left\| \frac{\sm(t)}{(\sv(t)+\varepsilon)^{\odot \frac 14}} \right\|^2
  + \|\nabla  F(\sx(t))\|^2
  + \|p_\infty S(\sx(t)) - q_\infty \sv(t)\|^2\right)\,.
  \end{equation}

  It can easily be seen that for every $z\in \bs S$,  $t\mapsto W_\delta(\bflot^{\infty}(z,t))$ is Lipschitz continuous, hence absolutely continuous.
  Its derivative almost everywhere coincides with $\cL_{W_\delta}$, which is nonpositive.
  Thus, $W_\delta$ is a Lyapunov function for $\bflot^{\infty}$.
  We prove that the Lyapunov function is strict.
  Consider $z = (v,m,x) \in \bs S$ s.t. $W_\delta(\bflot^{\infty}(z,t))=W_\delta(z)$ for all $t \geq t_0$.
  The derivative almost everywhere of  $t\mapsto W_\delta(\bflot^{\infty}(z,t))$ is identically zero,
  and by Eq.~\eqref{eq:strict_lyap}, this implies that
   $$-c\left(\left\|  \frac{\sm(t)}{(\sv(t)+\varepsilon)^{\odot \frac 14}}\right\|^2
   + \|\nabla  F(\sx(t))\|^2
   + \|p_\infty S(\sx(t)) - q_\infty \sv(t)\|^2\right)$$
    is equal to zero
  for every $t \geq t_0$ a.e. (hence, for every~$t \geq t_0$, by continuity of~$\bflot^{\infty}$).
  In particular for $t=t_0$, $m=\nabla F(x)=0$ and $p_\infty S(x)- q_\infty v=0$. Hence, $z\in \zer g_\infty \cap \bs S$. This concludes the proof since $\Lambda_{\Phi^\infty} = \zer g_\infty$.
\end{proof}

\noindent\textbf{End of the Proof of Th.~\ref{th:as_conv_under_stab}.}
Finally, Assumption~\ref{hyp:sard} implies that $W_\delta(\Lambda_{\Phi^\infty} \cap \bs S)$ is of empty interior.
Recall that Assumptions~\ref{hyp:F_loclip} and~\ref{hyp:S_loclip} both follow from Assumption~\ref{hyp:model_bis} made in Th.~\ref{th:as_conv_under_stab}.
Given Prop.~\ref{prop:Wstrict}, the proof is concluded by applying \cite[Prop.~6.4]{ben-(cours)99} to the restricted semiflow $\bar{\Phi}^\infty$ (with $(M,\Lambda) = (\bs S, \Lambda_{\bar{\Phi}^\infty})$). Note that a Lyapunov function for $\Lambda_{\bar{\Phi}^\infty}$ is what is called a strict Lyapunov function. Such a function is provided by Prop.~\ref{prop:Wstrict}. We obtain as a conclusion of \cite[Prop.~6.4]{ben-(cours)99} that $\bs S \subset \Lambda_{\bar{\Phi}^\infty}$. This gives the desired result (Eq.~\eqref{Seq}) given Prop.~\ref{prop:Wstrict}-\ref{eq_restricted}.

The last assertion of Th.~\ref{th:as_conv_under_stab} is a consequence of \cite[Cor.~6.6]{ben-(cours)99}.

\subsection{Proof of Th.~\ref{th:as_conv_under_stab_nesterov}}
\label{sec:proof_simp_as}

We can rewrite the iterates from Algorithm~\ref{algo-nesterov} as follows:
\begin{equation}\label{eq:nest_alg}
  \begin{cases}
    m_{n+1} &= m_n + \gamma_{n+1}(\nabla F(x_{n}) - \frac{\alpha}{\tau_n}\, m_n) + \gamma_{n+1} (\nabla f(x_n,\xi_{n+1}) - \nabla F(x_n)) \\
    x_{n+1} &= x_n - \gamma_{n+1} m_{n+1}\,.
  \end{cases}
\end{equation}

% As explained in Section~\ref{sec:proof_simp} this ODE doesn't admit a Lyapunov function, nevertheless we can analyze~\eqref{eq:nest_alg} in the same way as in Section~\ref{sec:proof_simp}.
We prove that the sequence $( y_n =(m_n,x_n): n \in \NN)$ of iterates of this algorithm converges almost surely towards the set $\bar\Upsilon$ defined in Eq.~\eqref{eq:bar_Upsilon} if it is supposed to be bounded with probability one.
The proof follows a similar path to the proof in Section~\ref{sec:proof_simp}.

Indeed, denote by $\sX$ and $\sM$ the linearly interpolated processes constructed respectively from the sequences $(x_n)$ and $(m_n)$ and let $\sss(t) = 1/t$.
Recall that $\Phi_N = (\Phi_N^m, \Phi_N^x, \Phi_N^s)$ is the semiflow induced by~\eqref{eq:odeNest_aut}. As in Section~\ref{proof:th_as_conv}, we have that $\sZ \eqdef (\sM, \sX, \sss)$ is an APT of~\eqref{eq:odeNest_aut}. In particular, this means that
\begin{equation}\label{eq:APT_Nest}
  \forall T >0\,, \quad \sup_{h \in [0, T]} \norm{ \sX(t + h) - \Phi_N^x(\sZ(t), h)} \xrightarrow[t\to\infty]{} 0 \, .
\end{equation}

By Lem.~\ref{lm:ch_var}, we also have that
\begin{multline}\label{eq:Apt_Gad}
      \sup_{h \in [0, T^2/\kappa^2]} \norm{ \sX(t + \kappa\sqrt{h}) - \Phi_N^x(\sZ(t), \kappa \sqrt{h})} \\
      = \sup_{h \in [0, T^2/\kappa^2]} \norm{ \sX(t + \kappa\sqrt{h}) - \Phi_H^u(\sZ(t), h)} \xrightarrow[t\to\infty]{} 0\,.
\end{multline}

Let $(m,x)$ be a limit point of the sequence $(y_n)$ and let $T > 0$. Using Lem.~\ref{lm:rel_comp}, we can proceed in the same manner as in Section~\ref{sec:proof_simp} and get a sequence $(t_k)$ such that
$$(\sM(t_k + \cdot), \sX(t_k + \cdot))\rightarrow (\sm, \sx)\,\, \text{and}\,\, (\Phi_H^y(\sZ(t_k), \cdot), \Phi_H^u(\sZ(t_k), \cdot)) \rightarrow (\sy, \su)\,,$$
where $(\sm(0), \sx(0)) = (m,x)$\,, and $(\sm, \sx)$ and $(\sx, \su)$ are respectively solutions to~\eqref{eq:lim_nest} and~\eqref{eq:lim_gad}. As in the end of Section~\ref{sec:proof_simp}, we obtain that $\su$ and $\sx$ are constant, therefore $\sm \equiv 0$ and $\nabla F(\sx) \equiv 0\,,$ which finishes the proof.

\subsection{Proof of Th.~\ref{th:stab}}
\label{sec:stability}

The idea of the proof is to apply Robbins-Siegmund's theorem~\cite{robbins1971convergence} to
 $$
 V_n = h_{n-1}F(x_n) + \frac{1}{2}\ps{m_n^{\odot 2},\frac 1{\sqrt{v_n + \varepsilon}}}
 $$
 (note the similarity of $V_n$ with the energy function~\eqref{eq:lyap}).
Since $\inf F>-\infty$, we assume without loss of generality that $F\geq 0$.
In this subsection, we use the notation $\nabla f_{n+1}$ as a shorthand notation for $\nabla f(x_n,\xi_{n+1})$ and $C$ denotes some positive constant which may change from line to line. We write $\EE_n = \EE[ \cdot \, | \, \mcF_n ]$ for the conditional expectation w.r.t the $\sigma$-algebra~$\mcF_n$.
Define
$P_n \eqdef \frac{1}{2}\ps{D_n,m_n^{\odot 2}}$, with~$D_n \eqdef \frac{1}{\sqrt{v_{n}+\varepsilon}}$.
We have the decomposition:
\begin{equation}
P_{n+1}-P_n=\frac{1}{2}\ps{D_{n+1}-D_n,m_{n+1}^{\odot 2}}+\frac {1}{2}\ps{D_n,m_{n+1}^{\odot 2}-m_n^{\odot 2}}.\label{eq:P-P}
\end{equation}
We estimate the vector
$$
D_{n+1}-D_n = \frac{\sqrt{v_n + \varepsilon} - \sqrt{ v_{n+1}+\varepsilon}}{\sqrt{v_{n+1}+\varepsilon}\odot \sqrt{v_n + \varepsilon}}\,.
$$
Remarking that $v_{n+1} \geq (1-\gamma_{n+1}q_n) v_n$ and using the update rule of $v_n$, we obtain for a sufficiently large $n$ that%after some algebra
\begin{align}
  \label{eq:subsubterm2}
\sqrt{v_n + \varepsilon} - \sqrt{ v_{n+1}+\varepsilon}
      &= \gamma_{n+1} \frac{q_n v_n -p_n \nabla f_{n+1}^{\odot 2}}{ \sqrt{v_n + \varepsilon} +  \sqrt{v_{n+1} + \varepsilon}}\nonumber\\
      &\leq \gamma_{n+1}q_n \frac{v_n}{(1 + \sqrt{1-\gamma_{n+1}q_n})\sqrt{v_n + \varepsilon}}\nonumber\\
      &= \frac{\gamma_{n+1}q_n}{1 + \sqrt{1-\gamma_{n+1}q_n}}\sqrt{v_n}\odot \frac{\sqrt{v_n}}{\sqrt{v_n + \varepsilon}}\nonumber\\
      &\leq c_{n+1} \sqrt{v_{n+1}}  \,\, \text{where} \,\,
      c_{n+1} \eqdef \frac{\gamma_{n+1}q_n}
      {\sqrt{1-\gamma_{n+1}q_n}(1 + \sqrt{1-\gamma_{n+1}q_n})}\,.
\end{align}
It is easy to see that $c_{n+1}/\gamma_n\to q_\infty/2$. Thus, for any $\delta>0$, $c_{n+1}\leq (q_\infty+2\delta)\gamma_{n}/2$ for all $n$ large enough.
Using also that $\sqrt{v_{n+1}}/\sqrt{v_{n+1}+\varepsilon}\leq 1$, we obtain
\begin{equation}\label{eq:dif_D_n}
  D_{n+1}-D_n \leq \frac{q_\infty+2\delta}2 \gamma_n D_n\,.
\end{equation}

Substituting the above inequality in Eq.~(\ref{eq:P-P}), we obtain
\begin{align*}
  P_{n+1}-P_n&\leq  \left(\frac{q_\infty+2\delta}{2}\right)\frac{ \gamma_n}{2}\ps{D_n,m_{n+1}^{\odot 2}}+\frac{1}{2}\ps{D_n,m_{n+1}^{\odot 2}-m_n^{\odot 2}} \\
&\leq \frac{q_\infty+2\delta}{2}\gamma_n P_n
+\left(1+\frac{q_\infty+2\delta}2\gamma_n\right)\frac{1}{2}\ps{D_n,m_{n+1}^{\odot 2}-m_n^{\odot 2}} \,.
\end{align*}
Using $m_{n+1}^{\odot 2} - m_n^{\odot 2} =  2 m_n
\odot (m_{n+1}-m_n)+(m_{n+1}-m_n)^{\odot 2} $, and noting that $\EE_n(m_{n+1}-m_n) = \gamma_{n+1}h_n\nabla F(x_n) - \gamma_{n+1}r_n m_n$,
\begin{multline*}
  \EE_n \frac{1}{2}\ps{D_n,m_{n+1}^{\odot 2}- m_n^{\odot 2}}
= \gamma_{n+1}h_n\ps{\nabla F(x_n),\frac{m_n}{\sqrt{v_n + \varepsilon}}}-2\gamma_{n+1}r_nP_n\\
+\frac{1}{2}\ps{D_n, \EE_n[(m_{n+1}-m_n)^{\odot 2}]}\,.
\end{multline*}
There exists $\delta >0$ such that $r_\infty - \frac{q_\infty}{4} - \frac{\delta}{2}>0$ by Assumption~\ref{hyp:hrpq}-\ref{hyp:stabode}.
As $\frac{\gamma_{n+1}}{\gamma_n} r_n - \frac{q_\infty}{4} \to r_\infty - \frac{q_\infty}{4}$, for all $n$ large enough, $\frac{\gamma_{n+1}}{\gamma_n} r_n - \frac{q_\infty}{4} > r_\infty - \frac{q_\infty}{4} - \frac{\delta}{2}>0$. Hence, for all $n$ large enough,
\begin{align}
  \label{eq:intermediate}
 \EE_n P_{n+1}-P_n&\leq
 -2\left(r_\infty-\frac{q_\infty}{4} -\frac{\delta}{2}\right)\gamma_n P_n + \gamma_{n+1}h_n\ps{\nabla F(x_n),\frac{m_n}{\sqrt{v_n + \varepsilon}}}
\nonumber\\
&+
C\gamma_n^2\ps{\nabla F(x_n),\frac{ m_n}{\sqrt{v_n + \varepsilon}}}
+C\ps{D_n, \EE_n[(m_{n+1}-m_n)^{\odot 2}] }\,.
\end{align}
Using the inequality $\ps{u,v} \leq (\|u\|^2+\|v\|^2)/2$ and Assumption~\ref{hyp:stability}-\ref{momentgrowth}, %Eq.~(\ref{eq:grows}),
it is easy to show the inequality
$\ps{\nabla F(x_n),\frac{m_n}{\sqrt{v_n + \varepsilon}}}\leq C(1+F(x_n)+P_n)$.
Moreover, using the componentwise inequality $(h_n\nabla f_{n+1}- r_nm_n)^{\odot 2} \leq 2h_n^2 \nabla f_{n+1}^{\odot 2} + 2 r_n^2 m_n^{\odot 2}$
along with Assumption~\ref{hyp:stability}-\ref{momentgrowth} and the boundedness of the sequences $(h_n), (r_n)$ and $(\gamma_{n+1}/\gamma_n)$, we obtain
\begin{equation}
  \label{eq:intermediate_2}
\ps{D_n, \EE_n[(m_{n+1}-m_n)^{\odot 2}] }
\leq C\gamma_n^2(1+F(x_n)+P_n)\,.
\end{equation}
Combining Eq.~\eqref{eq:intermediate} and Eq.~\eqref{eq:intermediate_2}, we get
\begin{multline}\label{eq:intermediate_3}
  \EE_n(P_{n+1} - P_n) \leq\gamma_{n+1}h_n\ps{\nabla F(x_n),m_n \odot D_n}
 + C\gamma_n^2(1+F(x_n)+P_n) \, .
\end{multline}

Denoting by $M$ the Lipschitz coefficient of $\nabla F$, we also have
\begin{align}
F(x_{n+1})  &\leq F(x_{n}) - \gamma_{n+1} \ps{\nabla F(x_{n}),m_{n+1}\odot D_{n+1}} + \frac{\gamma_{n+1}^2M}{2}\left\|m_{n+1}\odot D_{n+1}\right\|^2 \, .
\label{eq:lip}
\end{align}

Using~\eqref{eq:dif_D_n} and the update rule of $m_n$, we have\\
%pour éviter que l'équation déborde
$
\norm{m_{n+1} \odot D_{n+1}  - m_n \odot D_n}^2
$
\begin{align}\label{eq:dif_mnDn}
  \begin{split}
    %\norm{m_{n+1} \odot D_{n+1}  - m_n \odot D_n}^2
      &\leq C\norm{(m_{n+1} - m_n) \odot D_n}^2  +  C\norm{m_{n+1}\odot (D_{n+1} - D_n)}^2 \\
      &\leq C \gamma^2_{n+1}( \norm{\nabla f_{n+1}}^2 + \norm{m_n \odot D_n}^2)  + C\gamma^2_{n+1} \norm{m_{n+1}\odot D_n}^2\\
      &\leq C \gamma^2_{n+1} ( \norm{m_n \odot D_n }^2 + \norm{\nabla f_{n+1}}^2) \, .
  \end{split}
\end{align}

Finally, recalling that $V_{n} = h_{n-1} F(x_n) + P_n$, $(h_n)$ is decreasing, combining Eq.~\eqref{eq:intermediate_3},\eqref{eq:lip},\eqref{eq:dif_mnDn},
and using Assumption~\ref{hyp:stability},
 we have
\begin{equation*}\label{eq:Rob_sig}
  \begin{split}
      \EE_n[V_{n+1}] &\leq V_n +  \gamma_{n+1} h_n  \scalarp{\nabla F(x_{n})} {\EE_n \left[m_{n} \odot D_n - m_{n+1} \odot D_{n+1}\right]} \\
      &+ C \gamma_{n+1}^2 \left( 1 + F(x_n) + P_n + \norm{m_n \odot D_n}^2 \right)\\
      &+ C \gamma_{n+1}^2 \EE_n [\norm{m_{n} \odot D_n - m_{n+1} \odot D_{n+1}}^2] \\
      &\leq V_n + C\gamma_n^2 \left( 1 + F(x_n) + P_n + \norm{m_n \odot D_n}^2 + \EE_n\left[ \norm{\nabla f_{n+1}}^2\right]\right)\\
      &\leq V_n + C \gamma_n^2( 1 + F(x_n) + P_n)\\
      &\leq (1 + C\gamma_n^2 )V_n + C \gamma_n^2\, ,
  \end{split}
\end{equation*}
where we used Cauchy-Schwarz's inequality and the fact that $\norm{m_n \odot D_n}^2 \leq C P_n$.
By the Robbins-Siegmund's theorem  \cite{robbins1971convergence},
the sequence $(V_n)$ converges almost surely to a finite random variable $V_\infty \in \bR^+$.
Then, the coercivity of $F$ implies that $(x_n)$ is almost surely bounded.

We now establish the almost sure boundedness of $(m_n)$. Assume in the sequel that $n$ is large enough to have $(1 - \gamma_{n+1} r_n) \geq 0$.
Consider the martingale difference sequence $\Delta_{n+1}\eqdef \nabla f_{n+1} -\nabla F(x_n)$.
We decompose $m_{n} = \bar m_n + \tilde m_n$ where
$\bar m_{n+1} = (1-\gamma_{n+1}r_n) \bar m_n + \gamma_{n+1} h_n \nabla F(x_n)$ and
$\tilde m_{n+1} = (1-\gamma_{n+1}r_n) \tilde m_n + \gamma_{n+1} h_n \Delta_{n+1}$, setting $\bar m_0= 0$ and $\tilde m_0= m_0$.
We prove that both terms $\bar m_n$ and $\tilde m_n$ are bounded. Consider the first term:
$
\|\bar m_{n+1}\|\leq (1-\gamma_{n+1}r_n) \|\bar m_n\| + \gamma_{n+1} \sup_k\|h_k\nabla F(x_k)\| \,,
$
where the supremum in the above inequality is almost surely finite by continuity of $\nabla F$. We immediately get that if $\norm{\bar m_n} \geq \frac{\sup_k\|h_k\nabla F(x_k)\|}{r_{\infty}}$, then $\norm{\bar m_{n+1}} \leq \|\bar m_n\|$. Thus

\begin{equation*}
  \norm{\bar m_{n+1}} \leq  \frac{\sup_k\|h_k\nabla F(x_k)\|}{r_{\infty}} + \sup_k \gamma_{k+1} \|h_k\nabla F(x_k)\| \, ,
\end{equation*} %Set $o_n \eqdef \|\bar m_n\|/ \sup_k\|\nabla F(x_k)\|$.
which implies that $\bar m_{n}$ is bounded.

Consider now the term $\tilde m_n$:
\begin{equation*}
    \resizebox{\hsize}{!}{$
  \EE_n[\|\tilde m_{n+1}\|^2]
  = (1-\gamma_{n+1}r_n)^2\|\tilde m_n\|^2
  + \gamma_{n+1}^2h_n^2\EE_n[\|\Delta_{n+1}\|^2]
  \leq  \|\tilde m_n\|^2
  %+ \gamma_{n+1}^2h_n^2C\,,
  + \gamma_{n+1}^2h_n^2\EE_n[\|\Delta_{n+1}\|^2]\,.
    $}
\end{equation*}
Then, the inequality $\EE_n[\|\Delta_{n+1}\|^2] \leq \EE_n[\|\nabla f_{n+1}\|^2]$ combined with Assumption~\ref{hyp:moment}-\ref{hyp:moment-q} and the a.s. boundedness of the sequence $(x_n)$ imply that there exists a finite random variable $C_{\mathcal{K}}$ (independent of~$n$) s.t. $\EE_n[\|\nabla f_{n+1}\|^2] \leq C_{\mathcal{K}}$.
As a consequence, since $\sum_n \gamma_{n+1}^2<\infty$ and the sequence~$(h_n)$ is bounded, we obtain that a.s.:
\begin{equation*}
\sum_{n\geq 0} \gamma_{n+1}^2 h_n^2 \EE_n[\|\Delta_{n+1}\|^2] \leq C C_{\mathcal{K}} \sum_{n\geq 0} \gamma_{n+1}^2 < +\infty\,.
\end{equation*}
Hence, we can apply the Robbins-Siegmund theorem to obtain that $\sup_n\|\tilde m_n\|^2<\infty$ w.p.1. Finally, it can be shown that $(v_n)$ is almost surely bounded using the same arguments, decomposing $v_n$ into $\bar{v}_n + \tilde{v}_n$ as above. Indeed, first, we have:
\begin{equation*}
  \EE_n[\|\tilde v_{n+1}\|^2]
  \leq  \|\tilde v_n\|^2
  + \gamma_{n+1}^2p_n^2\EE_n[\|\nabla f_{n+1}^{\odot 2} - S(x_n)\|^2]\,.
\end{equation*}
Second, it also holds that:
\begin{equation*}
      \EE_n[\|\nabla f_{n+1}^{\odot 2} - S(x_n)\|^2]
      \leq \EE_n[\|\nabla f_{n+1}^{\odot 2}\|^2]
      \leq \EE_n[\|\nabla f_{n+1}\|^4]\,.
\end{equation*}
Then, using Assumption~\ref{hyp:moment}-\ref{hyp:moment-q} and the a.s. boundedness of the sequence $(x_n)$, there exists a finite random variable $C'_{\mathcal{K}}$ (independent of~$n$) s.t. $\EE_n[\|\nabla f_{n+1}\|^4] \leq C'_{\mathcal{K}}$.
Moreover, the sequence~$(p_n)$ is bounded and $\sum_n \gamma_{n+1}^2<\infty$. As a consequence, it holds that a.s:
$$
\sum_{n \geq 0} \gamma_{n+1}^2 p_n^2 \EE_n[\|\nabla f_{n+1}^{\odot 2} - S(x_n)\|^2] \leq C C'_{\mathcal{K}} \sum_{n \geq 0} \gamma_{n+1}^2  < + \infty\,.
$$
It follows that the Robbins-Siegmund theorem can be applied to the sequence~$\|\tilde{v}_n\|^2$ as for the sequence $\|\tilde{m}_n\|^2$ to obtain that $\sup_n\|\tilde v_n\|^2<\infty$ w.p.1.

\subsection{Proof of Th.~\ref{th:stab_nesterov}}

The proof of Th.~\ref{th:stab} easily adapts to Algorithm~\ref{algo-nesterov} by replacing~$V_n$ by $$\tilde{V}_n \eqdef F(x_n) + \frac{1}{2} \norm{m_n}^2\,.$$
 The boundedness of $(m_n)$ is an immediate consequence of the convergence of~$\tilde V_n$.
% and \eqref{algo-adagrad} by replacing~$V_n$ respectively by $F(x_n) + \frac{1}{2} \norm{m_n}^2$ and $F(x_n)$. Note that in the Nesterov-like case where $r_{\infty} = 0$, the boundedness of $(m_n)$ is an immediate consequence of the convergence of~$V_n$.

\subsection{Proof of Th.~\ref{th:clt}}

We shall use the following result.
\begin{theorem}[adapted from \cite{pelletier1998weak}, Th.~7]
\label{th:pelletier}
Let $k\geq 1$. On some probability space equipped with a filtration ${\mcF}=(\mcF_n)_{n\in\NN}$,
 %({\mathcal F}_n)_n
consider a sequence of r.v. on $\RR^k$ given by
$$
Z_{n+1} = (I+\gamma_{n+1} \bar H)Z_n+\gamma_{n+1} b_{n+1}+\sqrt{\gamma_{n+1}}\eta_{n+1}
$$
and $\EE[\|Z_0\|^2]<\infty$,
where $\bar H$ is a $k\times k$ Hurwitz matrix,
$(b_n)$ and $(\eta_n)$ are random sequences, and $\gamma_n=\gamma_0n^{-\alpha}$ for some $\gamma_0>0$ and $\alpha\in (0,1]$.
Let $\Omega_0\in {\mcF}_\infty$ have a positive probability. Assume that the following holds almost surely on~$\Omega_0$:
\begin{enumerate}[{i)}, leftmargin=*]
\item $\EE[\eta_{n+1}|{\mcF}_n]=0$.
\item There exists a constant $\bar b>2$ s.t. $\sup_{n\geq 0}\EE[\|\eta_{n+1}\|^{\bar b}|{\mcF}_n]<\infty$.
\item $\EE[\eta_{n+1}\eta_{n+1}^\T|{\mcF}_n] = \Sigma + \Delta_n$ where $\EE[\|\Delta_n\|\1_{\Omega_0}]\to 0$
and $\Sigma$ is a positive semidefinite matrix.
\item The sequence $(b_n)$ is the sum of two sequences $(b_{n,1})$ and $(b_{n,2})$, adapted to ${\mcF}$,
s.t. $\sup_{n\geq 0}\EE[\|b_{n,1}\|^2]<\infty$, $\EE[\|b_{n,1}\|\1_{\Omega_0}]\to 0$ and $b_{n,2}\to 0$ a.s. on $\Omega_0$.
\end{enumerate}
Then, given $\Omega_0$, $(Z_n)$ converges in distribution to the unique stationary distribution $\mu_{\star}$ of the generalized Ornstein-Uhlenbeck process
$$
dX_t = \bar H X_t dt + \sqrt \Sigma dB_t
$$
where $(B_t)$ is the standard Brownian motion and $\sqrt{\Sigma}$ is the unique positive semidefinite square root of $\Sigma$.
The distribution $\mu_{\star}$ is the zero mean Gaussian distribution with covariance matrix $\Gamma$
given as the solution to $(\bar H + \frac {\1_{\alpha=1}}{2\gamma_0}I_{k})\Gamma+\Gamma (\bar H + \frac {\1_{\alpha=1}}{2\gamma_0}I_{k})^\T = -\Sigma$.
\end{theorem}
\begin{proof}

The proof is identical to the proof of \cite[Th.~7]{pelletier1998weak}, only substituting the
inverse of the square root of $\Sigma$ by the Moore-Penrose inverse. Finally, the uniqueness of the
stationary distribution $\mu_{\star}$ and its expression follow from \cite[Th.~6.7, p. 357]{KarShr91}
\end{proof}

%Denote by $D: \RR_+^d \to \RR_+^d$ the application defined by $D(v) = $
We define $ v_n = \bar v_n+\delta_n$ where $\delta_0=0$, $\bar v_0 = v_0$ and
\begin{align*}
\delta_{n+1} &= (1 - \gamma_{n+1} q_n ) \delta_n
  + \gamma_{n+1} (p_{n}-q_n q_\infty^{-1} p_\infty) S(x_n) \\
\bar v_{n+1} &= (1 - \gamma_{n+1} q_n ) \bar v_n
+ \gamma_{n+1} q_n q_\infty^{-1} p_\infty S(x_{n})
+ \gamma_{n+1} p_{n}(\nabla f(x_n, \xi_{n+1})^{\odot 2}- S(x_n)) \,.
\end{align*}
% We define
% $$
% \eta_{n+1} =
% \begin{bmatrix}
%   p_{n}(\partial_1 f(x_n, \xi_{n+1})^{\odot 2}- S(x_n)) \\
%   h_n (\partial_1 f(x_n,\xi_{n+1})-\nabla F(x_n)) \\
% 0
% \end{bmatrix}\,.
% $$
For every $z=(v,m,x) \in \cZ_+$ and $\delta \geq 0$, we define
$$
r_n(z,\delta) \eqdef
\begin{bmatrix}
  q_n q_\infty^{-1} p_\infty (S(x-\gamma_{n} \frac{m}{\sqrt{v+\delta + \varepsilon}})-S(x)) \\
h_n (\nabla F(x-\gamma_{n}\frac{m}{\sqrt{v+\delta + \varepsilon}}) - \nabla F(x))\\
\frac{\gamma_n}{\gamma_{n+1}} (\frac{1}{\sqrt{v + \varepsilon}}-\frac{1}{\sqrt{v+\delta + \varepsilon}})\odot m
\end{bmatrix}\,.
$$
Moreover, for every $z=(v,m,x) \in \cZ_+$ and every $n\in \NN$, we set
$$
g_n(z) =
\begin{bmatrix}
  q_n q_\infty^{-1} p_\infty S(x)  - q_n  v  \\
h_n \nabla F(x)- r_n m \\
-\frac{\gamma_n}{\gamma_{n+1}} \frac{m}{\sqrt{v + \varepsilon}}
\end{bmatrix}\,.
$$
Defining $\zeta_n = (\bar v_n,m_n,x_{n-1})$ and recalling the definition of $(\eta_n)$ from Eq.~\eqref{eq:noise_eta},
we have the decomposition
$$
\zeta_{n+1} = \zeta_n+\gamma_{n+1}g_n(\zeta_n) + \gamma_{n+1}\eta_{n+1}+\gamma_{n+1}r_n(\zeta_n,\delta_n)\,.
$$
Define $z_{\star}\eqdef (x_{\star},0,v_{\star})$. Note that  $g_n(z_{\star})=0$. Evaluating the Jacobian matrix $G_n$ of
$g_n$ at $z_{\star}$, we obtain that there exist constants $C>0$, $\bar M>0$ and $n_0\in \NN$ s.t. for all $n\geq n_0$,
\begin{equation}
  \label{eq:1}
  \|g_n(z) - G_n(z-z_{\star})\|\leq C\|z-z_{\star}\|^2\quad (\forall z\in B(z_{\star},\bar M))\,,
\end{equation}
where $G_n$ is given by
$$
G_n\eqdef
\begin{bmatrix}
  -q_n I_d & 0 & q_n q_\infty^{-1} p_\infty \nabla S(x_{\star}) \\
0 & -r_n I_d & h_n \nabla^2 F(x_{\star}) \\
0 & -\frac{\gamma_n}{\gamma_{n+1}}V & 0
\end{bmatrix}\,,
$$
where $\nabla S$ is the Jacobian of $S$ and the matrix $V$ is defined in Eq.~\eqref{eq:V_matrix}. We define
$$
G_\infty \eqdef \lim_nG_n=
\begin{bmatrix}
  -q_\infty I_d & 0 & p_\infty \nabla S(x_{\star}) \\
0 & -r_\infty I_d & h_\infty \nabla^2 F(x_{\star}) \\
0 & -V & 0
\end{bmatrix}\,.
$$
One can verify that $\G$ is Hurwitz, and that the largest real part of its eigenvalues
is $-L'$, where $L'\eqdef L\wedge q_\infty$ and $L$ is defined in Eq.~\eqref{eq:L}.

We define $\Omega^{(0)}\eqdef \{ z_n\to
z_{\star}\}$. We assume $\PP(\Omega^{(0)})>0$.  Using for instance \cite[Lem.~4 and Lem.~5]{delyon1999convergence}, it holds that $\delta_n(\omega)\to 0$ for
every $\omega\in\Omega^{(0)}$, and since $x_n(\omega)-x_{n-1}(\omega)\to 0$
on that set, we obtain that $\Omega^{(0)} = \{\zeta_n\to z_{\star}\}$.  Let $M\in (0,\bar M)$
be a constant, whose value will be specified later on.  For every
$N_0 \in\NN$, define $\Omega^{(0)}_{N_0}\eqdef \{\zeta_n\to z_{\star}\text{ and
}\sup_{n\geq N_0}\|\zeta_n-z_{\star}\|\leq M\}$.  We seek to show that
$\sqrt{\gamma_n}^{-1}(\zeta_n-z_{\star})\Rightarrow \nu$ given $\Omega^{(0)}$, for
some Gaussian measure $\nu$, using Th.~\ref{th:pelletier}.
As $\Omega^{(0)}_{N_0}\uparrow \Omega^{(0)}$, it is
sufficient to show that the latter convergence holds given $\Omega^{(0)}_{N_0}$,
for every $N_0$ large enough. From now on, we consider that $N_0$ is fixed. We define the sequence
$(\tilde \zeta_n)_{n\geq N_0}$ as $\tilde \zeta_{N_0}=\zeta_{N_0}$ and for every $n\geq N_0$,
\begin{align*}
\tilde \zeta_{n+1} &= \tilde \zeta_n + \gamma_{n+1} \tilde g_n(\tilde\zeta_n) + \gamma_{n+1}(\eta_{n+1}+r_n(\tilde\zeta_n,\delta_n))\1_{{\mathcal A}_n}
\end{align*}
where ${\mathcal A}_n$ is the event defined by
$$
{\mathcal A}_n\eqdef \bigcap_{k=N_0}^n\{\|x_k-x_{\star}\|\leq M\}\cap\{\|\tilde\zeta_n-z_{\star}\|\leq M\}
$$ and
$$
\tilde g_n(z) \eqdef g_n(z)\1_{\|z-z_{\star}\|\leq M} - K (z-z_{\star})\1_{\|z-z_{\star}\|> M}\,,
$$
where $K>0$ is a large constant which will be specified later on.
The sequences $(\tilde \zeta_n)_{n\geq N_0}$ and $(\zeta_n)_{n\geq N_0}$ coincide on $\Omega^{(0)}_{N_0}$.
Thus, it is sufficient to study the weak convergence of $(\tilde \zeta_n)_{n\geq N_0}$.
\medskip

\noindent {\bf An estimate of $\|r_n(\tilde\zeta_n,\delta_n)\|\1_{{\mathcal A}_n}$.}
We start by studying the sequence $(\|\delta_n\|\1_{{\mathcal A}_n})$.
Unfolding the update rule defining $\delta_n$ and using the fact that $(q_n)$ is a sequence of positive reals converging to $q_\infty > 0$, we obtain that
\begin{align*}
  \|\delta_n\|\1_{{\mathcal A}_n} &\leq
  \sum_{k=1}^n \left[\prod_{j=k+1}^n |1 - \gamma_{j} q_{j-1}|\right] \gamma_k |p_{k-1}-q_{k-1}q_\infty^{-1} p_\infty| \|S(x_{k-1})\|\1_{{\mathcal A}_n} \\
&\leq C \sum_{k=1}^n \exp\left(-\beta \sum_{j=k+1}^n\gamma_j\right) \gamma_k |p_{k-1}-q_{k-1}q_\infty^{-1} p_\infty| \eqdef w_n \,,
\end{align*}
for some $\beta>0$.
%{\color{red} Bon, il faut faire commencer la somme à $N$ ou $N+1$, mais ça % %marche pareil.}
The sequence $(w_n)$ is deterministic and converges to zero by \cite[Lem.~4]{delyon1999convergence}.
There exists $n_1\geq n_0$ s.t. $w_n\leq M$. As $v \mapsto \frac{1}{\sqrt{v + \varepsilon}}$ is Lipschitz and $\nabla F$ and $S$ are locally Lipschitz, for every $z=(v,m,x)$ and $\delta$
s.t. $\|z-z_{\star}\|\leq M$ and $\|\delta\|\leq M$,  we have
\begin{align*}
\|r_n(z,\delta)\| &\leq C\gamma_{n+1}\|(v+\delta+ \varepsilon)^{\odot -\frac 12}\|\|m\| + C\|(v+\delta+ \varepsilon)^{\odot -\frac 12}-(v+\varepsilon)^{\odot -\frac 12}\|\|m\| \\
& \leq C\gamma_{n+1}\|z-z_{\star}\|+C \|\delta\|\|z-z_{\star}\|\,.
\end{align*}
This implies that for every $n\geq n_1$,
\begin{equation}
  \label{eq:rdelta}
  \|r_n(\tilde\zeta_n,\delta_n)\|\1_{{\mathcal A}_n} \leq C(\gamma_{n+1}+w_n)\|\tilde\zeta_n-z_{\star}\|\,.
\end{equation}

\medskip
\noindent {\bf Tightness of $\sqrt{\gamma_n}^{-1}(\tilde\zeta_n-z_{\star})$.}
%Define $\tilde r_n\eqdef r_n(\tilde\zeta_n,\delta_n))$.
We decompose
\begin{multline}
\label{eq:zetatilde}
  \tilde \zeta_{n+1} - z_{\star} = (I_{3d}+\gamma_{n+1} G_n)(\tilde \zeta_n -
  z_{\star}) +
  \gamma_{n+1}\left(g_n(\tilde\zeta_n)-G_n(\tilde\zeta_n-z_{\star})\right)\1_{\|\tilde\zeta_n-z_{\star}\|\leq
    M} \\ -
  \gamma_{n+1}(K+G_n)(\tilde\zeta_n-z_{\star})\1_{\|\tilde\zeta_n-z_{\star}\|>
    M} + \gamma_{n+1}(\eta_{n+1}+r_n(\tilde\zeta_n,\delta_n))\1_{{\mathcal A}_n}\,.
\end{multline}
For a given $t>0$, we write $\G = B_t^{-1} G_t B_t$ the Jordan-like decomposition of $\G$,
where the ones of the second diagonal of the usual Jordan decomposition are replaced by $t$,
and where~$B_t$ is some invertible matrix. We define  $S_n\eqdef B_t(\tilde\zeta_n-z_{\star})$.
Setting $G_n^{(t)} \eqdef B_t G_n B_t^{-1}$, we obtain
\begin{multline*}
  S_{n+1}  = (I_{3d}+\gamma_{n+1} G_n^{(t)})S_n +
  \gamma_{n+1}B_t\left(g_n(\tilde\zeta_n)-G_n(\tilde\zeta_n-z_{\star})\right)\1_{\|\tilde\zeta_n-z_{\star}\|\leq
    M} \\ -
  \gamma_{n+1}(K+G_n^{(t)})S_n\1_{\|\tilde\zeta_n-z_{\star}\|>
    M} + \gamma_{n+1}B_t(\eta_{n+1}+r_n(\tilde\zeta_n,\delta_n))\1_{{\mathcal A}_n}\,.
\end{multline*}
Choose $A\in (0,2L')$ and $A'\in (A,2L')$. There exists $\bar \gamma$ and $t>0$ s.t. for every $\gamma<\bar\gamma$,
$\|I+\gamma G_t\|_2\leq 1-\gamma (A'+2L')/2$, where $\|\cdot\|_2$ is the spectral norm. As $G_n^{(t)}\to G^t$, there exists $n_2\geq n_1$, such that
for all $n\geq n_2$, $\|I+\gamma G_n^{(t)}\|_2\leq 1-\gamma A'$.
Recall the notation $\EE_n = \EE[ \cdot \, | \, \mcF_n ]$.
%$\EE_n = \EE(\cdot|{\mathcal F}_n)$.
 We expand $\|S_{n+1}\|^2$ and use the inequality $\left\|g_n(\tilde\zeta_n)-G_n(\tilde\zeta_n-z_{\star})\right\|^2\1_{\|\tilde\zeta_n-z_{\star}\|\leq
    M}\leq C\|S_n\|^2$
to obtain after straightforward algebra
\begin{multline*}
  \EE_n\|S_{n+1}\|^2  \leq  (1-\gamma_{n+1}A')\|S_n\|^2
  + C\gamma_{n+1}^2\|S_n\|^2\\
  +C\gamma_{n+1}^2(\EE_n\|\eta_{n+1}\|^2+\|r_n(\tilde\zeta_n,\delta_n)\|^2)\1_{{\mathcal A}_n}\\
+2\gamma_{n+1}\real{S_n^*B_t\left(g_n(\tilde\zeta_n)-G_n(\tilde\zeta_n-z_{\star})\right)}\1_{\|\tilde\zeta_n-z_{\star}\|\leq M}\\
-2\gamma_{n+1}\real{S_n^*(K+G_n^{(t)})S_n}\1_{\|\tilde\zeta_n-z_{\star}\|> M}
+2\gamma_{n+1}\real{S_n^*B_t r_n(\tilde\zeta_n,\delta_n)} \1_{{\mathcal A}_n}\,.
\end{multline*}
Choose $c\eqdef (A'-A)/2$. If $M$ is chosen small enough,
$$
\|g_n(\tilde\zeta_n)-G_n(\tilde\zeta_n-z_{\star})\|\1_{\|\tilde\zeta_n-z_{\star}\|\leq M} \leq \frac c2\|B_t\|^{-1}\|B_t^{-1}\| \|\tilde\zeta_n-z_{\star}\|\,.
$$
Moreover, choosing $K>\sup_n\|G_n^{(t)}\|_2$, it holds that $\real{S_n^*(K+G_n^{(t)})S_n}\geq 0$. Then,
\begin{multline*}
  \EE_n\|S_{n+1}\|^2  \leq  (1-\gamma_{n+1}(A'-c))\|S_n\|^2 +
  C\gamma_{n+1}^2\|S_n\|^2\\
  + C\gamma_{n+1}^2(\EE_n\|\eta_{n+1}\|^2+\|r_n(\tilde\zeta_n,\delta_n)\|^2)\1_{{\mathcal A}_n}
+2\gamma_{n+1}\|B_t\|\|S_n\| \|r_n(\tilde\zeta_n,\delta_n)\| \1_{{\mathcal A}_n}\,.
\end{multline*}
Using Eq.~(\ref{eq:rdelta}),
\begin{multline*}
  \EE_n\|S_{n+1}\|^2  \leq  (1-\gamma_{n+1}(A'-c- w_n))\|S_n\|^2 +
  C\gamma_{n+1}^2(1+w_n^2)\|S_n\|^2\\
  + C\gamma_{n+1}^2\EE_n\|\eta_{n+1}\|^2\1_{{\mathcal A}_n}\,.
\end{multline*}
Therefore, there exists $n_3\geq n_2$ s.t. for all $n\geq n_3$,
\begin{equation*}
  \EE\|S_{n+1}\|^2  \leq  (1-\gamma_{n+1}A)\EE\|S_n\|^2 + C\gamma_{n+1}^2\EE(\|\eta_{n+1}\|^2\1_{\|x_n-x_{\star}\|\leq M})\,.
\end{equation*}
The second expectation in the righthand side is bounded uniformly in $n$ by the condition~(\ref{eq:momentsTCL}).
Using \cite[Lem.~4 and Lem.~5]{delyon1999convergence}, we conclude that $\sup_n \gamma_n^{-1}\EE\|S_n\|^2<\infty$.
Therefore, $\sup_n\gamma_n^{-1}\EE\|\tilde \zeta_n-z_{\star}\|^2<\infty$, which in turn implies
$\sup_n\gamma_n^{-1}\EE(\|\zeta_n-z_{\star}\|^2\1_{\Omega^{(0)}_{N_0}})<\infty$.
\medskip

\noindent{\bf Strongly perturbed iterations.} We define $\tilde y_n=\sqrt{\gamma_n}^{-1}(\tilde \zeta_n-z_{\star})$.
Define
$$
\bar G_n \eqdef \gamma_{n+1}^{-1}\left(\sqrt{\frac{\gamma_n}{\gamma_{n+1}}} -1\right)I_{3d}+ \sqrt{\frac{\gamma_n}{\gamma_{n+1}}}G_n\,.
$$
The sequence $\bar G_n$ converges to $\bar G_\infty\eqdef \G+ \frac {\1_{\alpha=1}}{2\gamma_0}I_{3d}$.
Recalling Eq.~(\ref{eq:zetatilde}), we can write
$$
  \tilde y_{n+1}  = (I_{3d}+\gamma_{n+1}\bar G_\infty)\tilde y_n +\gamma_{n+1} \bar r_n + \sqrt{\gamma_{n+1}} \bar\eta_{n+1}
$$
where $\bar \eta_{n+1} = \eta_{n+1}\1_{{\mathcal A}_n}$ and $\bar r_n = \bar r_{n,1}+\bar r_{n,2}+\bar r_{n,3}$, where
\begin{align*}
  \bar r_{n,1} &\eqdef \sqrt{\gamma_{n+1}}^{-1}r_n(\tilde\zeta_n,\delta_n)\1_{{\mathcal A}_n} + (\bar G_n-\bar G_\infty)\tilde y_n\\
\bar r_{n,2} &\eqdef \sqrt{\gamma_{n+1}}^{-1}\left(g_n(\tilde\zeta_n)-G_n(\tilde\zeta_n-z_{\star})\right)\1_{\|\tilde\zeta_n-z_{\star}\|\leq M}\\
\bar r_{n,3} &\eqdef  - \sqrt{\gamma_{n+1}}^{-1}(K+G_n)(\tilde\zeta_n-z_{\star})\1_{\|\tilde\zeta_n-z_{\star}\|> M} \,.
%\bar r_{n,4} &\eqdef
\end{align*}
We now check that the assumptions of Th.~\ref{th:pelletier} are fulfilled.
On the event $\Omega^{(0)}_{N_0}$, we recall that~$\tilde \zeta_n = \zeta_n$, hence $\bar r_{n,3}$
is identically zero. Moreover, using Eq.~(\ref{eq:rdelta}), it holds that for all $n$ large enough,
$$
\|\bar r_{n,1}\| \leq C\left(\sqrt{\frac{\gamma_{n}}{\gamma_{n+1}}} (\gamma_{n+1}+w_n)+ \|\bar G_n-\bar G_\infty\|\right)\|\tilde y_n\|
$$
and therefore, $\EE[\|\bar r_{n,1}\|^2] \to 0$. Now consider the term $\bar r_{n,2}$. By Eq.~(\ref{eq:1}),
$$
\|\bar r_{n,2}\| \leq C\sqrt{\gamma_{n+1}}^{-1}\|\tilde \zeta_n-z_{\star}\|^2\1_{\|\tilde\zeta_n-z_{\star}\|\leq M}\,.
$$
Thus, $\|\bar r_{n,2}\|^2 \leq C\|\tilde y_n\|^2$ which implies that $\sup_{n\geq N_0} \EE[\|r_{n,2}\|^2]<\infty$. Moreover,
$\EE[\|\bar r_{n,2}\|] \leq C\sqrt{\gamma_{n+1}}\EE \|\tilde y_n\|^2$ tends to zero.
Finally, consider $\bar \eta_{n+1}$.
Using condition~(\ref{eq:momentsTCL}), there exist $M>0$ and $b_M>4$ s.t.
\begin{align*}
  \EE_n[\|\bar \eta_{n+1}\|^{b_M/2}] &\leq \EE_n[\|\eta_{n+1}\|^{b_M/2}]\1_{\|x_n-x_{\star}\|\leq M} \\
&\leq
C %\left(\EE_n(\|\sigma(x_n, \xi_{n+1})\|^{b_M})+
\EE_n[\|\nabla f(x_n,\xi_{n+1})\|^{b_M}]%\right)
\1_{\|x_n-x_{\star}\|\leq M} \leq C \,.
\end{align*}
Moreover, $\EE_n[\bar \eta_{n+1}]=0$ and finally, almost surely on $\Omega^{(0)}_N$,
$\EE_n[\bar \eta_{n+1}\bar \eta_{n+1}^\T]$ converges to
\begin{equation}
%  \resizebox{.9\hsize}{!}{$
\Sigma \eqdef
\left[ \begin{array}[h]{cc}
   \EE_\xi\left[
  \begin{bmatrix}
    p_\infty(\nabla f(x_{\star},\xi)^{\odot 2}-S(x_{\star})) \\ h_\infty\nabla f(x_{\star},\xi)
  \end{bmatrix}\begin{bmatrix}
    p_\infty(\nabla f(x_{\star},\xi)^{\odot 2}-S(x_{\star})) \\ h_\infty\nabla f(x_{\star},\xi)
  \end{bmatrix}^\T\right] &
\begin{array}[h]{c}
  0 \\ 0
\end{array}
\\
   \begin{array}[h]{cc}
     0 & \ \ \ \ \ \ \ \ \ \ \ \ \ \ \ \ \ \ \ \ \ \ \ \ \ \ \ \ \ \ 0
   \end{array} & 0
 \end{array}\right]\,.\label{eq:Sigma}
%$}
\end{equation}
Therefore, the assumptions of Th.~\ref{th:pelletier} are fulfilled for the sequence $\tilde y_n$.
We obtain the desired result for the sequence $(m_n, x_{n-1})$. We now show that the same result also holds for the sequence $(m_n, x_n)$. For this purpose, observe that
  $$
  \frac{1}{\sqrt{\gamma_n}}
\begin{bmatrix}
  m_n \\ x_n-x_{\star}
\end{bmatrix} =
\frac{1}{\sqrt{\gamma_n}}
\begin{bmatrix}
  m_n \\ x_{n-1}-x_\star
\end{bmatrix} +
\begin{bmatrix}
  0 \\ \frac{1}{\sqrt{\gamma_n}}(x_n-x_{n-1})
\end{bmatrix}\,.
  $$
  Then, notice that $\|\frac{x_n - x_{n-1}}{\sqrt{\gamma_n}}\| = \sqrt{\gamma_n} \|\frac{m_n}{\sqrt{v_n + \varepsilon}}\| \leq \sqrt{\frac{\gamma_n}{\varepsilon}} \|m_n\| \to 0$ as $n \to \infty$ since it is assumed that $z_n \to z_\star$ (which implies in particular that $m_n \to 0$). Hence, it holds that $\sqrt{\gamma_n}^{-1}(x_n - x_{n-1})$ converges a.s. to $0$. We conclude by invoking Slutsky's lemma.

\noindent\textbf{Proof of Eq.~\eqref{eq:cov}.}
We have the subsystem:
\begin{equation}
\tilde H \Gamma + \Gamma \tilde H^\T =
\begin{bmatrix}
-h_\infty^2 \mathcal{Q} & 0 \\
0 & 0
\end{bmatrix}\qquad\text{where }\tilde H \eqdef
\begin{bmatrix}
  (\theta - r_\infty) I_d & h_\infty \nabla^2F(x_{\star}) \\
-V & \theta I_d
\end{bmatrix}\label{eq:lyap-reduced}
\end{equation}
and where $\mathcal{Q} \eqdef \textrm{Cov}\left(\nabla
  f(x_{\star},\xi)\right)$.  The next step is to triangularize the matrix
$\tilde H$ in order to decouple the blocks of $\Gamma$.  For
every $k=1,\dots,d$, set
$\nu_k^\pm \eqdef -\frac{r_\infty}{2}\pm\sqrt{r_\infty^2/4-h_\infty\pi_k}$
with the convention that $\sqrt{-1} =
\imath$ (inspecting the characteristic polynomial of $H$, these %quantities
are the eigenvalues of $H$). Set
$M^\pm\eqdef\diag{(\nu_1^\pm,
\cdots,\nu_d^\pm)}$ and $R^\pm\eqdef
V^{-\frac 12}PM^\pm P^\T V^{-\frac 12}$.  Using the identities $M^++M^-=-r_\infty I_d$ and
$M^+M^-=h_\infty\diag{(\pi_1,\cdots,\pi_d)}$,
%where $\Lambda\eqdef\diag{(\lambda_1,\cdots,\lambda_d)}$,
it can be checked that
$$
\cR\tilde H =\begin{bmatrix}
  R^- V + \theta I_d & 0 \\ -V & VR^+ + \theta I_d
\end{bmatrix}\cR,\text{ where }\cR\eqdef
\begin{bmatrix}
  I_d & R^+ \\ 0 & I_d
\end{bmatrix}\,.
$$
Set $\tilde \Gamma \eqdef \cR\Gamma\cR^\T$. Denote by
$(\tilde \Gamma_{i,j})_{i,j=1,2}$ the blocks of $\tilde \Gamma$.
Note that $\tilde \Gamma_{2,2} = \Gamma_{2,2}$. By left/right multiplication of Eq.~(\ref{eq:lyap-reduced})
respectively by $\cR$ and $\cR^\T$, we obtain

\begin{align}
  &(R^-V+\theta I_d) \tilde\Gamma_{1,1}+\tilde\Gamma_{1,1}(VR^-+\theta I_d) = - h_\infty^2 \mathcal{Q} \label{eq:lapremiere}\\
& (R^-V+\theta I_d) \tilde\Gamma_{1,2}+\tilde\Gamma_{1,2} (R^+V+\theta I_d) = \tilde\Gamma_{1,1}V \label{eq:ladeuxieme}\\
& (VR^++\theta I_d) \tilde\Gamma_{2,2} + \tilde\Gamma_{2,2} (R^+V+\theta I_d) = V\tilde\Gamma_{1,2} + \tilde\Gamma_{1,2}^\T V\,. \label{eq:laderniere}
%\check\Sigma_{1,2} V+ V\check\Sigma_{1,2}^\T
\end{align}
Set $\bar \Gamma_{1,1} = P^{-1}V^{\frac 12}\tilde\Gamma_{1,1} V^{\frac 12}P$.
Define $C\eqdef P^{-1}V^{\frac 12} \mathcal{Q} V^{\frac 12}P$.
Eq.~(\ref{eq:lapremiere}) yields
$$
(M^-+\theta I_d) \bar\Gamma_{1,1} + \bar\Gamma_{1,1} (M^-+\theta I_d) = -h_\infty^2C\,.
$$
Set $\bar \Gamma_{1,2} = P^{-1}V^{\frac 12}\tilde\Gamma_{1,2} V^{-\frac 12}P$.
Eq.~(\ref{eq:ladeuxieme}) is rewritten $(M^-+\theta I_d) \bar \Gamma_{1,2}+\bar \Gamma_{1,2} (M^++\theta I_d) = \bar \Gamma_{1,1}$.
The component $(k,\ell)$ is given by
$$
\bar\Gamma_{1,2}^{k,\ell} = (\nu_k^-+\nu_\ell^++2\theta)^{-1}\bar\Gamma_{1,1}^{k,\ell} = \frac{-h_\infty^2C_{k,\ell}}{(\nu_k^-+\nu_\ell^++2\theta)(\nu_k^-+\nu_\ell^-+2\theta)}\,.
$$
Set finally $\bar \Gamma_{2,2} = P^{-1}V^{-\frac 12}\Gamma_{2,2} V^{-\frac 12}P$.
Eq.~(\ref{eq:laderniere}) becomes
$$
(M^++\theta I_d)\bar \Gamma_{2,2} + \bar \Gamma_{2,2}(M^++\theta I_d) = \bar \Gamma_{1,2} + \bar \Gamma_{1,2}^\T\,.
$$
Thus,
\begin{multline*}
\bar\Gamma_{2,2}^{k,\ell}
= \frac{\bar\Gamma_{1,2}^{k,\ell}+\bar\Gamma_{1,2}^{\ell,k}}{\nu_k^++\nu_\ell^++2\theta}\\
=  \frac{-h_\infty^2C_{k,\ell}}{(\nu_k^++\nu_\ell^++2\theta)(\nu_k^-+\nu_\ell^-+2\theta)}\left(\frac{1}{\nu_k^-+\nu_\ell^++2\theta}
+\frac{1}{\nu_k^++\nu_\ell^-+2\theta}\right)\,.
\end{multline*}
After tedious but straightforward computations, we obtain
$$
\bar\Gamma_{2,2}^{k,\ell} =
\frac{h_\infty^2 C_{k,\ell}}
{(r_\infty -2\theta)( h_\infty(\pi_k+\pi_\ell)
+ 2 \theta (\theta - r_\infty))
+\frac{h_\infty^2(\pi_k-\pi_\ell)^2}{2(r_\infty - 2 \theta)}}\,,
$$
and the result is proved.

% =============== Proofs for the traps

\section{Proofs for Section~\ref{sec:avt}}
\label{sec:proofs_avt}

\subsection{Preliminaries}
\label{subsec:prel-avt}

Most of the avoidance of traps results in the stochastic approximation
literature deal with the case where the ODE that underlies the stochastic
algorithm under study is an autonomous ODE $\dot \sz = h(\sz)$.  In this
setting, a point $z_\star \in \zer h$ is called a trap if $h(z)$ admits an
expansion around $z_\star$ of the type $h(z) = D(z-z_\star) +
o(\|z-z_\star\|)$, where the matrix $D$ has at least one eigenvalue which real
part is (strictly) positive.  Initiated by Pemantle \cite{pem-90} and by
Brandi\`ere and Duflo \cite{bra-duf-96}, the most powerful class of techniques
for establishing avoidance of traps results makes use of Poincar\'e's invariant
manifold theorem for the ODE $\dot \sz = h(\sz)$ in a neighborhood of some
point $z_\star \in \zer h$.  The idea is to show that with probability~1, the
stochastic algorithm strays away from the maximal invariant manifold of the ODE
where the convergence to $z_\star$ of the ODE flow can take place.  As
previously mentioned, since we are dealing with algorithms derived from
non-autonomous ODEs, we extend the results of \cite{pem-90,bra-duf-96} to this
setting.  The proof of Th.~\ref{th:av-traps} relies on a non-autonomous version
of Poincar\'e's theorem.  We borrow this result from the rich literature that
exists on the subject \cite{dal-krei-(livre)74, klo-ras-(livre)11}.

Let us start by setting the context for the non-autonomous version that we
shall need for the invariant manifold theorem. Given an integer $d > 0$ and a
matrix $D \in \RR^{d\times d}$, consider the linear autonomous differential
equation
\begin{equation}
\label{auton}
\dot \sz(t) = D \sz(t) ,
\end{equation}
which solution is obviously $\sz(t) = e^{Dt} \sz(0)$ for $t \in \RR$. Let us
factorize $D$ as in~\eqref{jordan}, and write $D = Q \Lambda Q^{-1}$ with
$\Lambda = \begin{bmatrix} \Lambda^- \\ & \Lambda^+ \end{bmatrix}$ where we
recall that the Jordan blocks that constitute $\Lambda^- \in \RR^{d^-\times
d^-}$ (respectively~$\Lambda^+ \in \RR^{d^+ \times d^+}$) are those that
contain the eigenvalues $\lambda_i$ of $D$ such that $\Re\lambda_i \leq 0$
(respectively $\Re\lambda_i > 0$). Let us assume here that both $d^-$ and $d^+$
are positive.  It will be convenient to work in the basis of the columns of $Q$
by making the variable change
\[
z \mapsto y = \begin{bmatrix} y^- \\ y^+ \end{bmatrix} = Q^{-1} z,
\]
where $y^\pm \in \RR^{d^\pm}$. In this new basis, the ODE~\eqref{auton} is
written as
\begin{equation}
\label{auton-base}
\begin{bmatrix} \dot \sy^- \\ \dot \sy^+ \end{bmatrix} =
\begin{bmatrix} \Lambda^- \\ & \Lambda^+ \end{bmatrix}
\begin{bmatrix}  \sy^- \\ \sy^+ \end{bmatrix} ,
\end{equation}
which solution is $\sy^\pm(t) = \exp(t \Lambda^\pm) \sy^\pm(0)$. One can
readily check that for each couple of real numbers $\alpha^+$ and $\alpha^-$
that satisfy
\begin{equation}
\label{eq:alpha+-}
0 < \alpha^- < \alpha^+ < \min \{ \Re\lambda_i \, : \, \Re\lambda_i > 0 \},
\end{equation}
there exists a so-called exponential dichotomy of the ODE solutions, which
amounts in our case to the existence of two constants $K^-, K^+ \geq 1$ such
that
\begin{align*}
\| \exp(t \Lambda^-) \|
    &\leq K^- e^{\alpha^- t} \quad \text{for } t \geq 0, \\
\| \exp(t \Lambda^+) \|
    &\leq K^+ e^{\alpha^+ t} \quad \text{for } t \leq 0,
\end{align*}
see, \emph{e.g.}, \cite{hor-joh-(livre)-topics}.

We now consider a non-autonomous perturbation of this ODE, which is
represented in the basis of the columns of $Q$ as
\begin{equation}
\label{nonauton}
\dot \sy(t) = h(\sy(t), t) \quad \text{with} \quad
h(y,t) = \begin{bmatrix} \Lambda^- \\ & \Lambda^+ \end{bmatrix} y
            + \varepsilon(y, t) ,
\end{equation}
and $\varepsilon: \RR^d \times \RR \to \RR^d$ is a continuous function.  In the
sequel, we shall be interested in the asymptotic behavior of this equation for
the large values of~$t$, and therefore, restrict our study to the interval $\II
= [t_0, \infty)$ for some given $t_0 \geq 0$ that we shall fix later.  We
assume that $\varepsilon(0,\cdot) = 0$ on $\II$. We denote as $\phi : \II
\times \II \times \RR^d \to \RR^d$ the so-called general solution
of~\eqref{nonauton}, which is defined by the fact that $\phi(\cdot, t, x)$ is
the unique noncontinuable solution of~\eqref{nonauton} such that $\phi(t , t,
x) = x$ for $t\in \II$ and $x \in \RR^d$, assuming this solution exists and is
unique for each $(x,t) \in \RR^d \times \II$.

In the linear autonomous case provided by the ODE~\eqref{auton-base}, the
subspace
\[
\mathcal G = \left\{
  \left( t, \begin{bmatrix} y^- \\ 0 \end{bmatrix}\right)
 \in \RR \times \RR^d \ : \ y^- \in \RR^{d^-} \right\}
\]
is trivially invariant in the sense that if $(t, y) \in \mathcal G$, then,
$(s,\phi(s,t,y)) \in \mathcal G$ for each $s \in \RR$. This concept can be
generalized to the non-linear and non-autonomous case. We say that the
$\mathcal C^1$ function $w : \RR^{d^-} \times \II \to \RR^{d^+}$ defines a
global non-autonomous invariant manifold for the ODE~\eqref{nonauton} if
$w(0,t) = 0$ for all $t \in \II$, and, furthermore, if for each $t \in \II$ and
each $y^- \in \RR^{d^-}$, writing $y = (y^-, w(y^-, t))$, the general solution
$\phi(s, t, y) = (\phi^-(s, t, y), \phi^+(s, t, y))$ with $\phi^\pm(s, t, y)
\in \RR^{d^\pm}$ verifies $\phi^+(s, t, y) = w(\phi^-(s, t, y), s)$ for each $s
\in \II$.  The non-autonomous invariant manifold is the set
\[
\mathcal G = \left\{
  \left( t, \begin{bmatrix} y^- \\ w(y^-, t) \end{bmatrix}\right)
 \in \II \times \RR^d \ : \ y^- \in \RR^{d^-} \right\} ,
\]
which obviously satisfies
$(t,y) \in \mathcal G \ \Rightarrow \
  (s, \phi(s,t,y)) \in \mathcal G \ \text{for each } s \in \II$.

These invariant manifolds are described by the following proposition, which is
a straightforward application of \cite[Th.~A.1]{pot-ras-06}
(see also \cite[Th.~6.3 p.~106, Rem.~6.6 p.~111]{klo-ras-(livre)11}).
It is useful to note that under the conditions provided in the statement of this proposition,
the existence of the general solution $\phi$ of the ODE~\eqref{nonauton} is
ensured by Picard's theorem.

\begin{proposition}
\label{prop:inv-man}
Let $\II = [t_0, \infty)$ for some $t_0 \geq 0$. Assume that the function
$\varepsilon(y, t)$ is such that $\varepsilon(0,\cdot) \equiv 0$ on $\II$,
the function $\varepsilon(\cdot, t)$
%is $\mathcal C^1$
is continuously differentiable for each $t \in \II$,
and furthermore, the Jacobian matrix $\partial_1 \varepsilon (y,t)$ satisfies
\begin{equation}\label{eq:eps_lip}
  |\varepsilon |_1 \eqdef
   \sup_{(y,t) \in \RR^d \times \II} \| \partial_1\varepsilon(y,t) \|
    <  \frac{\alpha^+ - \alpha^-}{4 K}
\end{equation}
with $K = K^- + K^+ + K^- K^+ (K^- \vee K^+)$ and $\alpha^-, \alpha^+$ chosen as in Eq.~\eqref{eq:alpha+-}.
Then, for each
$\delta \in ( 2 K |\varepsilon |_1, (\alpha^+ - \alpha^-) /2)$ and
each $\gamma \in (\alpha^- + \delta, \alpha^+ - \delta)$, the set
\[
\mathcal G = \left\{ (t, y) \in \II \times \RR^d \ : \
  \sup_{s\geq t} \| \phi(s,t,y) \| \exp(\gamma(t-s)) < \infty \right\}
\]
is nonempty, and does not depend on $\gamma$. Moreover, this set is a
global invariant manifold for the ODE~\eqref{nonauton} that is defined by
a %$\mathcal C^1$
continuously differentiable mapping $w : \RR^{d^-} \times \II \to \RR^{d^+}$.
In addition,
if the partial derivatives $\partial_1^k \varepsilon : \RR^{d} \times \II$
exist and are continuous for $k \in \{1, \ldots, m \}$ with globally bounded
partial derivatives
\begin{equation}
\label{eq:deriv_eps_k}
|\varepsilon |_k \eqdef
   \sup_{(y,t) \in \RR^d \times \II} \| \partial_1^k \varepsilon(y,t) \| < \infty\,,
\end{equation}
under the gap condition
\begin{equation}
\label{gap}
m \alpha^- < \alpha^+, \ m \in \NN^*,
\end{equation}
the partial derivatives $\partial_1^k w : \RR^{d^-} \times \II$ exist and are
continuous with
\begin{equation}
\label{eq:deriv_w_k}
\sup_{(y^-,t)\in \RR^{d^-} \times\II} \| \partial_1^k w(y^-,t) \| < \infty
\quad \text{for all } k \in \{1, \ldots, m \}.
\end{equation}
Finally, if $\partial_2^n \partial_1^k \varepsilon$ exist and are continuous
for $0\leq n < m$ and $0\leq k+n \leq m$, then $w$ is
%$\mathcal C^m$.
$m$-times continuously differentiable.
\end{proposition}

Let us partition the function $h(y,t)$ as
\begin{equation}
\label{part-h}
h(y,t) = \begin{bmatrix} h^-(y,t) \\ h^+(y,t) \end{bmatrix} =
 \begin{bmatrix} \Lambda^- y^- + \varepsilon^-(y,t) \\
      \Lambda^+ y^+ + \varepsilon^+(y,t)
 \end{bmatrix} ,
\end{equation}
where $h^\pm : \RR^{d} \times\II \to \RR^{d^\pm}$, $y^\pm \in \RR^{d^\pm}$ and
$\varepsilon^\pm :  \RR^{d} \times\II \to \RR^{d^\pm}$. With these notations,
the previous proposition leads to the following lemma.
%, which is proven in Section~\ref{prf-g+g-}.

\begin{lemma}
\label{g+g-}
In the setting of Prop.~\ref{prop:inv-man}, for each $t$ in the interior of
$\II$ and each vector $y = (y^-, y^+)$ such that $y^\pm \in \RR^{d^\pm}$ and
$y^+ = w(y^-, t)$, it holds that
\begin{equation}\label{eq:g+g-}
  h^+(y,t) = \partial_{1} w(y^-, t) h^-(y,t) + \partial_2 w(y^-, t) \, .
\end{equation}
Assume that $\alpha^-$ is small enough so that Ineq.~\eqref{gap}
and Eq.~\eqref{eq:deriv_eps_k} hold true with $m=2$.
Assume in addition that
$\partial_2^n \partial_1^k \varepsilon$ exists and is continuous
for $0\leq n < 2$ and $0\leq k+n \leq 2$, and furthermore, that there exists
a bounded neighborhood $\cV \subset \RR^d$ of zero such that
\begin{equation}\label{eq:second_deriv_bound}
 \sup_{(y,t) \in \cV\times\II} \norm{\partial_2 \varepsilon(y,t)} < + \infty  .
 \end{equation}
Then, there exists a neighborhood $\cV^-\subset\RR^{d^-}$ of zero such that
% \begin{equation}
\begin{gather}
   \sup_{(y^-,t) \in \cV^{-}\times \II}
      \norm{\partial_1 \partial_2 w(y^-, t)} < +\infty \,,
      \label{bound-d12w}\\
   \sup_{(y^-,t) \in \cV^{-}\times \II}
       \norm{\partial_2^2 w( y^-, t) } < +\infty \,.
       \label{bnd-d22w}
 \end{gather}
 %\end{equation}
\end{lemma}

\begin{proof}
By Prop.~\ref{prop:inv-man}, the general solution $\phi(s, t, y)$ of the
ODE~\eqref{nonauton} can be written as $\phi(s, t, y) = ( \phi^-(s, t, y),
\phi^+(s, t, y) )$ with $\phi^+(s, t, y) = w(\phi^-(s, t, y), s)$ for each
$s\in \II$. Equating the derivatives with respect to $s$ of the two members of
this equation and taking $s=t$, we get the first equation.

Writing
$g : \RR^d{^-} \times \II \to \RR^d, (y^-, t)\mapsto ( y^-, w(y^-, t))$,
Eq.~\eqref{eq:g+g-} can be rewritten as
\begin{equation}
\label{d2w}
\partial_2 w(y^-, t) = h^+(g(y^-, t),t)
        - \partial_{1} w(y^-, t) h^-(g(y^-,t),t)  .
\end{equation}
By Prop.~\ref{prop:inv-man}, the function $w$ is twice differentiable,
and we can write
\begin{equation}
\label{d22w}
\partial_2^2 w(y^-, t) = \partial_1 h^+ \partial_2 g + \partial_2 h^+
 - (\partial_1 \partial_2 w) h^- - (\partial_1 w) (\partial_1 h^- \partial_2 g
 + \partial_2 h^- ) ,
\end{equation}
where, \emph{e.g.}, $h^+$ is a shorthand notation for $h^+(g(y^-, t),t)$.  It
holds from Eq.~\eqref{part-h} and the assumptions of
Prop.~\ref{prop:inv-man} that for each $(y,t) \in \RR^d \times \II$,
\begin{equation}
\label{bnd-d1h}
\| \partial_1 h(y,t) \| \leq \|
\Lambda \| + \| \partial_1 \varepsilon(y,t) \| \leq C,
\end{equation}
where the constant $C > 0$ is independent of $(y, t)$ and can change from an
inequality to another in the remainder of the proof.  By the mean value
inequality and Prop.~\ref{prop:inv-man}, we also get that
\[
\| w(y^-, t) \| = \| w(y^-,t) - w(0, t) \| \leq
 \sup_{(u,s)} \| \partial_1 w(u, s) \| \ \| y^- \|
 \leq C \| y^- \| ,
\]
thus, $\| g(y^-, t) \| \leq C \| y^- \|$. By the mean value inequality again,
\begin{multline*}
\norm{h(g(y^-,t),t)} =
\norm{h(g(y^-,t), t) - h(0, t)} \leq \sup_{(u,t)} \norm{\partial_1 h(u, t)}
\norm{g(y^-,t)}\\
 \leq C \norm{g(y^-,t)} \leq C \| y^- \|.
\end{multline*}
By Eq.~\eqref{d2w} and Prop.~\ref{prop:inv-man}, this implies that
\begin{gather}
\norm{\partial_2 g(y^-,t) } = \norm{\partial_2 w(y^-, t)}
 = \norm{h^+ - (\partial_1 w) h^-} \leq C \norm{y^-} , \quad \text{and}
 \label{bnd-d2g} \\
\norm{\partial_1 \partial_2 w(y^-, t)} =
\norm{\partial_1 h^+ \partial_1 g - (\partial_1^2 w) h^-
 - (\partial_1 w) (\partial_1 h^- \partial_1 g)} \leq
 C (\norm{y^-} + 1) .
\label{bnd-d12w}
\end{gather}
Let $\cV^- \subset \RR^{d^-}$ be a small enough neighborhood of zero so that
$g(y^-,t) \in \cV$ for each $y^- \in \cV^-$, which is possible by the
inequality $\| g(y^-, t) \| \leq C \| y^- \|$.  By the assumption on
$\| \partial_2\varepsilon(y,t)\|$ in the statement of Lem.~\ref{g+g-}, we have
\begin{equation}
\label{bnd-d2h}
\forall y^- \in \cV^-, \quad
\norm{\partial_2 h(g(y^-,t), t)} = \norm{\partial_2 \varepsilon(g(y^-,t), t)}
 \leq C .
\end{equation}
The bound~\eqref{bound-d12w} is an immediate consequence of Eq.~\eqref{bnd-d12w}.
Getting back to Eq.~\eqref{d22w}, the bound~\eqref{bnd-d22w} follows
from the inequalities~\eqref{bnd-d1h}--\eqref{bnd-d2h}.
\end{proof}

Prop.~\ref{prop:inv-man} deals with the case where the function $\varepsilon$ is globally Lipschitz continuous.
In practical cases, such a strong assumption is not necessarily verified. In particular, for the ODEs we consider for our application,
it is not satisfied (see the function $e$ defined in Subsec.~\ref{prf:lm-lin-g} below). Nonetheless, recall that we only need the existence of
a \textit{local} non-autonomous invariant manifold, i.e. defined in the vicinity of an arbitrary solution such as the trivial zero solution
(since we suppose here $\varepsilon(0,\cdot) = 0$) whereas the aforementioned strong assumption provides a global non-autonomous invariant
manifold. Indeed, as for the avoidance of traps result we intend to show, we will only need to look at the behavior of our ODE in the neighborhood
of a trap $z_\star$. Therefore, in prevision of the proof of Th.~\ref{th:av-traps}, we localize the ODE~\eqref{nonauton} in the neighborhood of zero.
This is the purpose of the next proposition.

\begin{proposition}\label{prop:inv-man-local}
  Let $\II = [t_0, + \infty)$ for some $t_0 \geq 0$ and let $h: \RR^d \times \II \rightarrow \RR^d$ be defined as in Eq.~\eqref{nonauton}. Assume that $\varepsilon(0, \cdot) \equiv 0$ on $\II$, that the function $\varepsilon(\cdot, t)$ is continuously differentiable for every $t \in \II$ and that
  \begin{equation}\label{eq:conv_inv_man_loc}
    \lim_{(y, t) \rightarrow (0, + \infty) } \norm{\partial_1 \varepsilon(y, t)} = 0 \, .
  \end{equation}
  Then, there exist $\sigma > 0, t_1 > 0$, a function $\tilde{\varepsilon} : \RR^d \times \II_1 \to \RR^d$ where $\II_1 \eqdef [t_1, + \infty)$ and a function $\tilde{h} : \RR^d \times \II_1 \rightarrow \RR^d$ defined for every $y \in \RR^d, t \in \II_1$ by $\tilde{h}(y,t) = \Lambda y + \tilde{\varepsilon}(y,t)$ s.t. $\tilde{h}$ and $\tilde{\varepsilon}$ verify the assumptions of Prop.~\ref{prop:inv-man} and for every $(y,t) \in B(0, \sigma) \times \II_1$, we have that $\tilde{h}(y,t) = h(y,t)$ and $\tilde{\varepsilon}(y,t) = \varepsilon(y, t)$. Moreover, for any $\delta >0$, we can choose $\sigma, t_1$ respectively small and large enough s.t. the mapping $w : \RR^{d^-} \times \II_1 \rightarrow \RR^{d^{+}}$ obtained from Prop.~\ref{prop:inv-man} (applied to $\tilde{h}$ and $\tilde{\varepsilon}$) satisfies
  \begin{equation}\label{eq:omega-loc-lip}
     |w|_1  = \sup_{(y,t) \in \RR^{d^-} \times \II_1} \| \partial_1 w(y,t) \| < \delta\,.
  \end{equation}
  Furthermore, Eq.~\eqref{eq:g+g-} holds for $\tilde{h}$ and $w$ for all $(y,t) \in B(0, \sigma) \times \II_1$. If, additionally, Eq.~\eqref{eq:second_deriv_bound} holds for $\varepsilon$, then there exists $\sigma_1 \leq \sigma$ such that
   \begin{gather}
      \sup_{(y^-,t) \in B(0, \sigma_1)\times \II_1}
         \norm{\partial_1 \partial_2 w(y^-, t)} < +\infty \,,
         \label{bound-d12w-loc}\\
      \sup_{(y^-,t) \in B(0, \sigma_1)\times \II_1}
          \norm{\partial_2^2 w( y^-, t) } < +\infty \,.
          \label{bnd-d22w-loc}
    \end{gather}
\end{proposition}

\begin{proof}
  The idea of the proof is to \emph{localize} the function $h(y,t)$ to a
  neighborhood of zero in the variable $y$ for the purpose of applying
  Prop.~\ref{prop:inv-man}. This cut-off technique is known in the non-autonomous ODE literature, see, \emph{e.g.}, \cite[Th.~6.10]{klo-ras-(livre)11}.
  Let $\psi : \RR^d \to [0,1]$ be a smooth function such that $\psi(y) = 1$
  if $\| y \| \leq 1$, and $\psi(y) = 0$ if $\| y \| \geq 2$. Let
  $C = \max_y \| \nabla\psi(y) \|$ where $\nabla\psi$ is the Jacobian
  matrix of~$\psi$. Thanks to the convergence~\eqref{eq:conv_inv_man_loc}, we can choose
  $t_1 > 0$ large enough and~$\sigma > 0$ small enough so that
  \[
  \sup_{(t,y) \in [t_1, \infty) \times B(0, 2\sigma)}
  \| \partial_1 \varepsilon(y,t) \| < \frac{\alpha^+ - \alpha^-}{4 K( 1 + 2 C)} ,
  \]
  and we set $\II_1 = [t_1, \infty)$. Writing
  $\tilde{\varepsilon}(y,t) = \psi(y/\sigma) \varepsilon(y,t)$, it holds that for each
  $(t,y) \in \II_1 \times \RR^d$,
  \begin{align*}
  \| \partial_1 \tilde{\varepsilon}(y,t) \| &\leq
    \sigma^{-1} C \1_{\| y \| \leq 2\sigma} \| \varepsilon(y,t) \|
        + \1_{\| y \| \leq 2\sigma} \| \partial_1 \varepsilon(y,t) \| \\
   &\leq
    \left( \max_{\| y \| \leq 2\sigma} \| \partial_1 \varepsilon(y,t) \| \right)
     \left( \sigma^{-1} C \| y \| + 1 \right)\1_{\| y \| \leq 2\sigma} \\
   &\leq \frac{\alpha^+ - \alpha^-}{4 K},
  \end{align*}
  where we used the mean value inequality along with $\varepsilon(0,t) = 0$ to
  obtain the second inequality. Thus, the function $\tilde{h}(y,t) = \Lambda y +
  \tilde{\varepsilon}(y,t)$ satisfies all the assumptions of Prop.~\ref{prop:inv-man}.
  In addition, the function~$\tilde\varepsilon$ coincides with the function~$\varepsilon$ on $B(0, \sigma_1) \times \II_1$, and
  so it is for the functions $\tilde h$ and $h$.
  %for $(y,t) \in B(0, \sigma_1) \times \II_1$ we have that $\tilde{h}(y,t) = h(y,t)$ and $\tilde{\varepsilon}(y,t) = \varepsilon(y, t)$.
  Finally, it follows from \cite[Th. 6.3]{klo-ras-(livre)11} that
  $$
  |w|_1 \leq \frac{2K^2}{\alpha_+ - \alpha_- -4K |\tilde{\varepsilon}|_1}|\tilde{\varepsilon}|_1\,
  $$
  %there exists a constant $C_1 >0$ s.t. $|w|_1 \leq C_1|\tilde{\varepsilon}|_1$
  (note that~$L$ in \cite[Th. 6.3]{klo-ras-(livre)11} corresponds to $|\tilde{\varepsilon}|_1$ with our notations). Using Eq.~\eqref{eq:conv_inv_man_loc},
  we can make $|\tilde{\varepsilon}|_1$ as small as needed by choosing~$\sigma, t_1$ respectively small and large enough, which gives us Eq.~\eqref{eq:omega-loc-lip}. The proof of the last two equations follows from the application of Lemma~\ref{g+g-} to $\tilde{h}$ and~$w$. The result is immediate after noticing that for $(y,t) \in \RR^{d} \times \II_1$, we have $\norm{\partial_2 \tilde{\varepsilon}(y,t)} \leq \norm{\partial_2 \varepsilon(y,t)} $.
\end{proof}

\subsection{Proof of Th.~\ref{th:av-traps}}
\label{proof:th-av-traps}

We shall rely on the following result of Brandi\`ere and Duflo. Recall that $(\Omega, \mcF, \PP)$ is a probability space equipped with a filtration $(\mcF_n)_{n\in\NN}$.

\begin{proposition}{(\cite[Prop.~4]{bra-duf-96})}
\label{prop:duflo}
Given a sequence $(\gamma_n)$ of deterministic nonnegative stepsizes
such that $\sum_k \gamma_k = +\infty$ and $\sum_k \gamma_k^2 < +\infty$,
consider the $\RR^d$--valued stochastic process $(z_n)_{n\in\NN}$ given by
\[
z_{n+1} = ( I +\gamma_{n+1} H_n ) z_n + \gamma_{n+1} \eta_{n+1} +
\gamma_{n+1} \rho_{n+1}.
\]
Assume that $z_0$ is $\mcF_0$--measurable and that the sequences $(\eta_n)$,
$(\rho_n)$ together with the sequence of random
matrices $(H_n)$ are $(\mcF_n)$--adapted. Moreover, on a given event $A \in \mcF$, assume the following facts:
\begin{enumerate}[{\it i)}]
\item $\sum_n \| \rho_n \|^2 < \infty$.
\item $\limsup \EE [\| \eta_{n+1} \|^{2+a} \, | \, \mcF_n ] < \infty$ for some
 $a > 0$, and $\EE [\eta_{n+1} \, | \, \mcF_n ] = 0$.
\item $\liminf \EE [\| \eta_{n+1} \|^{2} \, | \, \mcF_n ] > 0$.
\end{enumerate}
Let $H \in \RR^{d\times d}$ be a deterministic matrix such that the real parts
of its eigenvalues are all positive. Then,
\[
\PP\left( A \cap [ z_n \to 0 ] \cap [H_n \to H ] \right) = 0.
\]
\end{proposition}

We now enter the proof of Th.~\ref{th:av-traps}. Recall the
development~\eqref{b-dl-z*} of $b(z,t)$ near~$z_\star$ and the spectral
factorization~\eqref{jordan} of the matrix $D$. To begin with, it will be
convenient to make the variable change $y = Q^{-1} (z-z_\star)$, and set
\[
h(y,t) = Q^{-1} b(Q y + z_\star, t) = \Lambda y + \tilde e(y,t),
\]
with $\tilde e(y,t) = Q^{-1} e(Q y + z_\star, t)$, in such a way that our stochastic
algorithm is rewritten as
$$
y_{n+1} = y_n + \gamma_{n+1} h(y_n, \tau_{n})
+ \gamma_{n+1} \tilde\eta_{n+1} + \gamma_{n+1} \tilde\rho_{n+1}\,
$$
where $\tilde\eta_n$ is as in the statement
of the theorem and $\tilde\rho_n = Q^{-1} \rho_n$. Observe that the assumptions on the function $e$ in the statement of the theorem remain true
for $\tilde e$ with $z_\star$ replaced by zero.

If the matrix $\Lambda$ has only eigenvalues with (strictly) positive real
parts, \emph{i.e.}, $d^-=0$, then we can apply Prop.~\ref{prop:duflo} to the
sequence~$(z_n)$. Henceforth, we deal with the more complicated case where~$d^-
> 0$.

Apply Prop.~\ref{prop:inv-man-local} to $h$ to obtain $\tilde{h}$ and $\sigma,t_1$ respectively small and large enough and $w : \RR^{d^-}\times \II_1 \to \RR^{d^+}$ where $\II_1 := [t_1, + \infty)$. By Assumption~\ref{dt_eps} of Th.~\ref{th:av-traps} and Prop.~\ref{prop:inv-man-local} we can choose $\sigma_1 \leq \sigma$ such that Eq.~\eqref{bound-d12w-loc} and Eq.~\eqref{bnd-d22w-loc} hold.
Now, given $p \in \NN$, let us define the event
\[
 E_p = \left[ \forall n \geq p, \ \| y_n \| < \sigma_1, \tau_n \in \II_1 \right].
\]
On $E_p$, it holds that $h(y_n, \tau_{n})= \tilde{h}(y_n, \tau_n)$ and
\begin{align}
\forall n \geq p, \quad
y_{n+1} &= y_n + \gamma_{n+1} h(y_n, \tau_{n}) + \gamma_{n+1} \tilde\eta_{n+1}
              + \gamma_{n+1} \tilde\rho_{n+1} \nonumber \\
 &= \begin{bmatrix} y^-_{n} \\ y^+_n \end{bmatrix}
 + \gamma_{n+1}
      \begin{bmatrix} h^-(y_n, \tau_n) \\ h^+(y_n, \tau_n)\end{bmatrix}
 + \gamma_{n+1}
   \begin{bmatrix} \tilde\eta^-_{n+1} \\ \tilde\eta^+_{n+1} \end{bmatrix}
 + \gamma_{n+1}
   \begin{bmatrix} \tilde\rho^-_{n+1} \\ \tilde\rho^+_{n+1} \end{bmatrix}\,
\label{alg-y}
\end{align}
where $h$ is partitioned as in~\eqref{part-h}, and where
$\tilde\eta^\pm_n, \tilde\rho^\pm_n \in \RR^{d^\pm}$. Note that, by Prop.~\ref{prop:inv-man-local} and Assumptions~\ref{traps-noise-mom} and \ref{traps-noise} on the sequence $(\eta_n)$, we can choose $\sigma,t_1$ respectively small and large enough such that
\begin{equation}\label{eq:noise_repulsiv_loc}
  \liminf \EE[\norm{\tilde{\eta}^{+}_{n+1} }^2 |\mcF_n] \1_{E_p}(y_n)\\
  - 2 \limsup \EE[\norm{\partial_1 w(y_n^-, \tau_n)\tilde{\eta}_{n+1}^{-}}^2 | \mcF_n] \1_{E_p}(y_n) > \frac{c^2}{2} \, .
\end{equation}
This inequality will be important in the end of our proof.
Let $t$ be in the interior of $\II_1$, and let $y = (y^-, y^+)$ be in a neighborhood of $0$.
Make the variable change $(y^-, y^+) \mapsto (u^-, u^+)$ with
\begin{align*}
u^+ &= y^+ - w(y^-, t), \\
u^- &= y^- ,
\end{align*}
where $w$ is the function defined in the statement of
Prop.~\ref{prop:inv-man-local}, and let
\begin{align*}
W(u^-, u^+, t) &= h^+(y,t) - \partial_1 w(y^-, t) h^-(y,t)
    - \partial_2 w(y^-, t) \\
 &= h^+((u^-, u^+ + w(u^-, t)), t)\\
  &- \partial_1  w(u^-, t) h^-((u^-, u^+ + w(u^-, t)), t)
  - \partial_2 w(u^-, t) .
\end{align*}
By Prop.~\ref{prop:inv-man-local} and Lem.~\ref{g+g-}, it holds that $W(u^-, 0, t) = 0$. Moreover,
$W(u^-, \cdot, t) \in \mathcal C^1$ by the assumptions on $h$. Therefore,
writing $y(r) = ( u^- , r u^+ + w(u^-,t) )$ for $r\in[0,1]$, and using the
decomposition~\eqref{part-h}, we get that
\begin{align*}
W(u^-, u^+, t) &= \int_0^1 \partial_{2} W(u^-, r u^+, t) u^+ \, dr \\
 &= \Lambda^+ u^+ \\ &\phantom{=} + \int_0^1
\left( \partial_1 \varepsilon^+(y(r), t)
   \begin{bmatrix} 0 \\ I_{d^+} \end{bmatrix}
 - \partial_1 w(u^-,t) \partial_1 \varepsilon^-(y(r), t)
   \begin{bmatrix} 0 \\ I_{d^+} \end{bmatrix} \right)
     u^+ dr .
\end{align*}
We can also write $y(r) = ( y^- , r y^+ + (1-r) w(y^-,t))$.
Recalling that $w(0,t) = 0$ and that $\| \partial_1 w(y^-, t) \|$ is
bounded on $\RR^{d^-} \times \II$, we get by the mean value inequality that
$\norm{w(y^-,t)} \leq C \norm{y^-}$ where $C > 0$ is a constant. Thus,
$\norm{y(r)} \leq (1+C) \norm{y}$. Moreover, $\varepsilon(y,t) =
Q^{-1} e( Qy, t)$ for $\| y \| < \sigma$. Thus, we get by~\eqref{d1_eps} that
$\norm{\partial_1\varepsilon(y(r), t)} \to 0$ as $(y,t) \to (0,\infty)$
uniformly in $r\in [0,1]$. Using again the boundedness of
$\| \partial_1 w(\cdot, \cdot) \|$, we eventually obtain that
\[
W(u^-, u^+, t) = \left(\Lambda^+ + \Delta(y,t) \right) u^+,
\quad \text{with} \quad
 \lim_{(y,t)\to(0,\infty)} \Delta(y,t) = 0 .
\]
On the event $E_p$, assume that $n \geq p$, and write
\[
u^+_n = y^+_n - w(y^-_n, \tau_n), \quad u^-_n = y^-_n,
\]
(see Eq.~\eqref{alg-y}).
Choosing $\alpha_- > 0$ small enough so that the gap condition~\eqref{gap}
is satisfied with $m = 2$, we have by Taylor's expansion
\begin{align*}
w(y^-_{n+1}, \tau_{n+1}) - w(y^-_n, \tau_n) &=
w(y^-_{n+1}, \tau_{n+1}) - w(y^-_n, \tau_{n+1})
                  + w(y^-_n, \tau_{n+1}) - w(y^-_n, \tau_n)  \\
 &= \partial_1 w(y^-_n, \tau_{n+1}) ( y^-_{n+1} - y^-_n )
  + \gamma_{n+1} \partial_2 w(y^-_n, \tau_n)
+ \epsilon_{n+1}  +  \epsilon_{n+1}^{\gamma} \, ,
\end{align*}
\begin{align*}
  \text{with}\quad
\norm{\epsilon_{n+1}}  &\leq \underset{y^- \in [y^-_n, y^-_{n+1}]}{\sup} \norm{\partial_1^2 w(y^{-},\tau_{n+1})} \norm{ y^-_{n+1} - y^-_n}^2 \,,\\
\text{and}\quad \norm{\epsilon^{\gamma}_{n+1}} &\leq  \underset{\tau \in [\tau_n, \tau_{n+1}]}{\sup} \norm{\partial_2^2 w(y^{-}_n, \tau)}\gamma_{n+1}^2\,.
\end{align*}
Using this equation, we obtain
\begin{align*}
u^+_{n+1} - u^+_n &= \gamma_{n+1} W(u^-_n, u^+_n, \tau_n)
  + \gamma_{n+1} \left(\tilde\eta^+_{n+1}
      - \partial_1 w(y^-_n, \tau_{n+1}) \tilde\eta^-_{n+1}\right) \\
   &\phantom{=}
  + \gamma_{n+1} \left(\tilde\rho^+_{n+1}
      - \partial_1 w(y^-_n, \tau_{n+1}) \tilde\rho^-_{n+1} \right)
    - \epsilon_{n+1} -  \epsilon^{\gamma}_{n+1}\\
    &\phantom{=}
    + \gamma_{n+1} \left( \partial_1 w(y^-_n, \tau_{n})
    - \partial_1 w(y^-_n, \tau_{n+1}) \right) h^-(y_n,\tau_n) \, ,
\end{align*}
which leads to
\begin{equation}
  \label{eq:un+}
u^+_{n+1} = u^+_n
   + \gamma_{n+1} \left( \Lambda^+ + \Delta(y_n, \tau_n) \right) u^+_n
    + \gamma_{n+1} \bar\eta_{n+1} + \gamma_{n+1} \bar\rho_{n+1}
      ,
\end{equation}
with $\bar\eta_{n+1} = \tilde\eta^+_{n+1} - \partial_1 w(y^-_n, \tau_{n})
\tilde\eta^-_{n+1}$ and
\begin{multline}
  \bar\rho_{n+1} = \tilde\rho^+_{n+1}
   - \partial_1 w(y^-_n, \tau_{n}) \tilde\rho^-_{n+1} - \1_{\gamma_{n+1} > 0}\frac{\epsilon_{n+1}  + \epsilon^{\gamma}_{n+1}}{\gamma_{n+1}}\\
   + \left( \partial_1 w(y^-_n, \tau_{n})
   - \partial_1 w(y^-_n, \tau_{n+1}) \right) h^-(y_n,\tau_n)\,.
   \label{eq:rho_bar}
\end{multline}
% $\bar\eta_{n+1} = \tilde\eta^+_{n+1} - \partial_1 w(y^-_n, \tau_{n})
% \tilde\eta^-_{n+1}$, and $\bar\rho_{n+1} = \tilde\rho^+_{n+1}
%  - \partial_1 w(y^-_n, \tau_{n}) \tilde\rho^-_{n+1} - \1_{\gamma_{n+1} > 0}\frac{\epsilon_{n+1}  + \epsilon^{\gamma}_{n+1}}{\gamma_{n+1}}$.

To finish the proof, it remains to check that the noise sequence satisfies the assumptions of Prop.~\ref{prop:duflo} on the event~$A_p = E_p \cap [y_n \rightarrow 0 ]$. In the remainder, $C'$ will indicate some positive constant which can change from an inequality to another one.

First, we verify that $\sum_n \|\bar\rho_n \|^2 < \infty$ on $A_p$ by controlling each one of the terms of $\bar\rho_n$. Combining the boundedness of $\partial_1 w(\cdot,\cdot)$ with the summability assumption $\sum_{n} \|\tilde\rho_{n+1}\|^2 \1_{z_n \in \mathcal W} < +\infty$ a.s., we immediately obtain on $A_p$ that $\sum_n \|\tilde\rho^+_{n+1} - \partial_1 w(y^-_n, \tau_{n}) \tilde\rho^-_{n+1}\|^2 < +\infty$ given our choice of $\sigma$. Moreover, it holds that $\left(\|\epsilon^{\gamma}_{n+1}\|/\gamma_{n+1} \right)^2 \leq C' \gamma^2_{n+1}$ by invoking Prop.~\ref{prop:inv-man-local}. In addition, using the boundedness of $\partial_1^2 w(\cdot,\cdot)$, we can write
 \begin{multline*}
   \1_{\gamma_{n+1} > 0}\left\|\frac{\epsilon_{n+1}}{\gamma_{n+1}}\right\|^2
   \leq \1_{\gamma_{n+1} > 0} \frac{C'}{\gamma_{n+1}^2} \norm{y_{n+1} - y_n}^4\\
   \leq C'\gamma_{n+1}^2 ( \norm{h(y_n, \tau_n)}^4 + \norm{\tilde\eta_{n+1}}^4 + \norm{\tilde\rho_{n+1}}^4)\, .
 \end{multline*}
 A coupling argument (see \cite[p.~401]{bra-duf-96}) shows that we can simplify the condition\\ $\limsup \EE [\| \eta_{n+1} \|^{4} \, | \, \mcF_n ]
     \1_{z_n \in \mathcal W}< \infty$ to $ \EE [\| \eta_{n+1} \|^{4} \, | \, \mcF_n ]
         \1_{z_n \in \mathcal W}< C'$. The latter condition implies that
$\EE[ \1_{A_p}\sum_n \gamma_{n+1}^2 \norm{\eta_{n+1}}^4] \leq \sum_n C' \gamma_{n+1}^2$, and therefore $\sum_n \gamma_{n+1}^2 \norm{\eta_{n+1}}^4 \1_{A_p} < +\infty$ a.s. As a consequence, noticing also the boundedness of~$(h(y_n, \tau_n))$ and $(\tilde{\rho}_n)$ on $A_p$, we deduce that $\sum_n \1_{\gamma_{n+1} > 0}\left\|\frac{\epsilon_{n+1}}{\gamma_{n+1}}\right\|^2 < +\infty$ on $A_p$. We now briefly control the last term of $\bar\rho_n$.
By the mean value inequality, we obtain that
\begin{multline*}
\left\| \left(\partial_1 w(y^-_n, \tau_{n})
- \partial_1 w(y^-_n, \tau_{n+1}) \right) h^-(y_n,\tau_n)\right\|\\
\leq \gamma_{n+1}\underset{(y^-,t)}{\sup} \norm{\partial_2\partial_1 w(y^{-},t)} \|h^-(y_n,\tau_n)\|
\leq C' \gamma_{n+1} \,,
\end{multline*}
where the last inequality stems from Prop.~\ref{prop:inv-man-local}-Eq.~\eqref{bound-d12w-loc} together with the boundedness of the sequence $(h(y_n, \tau_n))$.
In view of Eq.~\eqref{eq:rho_bar} and the above estimates, we deduce that $\sum_n \|\bar\rho_{n+1}\|^2 \1_{A_p} < +\infty$ a.s. on $A_p$.

 We verify the remaining conditions on the noise sequence $(\bar \eta_n)$. We can easily remark that $\EE[\bar\eta_{n+1} | \mcF_{n}] = 0$ and $\norm{ \bar \eta_{n+1}} \leq C' \norm{\eta_{n+1}}$ on $A_p$. Hence, $\limsup \EE [\| \bar \eta_{n+1} \|^{4} \, | \, \mcF_n ]  \1_{z_n \in \mathcal W}< \infty$.
 The last condition, meaning that the noise is exciting enough, stems from noting that
\begin{align*}
    2  \liminf \EE[\norm{ \bar \eta_{n+1}}^2 |\mcF_n] \1_{A_p}
    &\geq \liminf \EE[\norm{\tilde{\eta}^{+}_{n+1} }^2 |\mcF_n] \1_{A_p}\\
    &\phantom{=}
    - 2 \limsup \EE[\norm{\partial_1 w(y_n^-, \tau_n)\tilde{\eta}_{n+1}^{-}}^2 | \mcF_n] \1_{A_p}\\
    &> \frac{c^2}{2}\, ,
\end{align*}
where we used our choice of $\sigma, t_1$ and Eq.~\eqref{eq:noise_repulsiv_loc}.
%that $\limsup \EE [\| \tilde{\eta}_{n+1}^- \|^{2} \, | \, \mcF_n ]  \1_{z_n \in \mathcal W}< \infty$ together with the fact that $\lim_{y \rightarrow 0} \sup_{t \in \II} \partial_1 w(y, t)= 0$ by definition of the invariant manifold.

Noticing that $[y_n \to 0 ] \subset [ \Delta(y_n, \tau_n) \to 0 ]$, we can now apply Prop.~\ref{prop:duflo} to the sequence~$(u^+_n)$ (see Eq.~\eqref{eq:un+}) with $A = A_p$ to obtain
\[
\PP \left( A_p \cap [u^+_n \to 0] \right)
= \PP \left( A_p \cap [u^+_n \to 0] \cap [ \Delta(y_n, \tau_n) \to 0 ] \right)
= 0\,.
%= \PP \left( E_p \cap [y_n \to 0] \cap [u^+_n \to 0] \right).
\]
We now show that $[y_n \to 0 ] \subset [u^+_n \to 0]$\,, which amounts to prove that $w(y^-_n,\tau_n) \to 0$ given~$y_n \to 0$. To that end, upon noting that $w(0,\cdot) \equiv 0$ and that $\partial_1 w(\cdot,\cdot)$ is bounded, it suffices to apply the mean value inequality, writing :
\[
\|w(y^-_n, \tau_n)\|
= \| w(y^-_n, \tau_n) - w(0,\tau_n)\|
\leq \sup_{(y^-,t)} \| \partial_1 w(y^-, t)\| \ \|y^-_n\|
\leq K \|y^-_n\|\,.
\]
We have shown so far that $\PP (A_p) = 0$.
Since $y_n = Q^{-1} z_n$ and $[y_n\to 0] \subset \bigcup_{p\in\NN} E_p$, we finally obtain that
\[
\PP[z_n \to 0] = \PP[y_n\to 0] =
 \PP \left( \bigcup_{p\in\NN} ( [y_n\to 0] \cap E_p ) \right) = \PP \left( \bigcup_{p \in \NN} A_p \right)= 0 .
\]
Th.~\ref{th:av-traps} is proven.

\subsection{Proofs for Section~\ref{subsec-apt-algo-gal}}
\label{prf:avt-algo-gal}

\subsubsection{Proof of Lem.~\ref{lem:lm-lin-g}}
\label{prf:lm-lin-g}
The matrix $D$ coincides with $\nabla g_\infty(z_\star)$, where the function
$g_\infty$ is defined in~\eqref{odeinfty}. As such, its expression is
immediate.
  Recalling that $p_\infty S(x_\star) - q_\infty v_\star = 0$, we get
  \begin{align*}
  g(z,t) - D(z-z_\star) &=
    \begin{bmatrix} \spp(t) S(x) - \sq(t) v
                 - p_\infty \nabla S(x_\star) (x - x_\star)
    + q_\infty( v - v_\star) \\
   \sh(t) \nabla F(x) - \sr(t) m - h_\infty \nabla^2F(x_\star) (x-x_\star) +
    r_\infty m \\
   -m \left( (v+\varepsilon)^{-\frac 12} - (v_\star+\varepsilon)^{-\frac 12} \right)
    \end{bmatrix}  \\
  &=
   \begin{bmatrix}
    - \sq(t) + q_\infty & 0 & (\spp(t) - p_\infty) \nabla S(x_\star) \\
      0    & r_\infty - \sr(t) &  (\sh(t) - h_\infty) \nabla^2 F(x_\star) \\
    \frac{m}{2(v_\star+\varepsilon)^{\frac 32}}  & 0 & 0
   \end{bmatrix}
   \begin{bmatrix} v - v_\star \\ m \\ x - x_\star \end{bmatrix} \\
   &\phantom{=}
    + \begin{bmatrix}
    \spp(t) ( S(x) - S(x_\star) - \nabla S(x_\star) (x - x_\star) ) \\
    \sh(t) ( \nabla F(x) - \nabla^2 F(x_\star) (x - x_\star) ) \\
   -m \odot \left( \frac{1}{\sqrt{v+\varepsilon}}
        - \frac{1}{\sqrt{v_\star+\varepsilon}}
      + \frac{v-v_\star}{2(v_\star+\varepsilon)^{\frac 32}} \right)
   \end{bmatrix} +
   \begin{bmatrix} \spp(t)S(x_\star) - \sq(t) v_{\star} \\
   0 \\
   0
   \end{bmatrix}\\
   &\eqdef e(z,t) +  c(t).
  \end{align*}
  Under the assumptions made, it is easy to see that the function $e(z,t)$ has
  the properties required in the statement of Th.~\ref{th:av-traps}.

\subsubsection{Proof of Th.~\ref{th:avt-application}}

Consider the matrix $D$ defined in the statement of Lem.~\ref{lem:lm-lin-g}.  A
spectral analysis of this matrix as regards its eigenvalues with positive real
parts is done in the following lemma.
\begin{lemma}
\label{lm-D}
Let $D$ be the matrix provided in the statement of Lem.~\ref{lem:lm-lin-g}.
Each eigenvalue $\zeta$ of the matrix $D$ such that $\Re\zeta > 0$ is real,
and its algebraic and geometric multiplicities
are equal. Moreover, there is a one-to-one correspondence $\varphi$
between these eigenvalues and the negative eigenvalues of $V^{\frac 12} \nabla^2
F(x_\star) V^{\frac 12}$.
Let $d^+$ be the dimension of the eigenspace of $V^{\frac 12} \nabla^2 F(x_\star)
V^{\frac 12}$ that is associated with its negative eigenvalues, let
\[
W = \begin{bmatrix} & w_1 & \\ & \vdots \\ & w_{d^+} & \end{bmatrix}
 \in \RR^{d^+ \times d}
\]
be a matrix which rows are independent eigenvectors of
$V^{\frac 12} \nabla^2 F(x_\star) V^{\frac 12}$ that generate this eigenspace, and
denote as $\beta_k < 0$ the eigenvalue associated with~$w_k$. Then, the rows
of the rank $d^+$-matrix
\[
A^+ = \begin{bmatrix}
 0_{d^+\times d}, & W V^{\frac 12}, &
  - \diag(r_\infty + \varphi^{-1}(\beta_k)) W V^{-\frac 12}
 \end{bmatrix}  \in \RR^{d^+ \times 3d}
\]
generate the left eigenspace of $D$, the row $k$ being an eigenvector for
the eigenvalue $\varphi^{-1}(\beta_k)$.
\end{lemma}

\begin{proof}
It is obvious that the block lower-triangular matrix $D$ has $d$ eigenvalues
equal to $- q_\infty$ and $2d$ eigenvalues which are those of the sub-matrix
 \[
 \widetilde D =
  \begin{bmatrix} - r_\infty I_d  & h_\infty \nabla^2 F(x_\star) \\
   - V & 0 \end{bmatrix} .
 \]
Given $\lambda \in \CC$, we obtain by standard manipulations involving
determinants that
\[
 \det(\widetilde D - \lambda) =
   \det(\lambda(r_\infty+\lambda) + h_\infty V \nabla^2 F(x_\star)) =
   \det(\lambda(r_\infty+\lambda)
             + h_\infty V^{\frac 12} \nabla^2 F(x_\star) V^{\frac 12}) .
\]
Denoting as $\{\beta_k\}_{k=1}^d$ the eigenvalues of $h_\infty V^{\frac 12}
\nabla^2 F(x_\star) V^{\frac 12}$ counting the multiplicities, we obtain from the
last equation that the eigenvalues of $\widetilde D$ are the solutions of the
second order equations
\[
  \lambda^2 + r_\infty\lambda + \beta_k = 0 , \quad k = 1,\ldots, d.
\]
The product of the roots of such an equation is $\beta_k$, and their sum is
$-r_\infty \leq 0$. Thus, denoting as $\zeta_{k,1}$ and $\zeta_{k,2}$ these roots,
it is easy to see that if $\beta_k \geq 0$, then $\Re\zeta_{k,1},
\Re\zeta_{k,2} \leq 0$, while if $\beta_k < 0$, then both $\zeta_{k,i}$ are
real, and only one of them is positive.  Thus, we have so far shown that the
eigenvalues of $D$ which real parts are positive are themselves real, and
there is a one-to-one map $\varphi$ from the set of positive eigenvalues of
$D$ to the set of negative eigenvalues of
$V^{\frac 12} \nabla^2 F(x_\star) V^{\frac 12}$.
Moreover, the algebraic multiplicity of the eigenvalue $\zeta > 0$ of $D$
is equal to the multiplicity of $\varphi(\zeta)$.

Let us now turn to the left (row) eigenvectors of $D$ that correspond to these
eigenvalues. To that end, we shall solve the equation
\begin{equation}
\label{leftvep}
u D = \zeta u \quad \text{with} \ u = [ 0, u_{1}, u_{2} ],
 \quad u_{1,2} \in \RR^{1\times d},
\end{equation}
for a given eigenvalue $\zeta > 0$ of $D$. Developing this equation, we get
\[
 - r_\infty u_{1} - u_{2} V = \zeta u_{1}, \quad
  h_\infty u_{1} \nabla^2 F(x_\star) = \zeta u_{2} .
\]
If we now write $\tilde u_{1} = u_{1} V^{-\frac 12}$ and
$\tilde u_{2} = u_{2} V^{\frac 12}$, this system becomes
\[
- r_\infty \tilde u_{1} - \tilde u_{2} = \zeta \tilde u_{1}, \quad
 h_\infty \tilde u_{1} V^{\frac 12} \nabla^2 F(x_\star) V^{\frac 12}
  = \zeta \tilde u_{2} ,
\]
or, equivalently,
\[
% \label{u1u2left}
\tilde u_{2} = - (r_\infty + \zeta) \tilde u_{1}, \quad
 \tilde u_{1} \left( \zeta^2 + r_\infty\zeta +
  h_\infty V^{\frac 12} \nabla^2 F(x_\star) V^{\frac 12}  \right) = 0 ,
\]
which shows that $\tilde u_{1}$ is a left eigenvector of
$V^{\frac 12} \nabla^2 F(x_\star) V^{\frac 12}$ associated with the eigenvalue
$\varphi(\zeta)$. What's more, assume that $r$ is the multiplicity of
$\varphi(\zeta)$, and, without generality loss, that the submatrix
$W_{r,\cdot}$ made of the first $r$ rows of $W$ generates the left eigenspace
of $\varphi(\zeta)$. Then, the matrix
\[
\begin{bmatrix}
 0_{r\times d} & W_{r\cdot} V^{\frac 12} & - (r_\infty + \zeta) W_{r\cdot} V^{-\frac 12}
 \end{bmatrix}
\]
is a $r$-rank matrix which rows are independent left eigenvectors that
generate the left eigenspace of $D$ for the eigenvalue $\zeta$. In particular,
the algebraic and geometric multiplicities of this eigenvalue are equal.
The same argument can be applied to the other positive eigenvalues of~$D$.
\end{proof}

We now have all the elements to prove Th.~\ref{th:avt-application}.
Recall Eq.~\eqref{eq:algo-gal-avt}:
\[
z_{n+1} = z_n + \gamma_{n+1} b(z_n, \tau_n) +  \gamma_{n+1} \eta_{n+1}
 + \gamma_{n+1} \rho_{n+1} ,
\]
where $b(z,t) = g(z,t) - c(t) = D (z-z_\star) + e(z,t)$ and
$\rho_n = c(\tau_{n-1}) + \tilde{\rho}_n$.  With these same notations,
we check that
%one can check that
Assumptions~\ref{e=0}--\ref{traps-noise-mom} in the statement of Th.~\ref{th:av-traps} are satisfied.
The function~$e(z,t)$ satisfies Assumptions~\ref{e=0}--\ref{dt_eps} by  Lem.~\ref{lem:lm-lin-g}. We now verify that the sequence $(\rho_n)$ fulfills Assumption~\ref{sum_rho_carre}. First, observe that $\sum_n \|c(\tau_n)\|^2<\infty$ under Assumption~\ref{hyp:avt}-\ref{rho_sum_carre_appli}. Then, we control the second term~$(\tilde \rho_n)$. After straightforward derivations, one can show the existence of a positive constant $C$ (depending only on $\varepsilon$ and a neighborhood $\mathcal W$ of $z_\star$) such that
\begin{equation}
  \label{eq:rho_bound}
\|\tilde \rho_{n+1}\|^2  \1_{z_n \in \mathcal W}
%\leq C( \|m_n - m_{n+1}\|^2 + \|m_n\|^2 \|v_{n+1}-v_n\|^2)\1_{z_n \in \mathcal W}
 \leq C( \|m_n - m_{n+1}\|^2 +\|v_{n+1}-v_n\|^2)\1_{z_n \in \mathcal W}\,.
\end{equation}
Using the boundedness of the sequences $(h_n)$ and $(r_n)$ together with the update rule of $m_n$ and Assumption~\ref{hyp:avt}-\ref{moment_appli_avt}, there exists a positive constant $C'$ independent of $n$ (which may change from an inequality to another) such that
\begin{equation}
  \label{eq:rho_subbound}
 \EE\left[\|m_n - m_{n+1}\|^2 \1_{z_n \in \mathcal W} \right]
 \leq \gamma_{n+1}^2 C' \EE \left[(1 + \EE_\xi\left[ \|\nabla f(x_n,\xi)\|^2 \right]) \1_{z_n \in \mathcal W} \right]
 \leq C' \gamma_{n+1}^2\,.
\end{equation}
A similar result holds for $\EE\left[\|v_n - v_{n+1}\|^2 \1_{z_n \in \mathcal W}\right]$ following the same arguments. In view of Eqs.~\eqref{eq:rho_bound}-\eqref{eq:rho_subbound} and the assumption $\sum_n \gamma_{n+1}^2 < +\infty$, it holds that $\EE\left[\sum_n \|\tilde \rho_{n+1}\|^2 \1_{z_n \in \mathcal W}\right] < +\infty$. Therefore, $\sum_n \|\tilde \rho_{n+1}\|^2 \1_{z_n \in \mathcal W} < +\infty$ a.s., which completes our verification of condition~\ref{sum_rho_carre} of Th.~\ref{th:av-traps}.
Assumption~\ref{traps-noise-mom} follows from condition~\ref{hyp:avt}-\ref{moment_appli_avt}.
Finally, let us make Assumption~\ref{traps-noise} of %this theorem
Th.~\ref{th:av-traps} more explicit.
Partitioning the matrix $Q^{-1}$ as
$Q^{-1} = \begin{bmatrix} B^- \\ B^+ \end{bmatrix}$ where $B^\pm$ has $d^\pm$
rows, Lem.~\ref{lm-D} shows that the row spaces of $B^+$ and $A^+$ are the
same, which implies that Assumption~\ref{traps-noise} can be rewritten
equivalently as $\EE[ \norm{A^+ \eta_{n+1}}^2 \, | \, \mcF_n ]
\1_{z_n \in \mathcal W} \geq c^2 \1_{z_n \in \mathcal W}$.
By inspecting the form of $\eta_{n}$ provided by Eq.~\eqref{eq:noise_eta} (written as a column vector),
one can readily check that Assumption~\ref{hyp:avt}-\ref{cond-noise} implies Assumption~\ref{traps-noise} of
Th.~\ref{th:av-traps} for a small enough neighborhood~$\cW$, using the
continuity of the covariance matrix
$V^{\frac 12} \EE_\xi
  (\nabla f(x, \xi) - \nabla F(x) )
  (\nabla f(x, \xi) - \nabla F(x) )^\T V^{\frac 12}$
when $x$ is near $x_\star$.

\subsection{Proof of Th.~\ref{th:avt-application-nesterov}}
As mentioned in Section~\ref{subsubsec:aplit-avt-nest}, the proof of Th.~\ref{th:avt-application-nesterov} is almost identical to the one of Th.~\ref{th:avt-application}. We point out the main differences here.
In Lem.~\ref{lem:lm-lin-g}, replace~$D$ by $\tilde D = \begin{bmatrix}
 0 & h_\infty \nabla^2 F(x_\star) \\
 - I_d & 0
\end{bmatrix}$ and set $c(t) = 0$. Then, in Lem.~\ref{lm-D}, replace the matrix $V^{1/2}\nabla^2F(x_{\star})V^{1/2}$ by the Hessian $\nabla^2 F(x_{\star})$.

\appendix

% \section{Appendix section}\label{app}

\section*{Acknowledgements}
We would like to thank the anonymous reviewers for their outstanding job of refereeing, especially for their comments on the avoidance of traps results and the stability of the iterates of the stochastic algorithm which helped us to improve and clarify our manuscript.\\
A.B. was supported by the ``Futur \& Ruptures'' research program which is jointly funded by the IMT, the Mines-T\'el\'ecom Foundation and the Carnot TSN Institute. S.S. was supported by the ``R\'egion Ile-de-France''.

\bibliographystyle{plain}
\bibliography{math}

\end{document}